\documentclass[12pt,reqno]{amsart}
\usepackage{amssymb,latexsym,amsmath}
\usepackage[usenames]{color}

\newcommand{\red}{\textcolor[rgb]{1.00,0.00,0.00}}

\theoremstyle{plain}
\newtheorem{theorem}{Theorem}[section]
\newtheorem{lemma}[theorem]{Lemma}
\newtheorem{proposition}[theorem]{Proposition}
\newtheorem{corollary}[theorem]{Corollary}
\newtheorem{claim}{Claim}
\newtheorem{fact}[theorem]{Proposition}

\theoremstyle{definition}
\newtheorem{definition}[theorem]{Definition}

\newtheorem{remark}[theorem]{Remark}

\newtheorem{example}[theorem]{Example}

\newcommand{\Fact}{Proposition }
\newcommand{\Facts}{Propositions }

\newenvironment{scproof}
{\begin{proof}[Proof of Claim.]}
{\end{proof}}

\DeclareMathOperator{\Mod}{Mod}
\DeclareMathOperator{\Th}{Th}
\DeclareMathOperator{\im}{{\sf im}}
\DeclareMathOperator{\kerr}{{\sf ker}}

\DeclareMathOperator{\End}{{\sf End}}
\DeclareMathOperator{\LEnd}{{\sf End}^\ast}
\DeclareMathOperator{\id}{{\sf id}}

\DeclareMathOperator{\rad}{{\sf rad}}
\DeclareMathOperator{\Soc}{{\sf Soc}}

\newcommand{\sk}[2]{\langle{#1}\mid{#2}\rangle}
\newcommand{\Sk}[2]{\bigl\langle{#1}\mid{#2}\bigr\rangle}

\newcommand{\mc}[1]{\mathcal{#1}}
\newcommand{\mb}[1]{\mathbb{#1}}

\newcommand{\ma}[1]{^{{#1} \times {#1}}}
\newcommand{\vep}{\varepsilon}
\renewcommand{\phi}{\varphi}

\newcommand{\lab}[1]{\label{#1}}

\newcommand{\xc}[1]{}

\begin{document}

\title[Representations of regular $\ast$-rings and CMILs]{Linear representations of regular rings and complemented modular lattices with involution}

\author[Christian Herrmann]{Christian Herrmann}
\address[Christian Herrmann]{Technische Universit\"{a}t Darmstadt FB4\\Schlo{\ss}gartenstr. 7\\64289 Darmstadt\\Germany}
\email{herrmann@mathematik.tu-darmstadt.de}

\author[Marina Semenova]{Marina Semenova}
\address[Marina Semenova]{Sobolev Institute of Mathematics\\Siberian
Branch RAS\\Acad. Koptyug prosp. 4\\630090 Novosibirsk\\Russia}
\address[]{Novosibirsk State University\\Pirogova str. 2\\630090 Novosibirsk\\Russia}
\email{udav17@gmail.com; semenova@math.nsc.ru}


\begin{abstract}
Faithful representations of regular $\ast$-rings and modular
complemented lattices with involution within orthosymmetric
sesquilinear spaces are studied within the framework of Universal Algebra.
In particular, the correspondence between classes of spaces and classes of
representable structures is analyzed; for a class $\mathcal{S}$ of spaces which is closed under
ultraproducts and non-degenerate finite-dimensional subspaces, the class of
representable structures is shown to be closed under complemented [regular] subalgebras, homomorphic images, and
ultraproducts. Moreover, this class is generated by its members which are isomorphic to subspace lattices with involution [endomorphism $\ast$-rings, respectively] of finite-dimensional spaces from $\mathcal{S}$.
Under natural restrictions, this result is refined to a $1$-$1$-correspondence between the two types of classes.
\end{abstract}

\maketitle

\section{Introduction}

For $\ast$-rings, there is a natural and well established
concept of representation in a vector space $V_F$
endowed with an orthosymmetric sesquilinear form:
a homomorphism $\varepsilon$ into the endomorphism ring of $V_F$
such that $\varepsilon(r^\ast)$ is the adjoint of $\varepsilon(r)$.
Famous examples of [faithful] representations
are due to Gel'fand-Naimark-Segal ($C^\ast$-algebras in Hilbert space) and
Kaplansky (primitive $\ast$-rings with a minimal right ideal), cf. \cite[Theorem 4.6.6]{beidar}.

[Faithful] representability of $\ast$-regular rings within anisotropic inner product spa\-ces
has been studied by Micol \cite{Flo} and used to derive results in the universal algebraic
theory of these structures.
For the $\ast$-regular rings of classical quotients
of finite Rickart $C^\ast$-algebras (cf. Ara and Menal \cite{ara}),
existence of representations has been established in \cite{PartI}.
For complemented modular lattices  with
involution   $a\mapsto a^\prime$ (CMILs for short), an analogue of the concept of representation
is a lattice homomorphism $\varepsilon$, preserving the bounds $0$ and $1$,
into the lattice of all subspaces such that
$\varepsilon(a^\prime)$ is the orthogonal subspace to $\varepsilon(a)$ (cf. Niemann \cite{Nik}).
The latter has been considered in the context of synthetic orthogeometries in \cite{he4},
continuing earlier work on anisotropic geometries and
modular ortholattices \cite{proat,hr,hr2}.
Primary examples are atomic CMILs associated with irreducible desarguean orthogeometries and those CMILs
which arise from lattices of principal right ideals of representable regular $\ast$-rings. 

The [proofs of the] main results of these studies
relate closure properties of a class $\mathcal{S}$ of spaces
with closure properties of the class $\mathcal{R}$ of algebraic structures
[faithfully] representable within spaces from $\mathcal{S}$.
In particular, for a class $\mathcal{S}$ closed under ultraproducts and
non-degenerate finite-dimensional subspaces,
one has $\mathcal{R}$ closed under ultraproducts, homomorphic
images, and regular [complemented, respectively] subalgebras.
Moreover, with an approach due to Tyukavkin \cite{tyu},
it has been shown that $\mathcal{R}$ is generated, with respect to these operators,
by the endomorphism $\ast$-rings [by the subspace lattices
with involution $U\mapsto U^\perp$, respectively] of
finite-dimensional spaces from $\mathcal{S}$ (cf. Theorem \ref{flo}).
Conversely, any class $\mathcal{R}$ of structures generated in this way
has its members representable within $\mathcal{S}$.
 
The first purpose of the present paper is to extend these results to
regular $\ast$-rings
 on one hand, to representations within
orthosymmetric sesquilinear spaces on the other 
-- thus  allowing  
 regular rings with an involution
which may have $r^\ast r=0$ for some $r\neq 0$, that is regular
$\ast$-rings which are not $\ast$-regular;
for example, $\ast$-rings associated with   finite dimensional spaces
having some isotropic points. 
The second one is to give a more transparent presentation
by dealing with types of classes naturally associated
with representations in linear spaces.
We call a class of structures $\mathcal{R}$ as above
an $\exists$-\emph{semivariety} of regular $\ast$-rings
[CMILs]
and we call $\mathcal{S}$ a \emph{semivariety} of spaces.
The quantifier `$\exists$' refers to the required existence of quasi-inverses [complements, respectively].
In this setting, the above-mentioned relationship between classes of spaces $\mathcal{S}$ and classes of
representable structures $\mathcal{R}$ can be refined to a $1$-$1$-correspondence (cf. Theorem \ref{11}).
Also, we observe that $\mathcal{R}$ remains
unchanged if $\mathcal{S}$ is enlarged by forming
two-sorted substructures, corresponding to the subgeometries in the sense of \cite{he4},
(cf. Theorem \ref{flo}).
We also provide a useful condition on $\mathcal{S}$ which implies that
$\mathcal{R}$ is an $\exists$-\emph{variety}, i.e. that  $\mathcal{R}$
is also closed under direct products
(see Proposition \ref{varsp}).
For a reference in later applications,
e.g. to decidability results refining those of \cite{heu},
we consider $\ast$-rings which are also
algebras over a fixed commutative $\ast$-ring.

In the context of synthetic orthogeometries,
the class $\mathcal{R}$ of representable structures is an $\exists$-variety
if $\mathcal{S}$ is also closed under orthogonal disjoint unions.
No such natural construction is available for sesquilinear spaces.
The alternative, chosen by Micol \cite{Flo},
was to generalize the concept of faithful representation
to residually faithful representation;
thus, associating with any semivariety of spaces
an $\exists$-variety of generalized representables.
We derive these results in our more general setting (cf. Proposition \ref{flonik}).

We first present background on sesquilinear  spaces
(Section 2), rings (Sections 3--4), and lattices (Sections 5,7).
Synthetic orthogeometries are included (section 6) for use of the
results in \cite{he4}.  A key to results
on representations is to view them as multi-sorted structures
(Section 8). The Universal Algebra  point of view
and the class operators are introduced in Section 9. 
The basic reduction to finite dimensions
is in Section 10,  applications to  correspondences between
classes in Section 11. Section 12 relates these  
to Micol's \cite{Flo} the more general concept of representation. 
In Sections 8--12 results on rings and on lattices
are presented in parallel. Proofs of the former do not depend
on the latter. Though, the other way round,  
we have to use basic results on lattices of principal right ideals
of regular rings.

Thanks are due to the referee for a lot of corrections and
 improvements which, as  we hope, make the paper more widely
accessible.

\section{$\varepsilon$-Hermitian spaces}

We first define the linear  structures providing
representations both for lattices and rings with involution.
For any division ring $F$, endowed with an anti-automorphism
$\nu$ (we write $\nu(\lambda)=\lambda^\nu$) 
 we consider \emph{sesquilinear spaces} which are
[right] vector spaces $V_F$ endowed with a \emph{scalar product} or a \emph{sesquilinear form}
$\sk{\ }{\ }\colon V\times V\to F$; that is,
for all $u$, $v$, $w\in V$ and all $\lambda$, $\mu\in F$, one has
\[
\sk{u}{v+w}=\sk{u}{v}+\sk{u}{w},\quad 
\sk{u\lambda}{v\mu}=\lambda^\nu\sk{u}{v}\mu,
\]

Our basic reference is \cite[Chapter I]{gro} (though, we use
``sesquilinear space'' in a more general meaning).
Observe that,
from a right vector space $V_F$ one  obtains a left vector
space $_FV$ putting  $\lambda v=v \lambda^{\nu^{-1}}$.
By this, a sesquilinear form on $V_F$ 
as defined above turns out a sesquilinear form
in the sense of \cite{gro} on $_FV$ -- both with respect to $\nu$.
This gives access to results of \cite{gro} in the
left vector space setting.
Introductions to orthogonal geometry in infinite dimension
are also given in
 \cite{mackey},  \cite[Chapter 14]{FF},
cf. \cite[Chapter IV]{jac}, \cite[\S 1.21]{herstein},
\cite[\S 4.6]{beidar}.

Given  a second sesquilinear space $V'_{F'}$
with $\nu'$ and $\sk{\,}{\,}'$,
we have the following concepts relating it with the first:
An \emph{isomorphism}  between the sesquilinear spaces
 is a bijection $\omega:V\to V'$ 
which is an $\alpha$-semilinear map $V_F\to V'_{F'}$  for some isomorphism
$\alpha:F\to F'$ such that $\alpha \circ \nu =\nu'\circ \alpha$
and  $\sk{\omega(v)}{\omega(w)}'=\alpha(\sk{v}{w})$
for all $v,w \in V$.
The second space arises from the first 
 by \emph{scaling} with $\mu\in F$  if $F^\prime=F$ and
$V_F=V'_{F'}$ as vector spaces, and
if  $\mu \neq 0$,  $r^{\nu'}=\mu r^\nu\mu^{-1}$, 
and $\sk{u}{v}'= \mu\sk{u}{v}$.
Finally, $V_F$ and $V'_{F'}$ are \emph{similar},
if one arises from the other by any composition of isomorphisms and scalings.
It is easy to see that 
any similitude can be expressed as an isomorphism followed by a scaling.

Any vector space $V$ over a division ring $F$ with
anti-automorphism $\nu$  can be turned into a non-degenerate sesquilinear
space: given a basis $v_i\, (i \in I)$ and $0\neq \delta_i \in F$
define $\sk{\sum_{i\in J} v_i \lambda_i}{\sum_{i \in J} v_i \mu_i}
= \sum_{i\in J} \lambda_i^\nu\delta_i\mu_i$ for finite
$J \subseteq I$ and $\lambda_i,\mu_i \in F$.
Though, these examples are far from being exhaustive.

Since we consider only one anti-automorphism
$\nu$ on $F$ and only one
 scalar product on $V_F$ at a time,
we use $F$ to include $\nu$ (and write $\lambda^\nu= \lambda^*$)
and  $V_F$ to denote the space endowed with the scalar product.

A sesquilinear space $V_F\neq 0$ is \emph{non-degenerate} if
$\sk{u}{v}=0$ for all $v\in V$ implies $u=0$.
For $\varepsilon\in F$, 
$V_F$ is $\varepsilon$-\emph{hermitian} if $\sk{v}{u}=\varepsilon\cdot\sk{u}{v}^\ast$ for
all $u$, $v\in V$; $V_F$ is \emph{hermitian} if it is $1$-hermitian;
$V_F$ is \emph{skew symmetric} if it is $(-1)$-hermitian
and $\lambda^\ast=\lambda$ for all $\lambda\in F$;
$V_F$ is \emph{alternate}, if $\sk{v}{v}=0$ 
for all $v\in V$ (observe that  \cite[\S 4.6]{beidar} requires characteristic $\neq 2$).
$V_F$ is \emph{anisotropic} if $\sk{v}{v}\neq 0$ for all $v\in V$,
$v\ne 0$.

For endomorphisms $\varphi$, $\psi$ of
the vector space  $V_F$ we say that $\psi$ is an \emph{adjoint} of $\varphi$
if $\sk{\varphi(u)}{v}=\sk{u}{\psi(v)}$ for all $u$, $v\in V$.
If  $V_F$ is non-degenerate, then any endomorphism  $\varphi$ has
at most one adjoint $\psi$;
if such $\psi$ exists, we write $\psi=\varphi^\ast$.
If $\varphi^\ast$ and $\chi^\ast$ exist, then
$(\chi \circ \phi)^\ast= \phi^\ast \circ \chi^\ast$.
The space $V_F$ is \emph{orthosymmetric}, or \emph{reflexive},
if $\perp$ is a symmetric relation.

\begin{fact}\lab{sym}\lab{symm}
The relations of orthogonality and adjointness are left unchanged under scaling;
in particular, orthosymmetry is preserved under scaling.
Consider a non-degenerate sesquilinear  space $V_F$.
The following are equivalent if
 $\dim V_F >1$:
\begin{enumerate}
\item
the sesquilinear space
$V_F$ is orthosymmetric;
\item
the sesquilinear space
$V_F$ is $\varepsilon$-hermitian for some $($unique$)$ $\varepsilon\in F\setminus\{0\}$;
\item
up to scaling, $V_F$ is either hermitian or skew-symmetric;
\item the adjointness relation
is symmetric on $\End(V_F)$;
\item if  $\varphi^\ast$
exists then  $\varphi^{\ast\ast}=\varphi$.
\end{enumerate}
Furthermore, if $V_F$ is $\vep$-hermitian and non-degenerate,
 then   $\lambda\mapsto\lambda^\ast$ is an \emph{involution} on $F$,
that is $(\lambda^\ast)^\ast=\lambda$ for all $\lambda \in F$.
If $V_F$ is alternate and non-degenerate  then it is skew symmetric and  $F$ is commutative; moreover, any $V'_{F'}$ 
similar to $V_F$ is alternate, too.
\end{fact}
\begin{proof}
The first statement is obviously true. Now, assume $V_F$
non-degenerate and $\dim V_F >1$.
 The following references are to 
\cite[Chapter I]{gro}.
(i) implies (ii) by Theorem 1 of \S1.3.
(ii) implies (iii) by (15) of \S1.5. (iii) implies (i), obviously,
proving pairwise  equivalence of (i), (ii), and (iii).

Assuming (iii),
symmetry of adjointness follows, easily. That, in turn, implies
$\varphi=\varphi^{\ast\ast}$ for every  $\varphi\in\LEnd(V_F)$.   
Thus, we have (iii) $\Rightarrow$ (iv) and (iv) $\Rightarrow$ (v).

Assuming (v),
let $\lambda \mapsto \lambda^+$ denote the inverse
of $\lambda \mapsto \lambda^\ast$.
Given $u\in V$ such that $\mu=\sk{u}{u}\neq 0$, consider two
linear  maps:
\[
\varphi_u(v)= u\bigl(\sk{v}{u}\mu^{-1}\bigr)^+\ \text{and}\ 
\psi_u(w)=u\mu^{-1}\sk{u}{w},\quad v,\ w\in V.
\]
Observe that $\psi_u=\varphi_u^\ast$, whence by our hypothesis, $\varphi_u=\psi_u^\ast$.
Moreover, $\varphi_u$ and $\psi_u$ are the projections onto $uF$ associated with
the decompositions $V=uF\oplus\{v\in V\mid\sk{v}{u}=0\}$
and $V=uF\oplus\{w\in V\mid\sk{u}{w}=0\}$, respectively.
Now, $\varphi_u=\psi_u\circ\varphi_u$ and we get
$\psi_u=\varphi_u^\ast=\varphi_u^\ast\circ\psi_u^\ast=\varphi_u$,
an orthogonal projection.

It follows that $\sk{v}{w}=0$ is equivalent to $\sk{w}{v}=0$ unless $\sk{v}{v}=\sk{w}{w}=0$.
In the latter case, let $u=v+w$.
If $\sk{u}{u}=0$ then $\sk{v}{w}=-\sk{w}{v}$;
otherwise, $\sk{v}{w}=0$ iff $\sk{v}{u}=0$ iff $\sk{u}{v}=0$ iff $\sk{w}{v}=0$.
This proves that (v) implies (i).

For non-degenerate  $\vep$-hermitian $V_F$, in order to prove that $\lambda \mapsto \lambda^*$ is an involution,
by (15) of \S1.5 we may assume  $\vep \pm 1$.  As (7)  in \S1.3
follows from (9),
we have $(\lambda^*)^*= \vep^{-1} \lambda \vep = \lambda$. 
The alternate case is dealt with in 
(12) of \S1.4.
\end{proof}

\noindent
A sesquilinear space $V_F$, over a division ring $F$ with
involution,   
 which is $\vep$-hermitian for some $\vep$ and non-degenerate will be called \emph{pre-hermitian}.
In the sequel, we consider only pre-hermitian spaces.
If $V_F$ is, in addition, anisotropic, we also speak of an \emph{inner product space}.

For vectors $u$, $v\in V$, we say that $v$ is \emph{orthogonal} to $u$ and write $u\perp v$, if $\sk{u}{v}=0$.
The \emph{orthogonal} of $X \subseteq V$ is the subspace
$X^\perp=\{v \in V\mid \forall u \in X.\;  u \perp v \}$.
A subspace $U$ is \emph{closed} if $U=U^{\perp\perp}$.
If $U$ is a subspace of $\dim U=1$ then
$U^\perp= \kerr f$
for  the linear map $f:V \to F$ 
given as $f(w)=\sk{v}{w}$ where $U=vF$;
$f$ is surjective since $V_f$ is non-degenerate,
whence  $\dim V/U^\perp=1$. Since $U^\perp= \bigcap_{v \in B} vF^\perp$
 for any basis $B$ of $U$,
 it follows $\dim V/U^\perp
\leq \dim U$ for any $U$ with $\dim U <\omega$
(actually, equality holds and $U$ is closed, see \Facts \ref{og0} and \ref{gal}).

On any linear  subspace $U$ of $V_F$,  one has
the sesquilinear \emph{subspace} $U_F$ with  the induced scalar product.
When $U_F$ is non-degenerate, $U_F$ is pre-hermitian, too.
A finite-dimensional subspace $U_F$ of $V_F$ is non-degenerate if and only if
$U\cap U^\perp=0$, if and only if $V=U\oplus U^\perp$ 
(as $\dim V{\slash}U^\perp\leq \dim U$).
We write in this case $U\in\mathbb{O}(V_F)$ and say that
$U$ is a \emph{finite-dimensional orthogonal summand};
in particular, $U$ is closed.

\begin{fact}\lab{du}
Every pre-hermitian space $V_F$ is the directed union
of the subspaces  $U_F$,   $U\in\mathbb{O}(V_F)$.
Actually, for any  finite-dimensional subspace $W\in \mathbb{L}(V_F)$ there is $U\in \mathbb{O}(V_F)$
such that $W \subseteq U$ and $\dim U \leq 2 \dim W$.
\end{fact}

\begin{proof}
This is
\cite[Chapter I, \S5 Lemma 4]{gro}, cf. \cite[Remark 4.6.14]{beidar}.
Alternatively, one can apply  \cite[Theorem 1.2]{he4}
to the ``orthogeometry''
 $\mathbb{G}(V_F)$ 
associated with $V_F$ 
and \Fact \ref{og0}, below.
\end{proof}

\noindent
For a subspace $U_F$ of $V_F$,
the linear  subspace $\rad U=U\cap U^\perp$ is the \emph{radical} of $U_F$.
With $\langle v+\rad U\mid w+\rad U\rangle:=\langle v\mid w\rangle$,
the $F$-vector space $U{\slash}\rad U$ is 
a sesquilinear space  $U_F{\slash}\rad U$ with respect to the given anti-automorphism of $F$.
We call $U_F{\slash}\rad U$ a \emph{subquotient space}.

\begin{fact}\lab{sq}
Let $V_F$ be a pre-hermitian space and let $U_F$ be a subspace of $V_F$.
Then $U_F{\slash}\rad U$ is non-degenerate; it is $\varepsilon$-hermitian if $V_F$ is.
The space $U_F{\slash}\rad U$ is isomorphic to any subspace
$W_F$ of $V_F$ such that $U=W\oplus\rad U$.
\end{fact}

\begin{proof}
The map $w\mapsto w+\rad U$ establishes an isomorphism of sesquilinear spaces
from $W_F$ onto $U_F{\slash}\rad U$.
\end{proof}

\section{Rings and algebras with involution}

When mentioning rings, we always mean associative rings $R$ possibly without unit.
 The principal right ideal $aR$ generated by $a$ equals
$\{za\mid z\in\mathbb{Z}\}\cup\{ar\mid a\in I\}$.
A $\ast$-\emph{ring} is a ring $R$ endowed with an \emph{involution};
that is, an anti-automorphism $x\mapsto x^\ast$ of order $2$, such that
\[
(r+s)^\ast=r^\ast+s^\ast,\quad
(rs)^\ast=s^\ast r^\ast,\quad
(r^\ast)^\ast=r\quad\text{for all}\ r,s\in R,
\]
cf. \cite[\S 1]{herstein}, \cite[\S 2.13]{rowen},
\cite[\S  4]{beidar}. 

\medskip
An element $e$ of a $\ast$-ring $R$ is a \emph{projection}, if $e=e^2=e^\ast$.
A $\ast$-ring $R$ is \emph{proper} if $r^\ast r=0$ implies $r=0$ for all $r\in R$.
Throughout this paper, let $\Lambda$ be a commutative $\ast$-ring with unit.
A $\ast$-$\Lambda$-\emph{algebra} $R$ is an associative (left) unital $\Lambda$-algebra, with unit $1$ considered a constant,
which is a $\ast$-ring   such that
\[
(\lambda r)^\ast=\lambda^\ast r^\ast\ \text{for all}\ r\in R,\ \lambda\in \Lambda.
\] 
For example, involutive  Banach algebras are $\ast$-$\mathbb{C}$-algebras.
Unless stated otherwise, we consider the scalars $\lambda\in \Lambda$
as unary operations $r\mapsto\lambda r$ on $R$; in other words, we consider
$\ast$-$\Lambda$-algebras as $1$-sorted algebraic structures.
The map
$\lambda \mapsto \lambda 1$ is a $\ast$-ring homomorphism from
$\Lambda$ into the center of $R$; 
in view of this, denoting both involutions on $R$ and on $\Lambda$
by the same $^\ast$  
should not cause confusion; also, 
 most arguments concerning the action of  
$\Lambda$  are obvious and left to the reader.

A $\ast$-\emph{ideal} of a $\ast$-ring or an $\ast$-$\Lambda$-algebra $R$
is an  ideal $I$ with $I=I^\ast$, where $I^\ast=\{r^\ast\mid r\in I\}$.
We call $R$ \emph{strictly subdirectly irreducible}
if the underlying  ring  is subdirectly irreducible, i.e. has a smallest non-zero ideal $I$;
in this case, $I=I^\ast$.
Similarly, $R$ is \emph{strictly simple} if $0$ and $R$ are the only ideals.
In the $\ast$-ring literature, such $\ast$-rings are called `simple',
while simple $\ast$-rings are called `$\ast$-simple'
 cf. \cite{beidar2}.

 The 
[right] \emph{socle}
$\Soc(R)$ consists of 
all  $a \in R$ such that $aR$ is the sum of finitely many
minimal right ideals;  $\Soc(R)$  is an ideal of $R$.
We say that a $\ast$-$\Lambda$-algebra is \emph{atomic}
if any non-zero right ideal contains a minimal 
 one. 

A ring $R$ is [\emph{von Neumann}] \emph{regular} if for any $a\in R$,
there is an element $x\in R$ such that $axa=a$; such an element is called a \emph{quasi-inverse} of $a$.
A $\ast$-ring $R$ is $\ast$-\emph{regular} if it is regular and proper.
The reader interested in more details is referred to any of
 \cite{Halp,ber,ber4,mae,good,sko}.

\xc{For any subset $X$ of a ring $R$, we call the set
\[
{\sf Ann}^l(X)=\{s\in R\mid sx=0\ \text{for all}\ x\in X\}
\]
the \emph{left annihilator} of $X$. The \emph{right annihilator} ${\sf Ann}^r(X)$
is defined symmetrically.}
Recall that, for  a vector space $V_F$ over a division ring $F$,  $\End(V_F)$ denote the set of all endomorphisms of $V_F$.

\begin{fact}\lab{reri}\begin{enumerate}
\item
For any vector space $V_F$, $\End(V_F)$ is a regular simple ring.
\item 
A ring $R$ is regular if it admits a regular ideal $I$ such that $R\slash I$ is regular.
Any ideal of a regular ring is  regular.
\item
A ring $R$ is regular [$\ast$-ring $R$ is $\ast$-regular] if and only if
for any $a\in R$ there is an idempotent $[$a $($unique$)$ projection, respectively$]$ $e\in R$
such that $aR=eR$.
\item
For any $a$, $b$ in a regular ring $R$,
there is an idempotent $e\in aR+bR$ such that $ea=a$ and $eb=b$.
\xc{\item
For idempotents $e$, $f$ in a regular ring $R$,
${\sf Ann}^l(eR)=\{s\in R\mid se=0\}=R(1-e)$
and $eR\subseteq fR$ if and only if ${\sf Ann}^l(fR)\subseteq{\sf Ann}^l(eR)$.}
\item
Homomorphic images and direct products of regular $(\ast$-regular$)$
$\ast$-$\Lambda$-algebras
are again regular $($ $\ast$-regular$)$
$\ast$-$\Lambda$-algebras.
\end{enumerate}
\end{fact}

\begin{proof}
Statements (i)-(v) are well known, cf. \cite[1.26]{ber4}, \cite[Lemma
1.3]{good}, \cite[Theorem 1.7]{good}. For the existence of projections
see \cite[Part II Chapter IV Theorem 4.5]{Halp} or 
\cite[Proposition 1.13]{ber4}. 
In (v), the claim for products is obvious, 
for homomorphic images it  follows by (iii).
\end{proof}

\noindent
In particular, in the $\ast$-regular case, any ideal is a $\ast$-ideal by \Fact \ref{reri}(iii);
thus subdirectly irreducibles [simples] are strictly subdirectly irreducible [strictly simple, respectively].
Call a $\ast$-$\Lambda$-algebra $R$ \emph{primitive} if the underlying ring is primitive, that is,
admits a faithful irreducible module.

\begin{fact}\lab{prim}
Every regular strictly subdirectly $\ast$-$\Lambda$-algebra  
is primitive.
\end{fact}
\begin{proof}
Any regular ring is semi-simple, i.e. has zero radical, cf. \cite[Corollary 1.2]{good}.
Hence the ring  $R$ is a subdirect product of primitive rings, cf. \cite[Chapter I, \S 3, Theorem 1]{jac}.
Being subdirectly irreducible, the ring $R$ is therefore  primitive, cf. the proof of \cite[Corollary 3.4]{Flo}.
\end{proof}

\section{Endomorphism rings}

In the sequel, let $F$ be  
 a $\ast$-$\Lambda$-algebra where the underlying ring of $F$ is a division ring and $ V_F$ a pre-hermitian space over $F$.
$\End(V_F)$ denotes the
unital $\Lambda$-algebra of $V_F$ of all endomorphisms of the
vector space $V_F$. The algebra 
$\LEnd(V_F)$ is defined in the following Proposition
which is obvious in view of \Fact \ref{sym}.

\begin{fact}\lab{symm2}
The endomorphisms of $V_F$ having an adjoint
form a $\Lambda$-subalgebra
$\LEnd(V_F)$ of $\End(V_F)$
which is  $\ast$-$\Lambda$-algebra  with the involution $\phi\mapsto \phi^\ast$.
If $V'_{F'}$ is similar to $V_F$ then $\LEnd(V_F)$ and 
$\LEnd(V'_{F'})$ are isomorphic $\ast$-$\Lambda$-algebras.
\end{fact}

\noindent
Observe that for $v\in V$, $\lambda\in\Lambda$, and $\varphi\in\LEnd(V_F)$, one has
\[
(\lambda\varphi)(v)=\varphi(v)\lambda,\quad
(\lambda\varphi)^\ast=\lambda^\ast\varphi^\ast.
\]
Also recall the well known facts that for any $\phi,\psi \in \LEnd(V_F)$
\[ \im \phi \subseteq (\im \psi)^\perp 
\;\Leftrightarrow \; \phi^*\circ \psi =0
\;\mbox{ and }  (\im \phi)^\perp =\kerr \phi^*.\]

\begin{fact}\lab{pu}
For any subspace $U$ of $V_F$, one has $V=U\oplus U^\perp $ if and only if there is 
a projection $\pi_U\in\LEnd(V_F)$ such that $U=\im\pi_U$.
Such a projection $\pi_U$ is unique.
\end{fact}

\noindent
Projection $\pi_U$ in terms of \Fact \ref{pu} is called the \emph{orthogonal projection} onto $U$.
Par abus de langage, $\pi_U$ also denotes the induced epimorphism $V\to U$,
while $\varepsilon_U$ denotes the inclusion map  $U\to V$.
Observe that $\pi_U$ and $\varepsilon_U$ are adjoints of each other in the sense that
\[
\sk{\varepsilon_U(u)}{v}=\sk{u}{\pi_U(v)}\quad\text{for all}\ u\in U,\ v\in V.
\]
Moreover, the computational rules of $\LEnd(V_F)$ yield, in particular,
$(\varepsilon_U\varphi\pi_U)^\ast=\varepsilon_U\varphi^\ast\pi_U$ for any $\varphi\in\LEnd(U_F)$.
Finally, $\pi_U\varepsilon_U=\id_U$, while $\pi_U\varepsilon_U\pi_U=\pi_U$ and
$U^\perp=\kerr(\varepsilon_U\pi_U)$.

Let $\dim V_F=n<\omega$. We say that bases $\{v_1,\ldots,v_n\}$ and $\{w_1,\ldots,w_n\}$ of $V_F$
are a \emph{dual pair of bases},
whenever $\sk{v_i}{w_i}=1$ for all $i\in\{1,\ldots,n\}$ and $\sk{v_i}{w_j}=0$ for all $i\neq j$.

\begin{fact}\lab{db}
Let $V_F$ be a pre-hermitian space and let $\dim V_F=n<\omega$.
\begin{enumerate}
\item
There is a dual pair of bases $\{v_1,\ldots,v_n\}$ and $\{w_1,\ldots,w_n\}$ of $V_F$.
Moreover, for any $\varphi\in\End(V_F)$ with $\varphi(v_j)=\sum_iw_ia_{ij}$,
$\varphi^\ast\in\End(V_F)$ exists and $\varphi^\ast(v_i)=\sum_jw_ja_{ij}^\ast$.
In particular, $\LEnd(V_F)$ contains all endomorphisms of $V_F$ and $\LEnd(V_F)$
is regular.
\item
$\LEnd(V_F)$ is $\ast$-regular if and only if $V_F$ is anisotropic.
\item
If $U_F\in\mathbb{O}(V_F)$ then 
$\LEnd(U_F) \times \LEnd(U^\perp_F)$
embeds into $\LEnd(V_F)$, in particular
$\LEnd(U_F)$ is a homomorphic image if a regular
$\ast$-$\Lambda$-subalgebra of $\LEnd(V_F)$.
\end{enumerate}
\end{fact}

\begin{proof}
For existence of dual bases, see \cite[\S IV 15]{jac} or \cite[\S II.6]{kap2}.
Straightforward and well known calculations prove (i).
Regularity of $\LEnd(V_F)$ follows from \Fact \ref{reri}(i).
(ii) 
Assume $R=\LEnd(V_F)$ is $*$-regular;
given $v \neq 0$, choose $V=vF\oplus U$ and the
endomorphism $\varphi$ with $\varphi(v)=v$ and $\varphi|U=0$. Since $R$ is $\ast$-regular,
one has $\varphi R=\pi R$ for some
(orthogonal) projection and $V=vF\oplus^\perp W$ for some $W$.
If one had $\sk{v}{v}=0$, then $v \perp V$,  contradicting
the assumption that $V_F$ is non-degenerate. 
Then converse is well known and easy to prove.
In (iii), let $R$ consist of all
$\varphi\in\LEnd(V_F)$ which leave both $U$ and $U^\perp$ invariant.
As $R\cong\LEnd(U_F)\times\LEnd(U_F^\perp)$, we get (iii).
\end{proof}

\noindent
We put $J(V_F)=\{\varphi\in\LEnd(V_F)\mid\dim\im\varphi<\omega\}$.
Cf. \cite[Theorem 4.6.15]{beidar} for the following.

\begin{fact}\lab{endreg}\lab{alt}
Let $V_F$ be a pre-hermitian space.
\begin{enumerate}
\item First of all,
$J(V_F)$ is an ideal and a strictly simple regular
$\ast$-$\Lambda$-subalgebra of $\LEnd(V_F)$ without unit.
\item
The principal right ideals of $J(V_F)$
are in $1$-$1$-correspondence with the finite-dimensional subspaces of $V_F$ via the map $\varphi J(V_F)\mapsto\im\varphi$; moreover, one has
 $\varphi_0J(V_F)\subseteq\varphi_1J(V_F)$  equivalent to $\im\varphi_0\subseteq\im\varphi_1$. 
\item 
Let  $R$ be a subring  of $\End(V)$ with $R\supseteq J(V_F)$. Then
the minimal right ideals of $R$  are of the form $\varphi R$, where $\varphi\in J(V_F)$ is an idempotent such that 
$\varphi J(V_F)$ is a minimal right ideal of $J(V_F)$, 
that is, $\dim \im \phi=1$.
In particular, $R$ is atomic and $J(V_F)=\Soc(R)$ is its smallest
non-zero ideal.
\item 
For any $\varphi_1$, \ldots, $\varphi_n\in J(V_F)$,
there is $U\in\mathbb{O}(V_F)$ such that $\pi_U\varphi_i=\varphi_i=\varphi_i\pi_U$
for all $i\in\{1,\ldots,n\}$.
\item
The space  $V_F$  is alternate if and only if $J(V_F)$ does not contain a projection
generating a minimal right ideal. If $V_F$ is alternate
then $\pi \circ \pi^* =0=\pi^*\circ \pi$ for any
idempotent $\pi$ with $\dim \im \pi =1$.
\end{enumerate}
\end{fact}

\begin{proof}
(i)
\xc{If $\dim V_F <\omega$ then $J(V_F)=\End(V_F)$ by \Fact \ref{db}(i), and the statements (i)-(iii)
hold trivially in this case.
Let $\dim V_F\geqslant\omega$.}
Clearly, $J(V_F)$ is an ideal and a $\Lambda$-subalgebra of $\LEnd(V_F)$ (without unit).
Observe that $\pi_U\in J(V_F)$ for any $U\in\mathbb{O}(V_F)$ by \Fact \ref{pu}.
Moreover by \Fact \ref{du}, for any subspace  $W$
of $V_F$ with $\dim W<\omega$, there exists
$U\in\mathbb{O}(V_F)$ such that $W\subseteq U$.

Consider $\varphi\in J(V_F)$ and recall that the subspaces
$\kerr\varphi=(\im \varphi^\ast)^\perp$ and $\kerr\varphi^\ast=(\im\varphi)^\perp$ are both closed.
To prove that $\varphi^\ast\in J(V_F)$, choose $W\in\mathbb{O}(V_F)$ such that
$W\supseteq\im\varphi=(\kerr\varphi^\ast)^\perp$.
Then $W^\perp\subseteq\kerr\varphi^\ast$,
whence $\im\varphi^\ast=\varphi^\ast(W)$ is finite-dimensional.
It follows that
\begin{itemize} 
\item[($\ast$)]
For any $\varphi_1$, \ldots, $\varphi_n\in J(V_F)$,
there is $U\in\mathbb{O}(V_F)$ such that 
$U\supseteq\im\varphi_i+\im\varphi_i^\ast$ for all $i\in\{1,\ldots,n\}$
and $\varphi_i(U)=\im\varphi_i$ and $\varphi_i^\ast(U)=\im\varphi_i^\ast$.
In particular,
\begin{enumerate}
\item[(a)]
$U$ is a finite-dimensional pre-hermitian space;
\item[(b)]
$V=U \oplus U^\perp$;
\item[(c)]
$U^\perp\subseteq\bigcap_i\kerr\varphi_i\cap\kerr\varphi_i^\ast$;
\item[(d)]
$\pi_U\in J(V_F)$;
\item[(e)]
$\varepsilon_U\psi\pi_U\in J(V_F)$ and $(\varepsilon_U\psi\pi_U)^\ast=\varepsilon_U\psi^\ast\pi_U$
for any $\psi\in\End(U_F)$.
\end{enumerate}
\end{itemize}
To prove that $\varphi$ has a quasi-inverse in $J(V_F)$,
choose for $\varphi$ a subspace $U\in\mathbb{O}(V_F)$ according to ($\ast$).
By \Fact \ref{db}(i), $\pi_U\varphi\varepsilon_U\in\LEnd(U_F)$ has a quasi-inverse $\psi\in\LEnd(U_F)$.
We claim that $\chi=\varepsilon_U\psi\pi_U$ is a quasi-inverse of $\varphi$ in $J(V_F)$.
Indeed, $\chi\in J(V_F)$ by (e) and $\varphi(v)=0=\chi(v)$ for any $v\in U^\perp$ by (c) and
$\varphi\chi\varphi(v)=\pi_U\varphi\varepsilon_U\psi\pi_U\varphi\varepsilon_U(v)=
\pi_U\varphi\varepsilon_U(v)=\varphi(v)$ for any $v\in U$.

To prove that $J(V_F)$ is strictly simple,
it suffices to show that for any $0\neq\varphi$, $\psi\in J(V_F)$,
$\psi$ belongs to the ideal generated by $\varphi$.
Again, choose for $\varphi$ and $\psi$ a subspace $U\in\mathbb{O}(V_F)$ according to ($\ast$).
Applying \Fact \ref{reri}(i) to $\pi_U\varphi\varepsilon_U$, $\pi_U\psi\varepsilon_U\in\End(U_F)$,
we get that there are $m<\omega$ and $\sigma_1$, \ldots, $\sigma_m$, $\tau_1$, \ldots, $\tau_m\in\End(U_F)$
such that $\pi_U\psi\varepsilon_U=\sum_{i=1}^m\tau_i\pi_U\varphi\varepsilon_U\sigma_i$.
Then according to ($\ast$), $\psi=\sum_{i=1}^m\varepsilon_U\tau_i\pi_U\varphi\varepsilon_U\sigma_i\pi_U$
and $\varepsilon_U\sigma_i\pi_U$, $\varepsilon_U\tau_i\pi_U\in J(V_F)$ for all $i\in\{1,\ldots,m\}$
by (e).

(ii)
We prove first that $\varphi_0J(V_F)\subseteq\varphi_1J(V_F)$ is equivalent to $\im\varphi_0\subseteq\im\varphi_1$
for any $\varphi_0$, $\varphi_1\in\LEnd(V_F)$.
Suppose first that $\im\varphi_0\subseteq\im\varphi_1$ and take an arbitrary $\psi\in J(V_F)$;
then $\varphi_0\psi$, $\varphi_1\psi\in J(V_F)$.
Choose for $\varphi_0\psi$ and $\varphi_1\psi$ a subspace $U\in\mathbb{O}(V_F)$ according to ($\ast$).
Then $\xi_i=\pi_U\varphi_i\psi\varepsilon_U\in\End(U_F)$ for any $i<2$ and
$\im\xi_0\subseteq\im\xi_1$.
As $\dim U_F<\omega$, $\xi_0=\xi_1\chi$ for some $\chi\in\End(U_F)$.
According to (c), $\varphi_0\psi(v)=\varphi_1\psi(v)=0$ for any $v\in U^\perp$, whence
\[
\varphi_0\psi=\pi_U\varphi_0\psi\varepsilon_U\pi_U=\xi_0\pi_U=\xi_1\chi\pi_U=
\pi_U\varphi_1\psi\varepsilon_U\chi\pi_U=
\varphi_1\psi\varepsilon_U\chi\pi_U\in\varphi_1J(V_F),
\]
as $\psi\varepsilon_U\chi\pi_U\in J(V_F)$ by (e).
The reverse implication is trivial by \Fact \ref{du}.

Besides that, for any finite-dimensional subspace $W$ of $V_F$, there is $\varphi\in J(V_F)$ such that $W=\im\varphi$.
Indeed by \Fact \ref{du}, there is $U\in\mathbb{O}(V_F)$ such that $W\subseteq U$,
whence $W=\im\psi$ for some $\psi\in\End(U_F)$.
Then $W=\im\phi$ with $\phi=\varepsilon_U\psi\pi_U\in J(V_F)$ by \Fact \ref{pu} and (e).
This establishes the claimed $1$-$1$-correspondence.

(iii)
For any  $v\neq 0$ in $V$ there is an idempotent $\pi_v$ in 
$J(V_F)$ such that $\im \pi =vF$;
for such $\pi_v J(V_F)= \pi_vR$
is a  minimal right ideal of both  $J(V_F)$ and $R$.   
Indeed, choose $W$ such that $V=vF\oplus W $ and 
let $\pi(v)=v$, $ \pi|W=0$.
Now, for any $0\neq \phi \in R$ one has $\phi(v)\neq 0$
for some $ v \in V$ 
whence $\dim \im \phi \circ \pi_v =1$ and 
$\pi_v R$ a minimal right ideal contained in $\phi R$;
and the  minimal right ideals of $R$ are exactly the $\pi_vR$.
If $ \dim \phi=n $, then $\phi R= \sum_{i=1}^n \pi_{v_i} R$
where $v_1, \ldots ,v_n$ is a basis of  $\im \phi$. 
Thus, $R$ is atomic with $\Soc(R) = J(V_F)$
contained in any  non-zero ideal.

(iv)
Given $\varphi_1$, \ldots, $\varphi_n\in J(V_F)$,
choose a subspace $U\in\mathbb{O}(V_F)$ according to ($\ast$).
Then $\im \varphi_i+\im\varphi_i^\ast\subseteq U$,
whence $\pi_U\varphi_i=\varphi_i$ and $\pi_U\varphi_i^\ast =\varphi_i^\ast$.

(v) 
If $\pi$ is a projection in the $\ast$-ring $J(V_F)$ then by \Fact \ref{pu},
it is an orthogonal projection of $V_F$ and $\sk{v}{v}\neq 0$ for any $0\neq v\in\im\pi$.
Thus, $V_F$ is not alternate.
Conversely, 
assume $V_F$ not alternate. 
If $\dim V_F\geq 2$, in view of \Fact \ref{sym} we may assume that $V_F$ is hermitian.
By \Fact \ref{du}, there is a non-alternate space $0\neq U\in\mathbb{O}(V_F)$.
By \cite[Chapter II \S 2, Corollary 1]{gro}, $U$ has an orthogonal basis.
Thus $U=W\oplus W^\prime$, where $W^\prime\subseteq W^\perp$ and $\dim W=1$.
It follows that $W\in\mathbb{O}(V_F)$ and that $\pi_W$ is a projection generating
a minimal right ideal of $J(V_F)$. 
If $\dim V_F=1$, then
$\id_V$  will do.

Now, let $V_F$ be alternate and $\pi$ an idempotent with
$\im \pi =wF \neq 0$.
Then for all $v \in V$,
$\sk{v}{\pi^*(w)}=\sk{\pi(v)}{w}= \sk{w\lambda}{w} =0$
whence $\pi^*(\pi(w))=\pi^*(w)=0$ and    $\pi^*\circ \pi=0$.
The claim for $\pi^*$ follows since
$\dim \im \pi^*= 1$ by (*).

\end{proof}

\begin{fact}\lab{atex2}
Any $\ast$-$\Lambda$-subalgebra $R$ of $\LEnd(V_F)$ extends to a $\ast$-$\Lambda$-subalgebra $\hat{R}$ of $\LEnd(V_F)$ such that
$J(V_F)$ is the  unique minimal ideal  of $\hat{R}$.
In particular, $\hat{R}$ is strictly subdirectly irreducible and atomic
with the minimal left $[$right$]$ ideals being those of $J(V_F)$.
Moreover, if $R$ is regular then $\hat{R}$ is also regular.
\end{fact}

\begin{proof}
The $\ast$-regular case is due to \cite[Proposition 3.12]{Flo}.
Let $\hat{R}=R+J(V_F)$. Clearly, $\hat{R}$ is a subalgebra of $\LEnd(V_F)$
and $J(V_F)$ is an ideal of $\hat{R}$ by \Fact \ref{endreg}(i).
If $I\neq 0$ is a left ideal of $\hat{R}$ then choose $\varphi\in I$
such
 that $\varphi\neq 0$.
Then by \Fact \ref{du}, $0\neq\pi_U\varphi\in J(V_F)\cap I$ for some $U\in\mathbb{O}(V_F)$.
By \Fact \ref{endreg}(iii), there is a minimal left ideal $M\subseteq J(V_F)\pi_U\varphi\subseteq I$ of $J(V_F)$.
Then $M$ is also a minimal left ideal of $\hat{R}$.
If $I$ is an ideal of $\hat{R}$, then arguing as above and applying simplicity of $J(V_F)$,
which follows from \Fact \ref{endreg}(i), we get that $J(V_F)\subseteq I$.
Finally, \Facts \ref{reri}(ii) and \ref{endreg}(i) imply regularity of $\hat{R}$ when $R$ is regular.
\end{proof}
\noindent
In particular, \Fact \ref{atex2} applies to $R$
consisting of the endomorphisms $v\mapsto v\lambda$ (also denoted as $\lambda\id_V$),
where $\lambda$ is in the center of $F$;
in this case, we denote the corresponding subalgebra $\hat{R}$ by $\hat{J}(V_F)$.\\

\noindent
A \emph{representation} of a $\ast$-$\Lambda$-algebra $R$ within a per-hermitian space $V_F$
is a homomorphism $\vep:R \to \LEnd(V_F)$ of $\ast$-$\Lambda$-algebras;
it is \emph{faithful} if $\vep$ is an injective map. 
In the following, existence is due to Jacobson 
\cite[Chapter IV, \S 12, Theorem 2]{jac}
and Kaplansky  \cite[Theorem 1.2.2]{herstein}.
Uniqueness  is based on an approach via
the Jacobson Density Theorem,
cf.   \cite[Theorem 4.6.8]{beidar}.

\begin{theorem}\lab{atrep2}
Any primitive
$\ast$-$\Lambda$-algebra  having  a minimal right ideal 
is atomic with $\Soc(R)$ as smallest non-zero ideal and admits a
faithful representation  $\vep$ within  some pre-hermitian space $V_F$
such that $\varepsilon(\Soc(R)) =J(V_F)$.
Up to similitude, the space $V_F$ is uniquely determined by $\Soc(R)$. 
\end{theorem}

\begin{proof}
If the underlying ring of   $R$ is a division ring 
then a  representation is given via the scalar product $\sk{\lambda}{\mu}=\lambda^\ast\mu$. Conversely,
given a representation in $V_F$ we have $J(V_F)\cong R$
and may assume $R=\LEnd(F_F)$.
Up to scaling we have $F$ with  
 involution $\nu$ and
scalar product $\sk{\lambda}{\mu}=\lambda^\nu\sk{1}{1} \mu
=\lambda^\nu \mu$.
For the endomorphism given by $\phi_\lambda(\mu) =\lambda \mu$ 
one obtains $\lambda^\nu=\sk{\phi_\lambda(1)}{1}
=\sk{1}{\phi^*_\lambda(1)}= \phi^*_{\lambda}(1)$;
that is, $\nu$ is determined by the involution on $R$.

 Assume that $R$ is not a division ring.
First, we ignore the action of $\Lambda$.
By \cite[Chapter IV, \S 12, Theorem 1]{jac}
there is a non-degenerate sesquilinear space $V_F$  
and an  embedding $\varepsilon\colon R\to\LEnd(V_F)$ such that $\varepsilon(R)\supseteq J(V_F)$.
Since $\dim V_F\geqslant 2$, \Fact \ref{symm} applies and $V_F$ is pre-hermitian,
cf. \cite[Chapter IV, \S 12, Theorem 2]{jac}.
The remaining claims about $R$ follow 
from \Fact \ref{endreg}.

In order to discuss uniqueness of $V_F$ as well as 
the action of $\Lambda$, 
 we have a  closer look on how $R$ relates to 
the pre-hermitian space $V_F$, given a
$*$-ring  embedding
$\vep:R \to \LEnd(V_F)$ such that $\vep$ maps
 $J=\Soc(R)$ onto $J(V_F)$. 

Let $e$ be an idempotent such that $eR$ is a minimal right ideal,
 $\pi= \vep(e)$, $U= \im \pi$,  and $W=\im (\id_V- \pi)$;
then $V=U \oplus W$  and $\dim U=1$.
Choose $0\neq u_0\in U$.
For $\lambda\in F$, there is unique $\varphi_{\lambda}\in\End(V_F)$
such that $\varphi_{\lambda}(u_0)=u_0\lambda$ and $\varphi_{\lambda}|W=0$.
Then $\alpha(\lambda)=\varepsilon^{-1}(\varphi_{\lambda})$
defines a ring  isomorphism of $F$ onto the subring  $eRe$ of $R$.
Moreover, one has an $\alpha$-semilinear bijection
$\omega$ from $V_F$ onto the right $eRe$-vector space $Re$; it is 
given by
 $\omega(v)= \vep^{-1}(\phi_v)$ where $\phi_v(u_0)=v$ and  
$\phi_v|W=0$.  The given  ring  embedding $\vep:R \to \LEnd(V_F)$
can be now described by the formula $\vep(r)(v)= \omega^{-1}(r \omega(v))$.
Compare the proof of \cite[Proposition 4.6.4]{beidar}.

If the action of  $\Lambda$  on $F$  is still to be defined,
put $\zeta \lambda= \alpha^{-1} (\zeta \alpha(\lambda))$  
for $\lambda \in F$ and $\zeta \in \Lambda$.
Then $\vep$  is a $\Lambda$-algebra homomorphism
from $R$ into the $\Lambda$-algebra $\End(V_F)$; indeed
 for any $\zeta \in \Lambda$,
$r\in R$, and $v \in V$ one has 
$\omega(v)=se$ for some $s \in R$ whence
$\vep(\zeta r)(v)=   \omega^{-1}( (\zeta r) \omega(v))
= \omega^{-1}( \zeta r se )
= \omega^{-1}( r se e \zeta e)
= \omega^{-1}( (r \omega(v))  (e \zeta e))
= \omega^{-1}( (r \omega(v)))  \alpha^{-1}(e \zeta e))
=(\vep(r)(v)) \alpha^{-1}(e \zeta e))$.

Assume that $e$  is a projection;
then so is $\pi$ whence $u_0\not\perp u_0$ and $W=U^\perp$.
In view of scaling, we may assume that $\sk{u_0}{u_0}=1$.
Thus,    $\phi_\lambda^\ast= \phi_{\lambda^\ast}$
and $\alpha$ is an isomorphism of $\ast$-rings.
Also,  one obtains for all $v,w \in V$ 
\[ \sk{v}{w}= \sk{\phi_v(u_0)}{\phi_w(u_0)}
= \sk{u_0}{\phi^*_v(\phi_w(u_0))}
= \sk{u_0}{u_0\lambda} =\lambda  \]
where $\phi_v^*(\phi_w(u_0))=u_0\lambda$
for some $\lambda \in F$ since
$\im \phi_v^*=  W^\perp =\im \pi$.
From
 $\phi_\lambda|W=0=\phi_w|W$, it follows
$\phi_\lambda = \phi_v^* \circ \phi_w$.
Summarizing,  the space $V_F$ is
 determined, up to scaling,  by the $*$-ring $J(V_F)$.
Given another $V'_{F'}$ and $\vep':R \to \LEnd(V'_{F'})$,
as in the Theorem, and $u'_0\in \im \vep'(e)$ chosen, accordingly, 
we have $\alpha':F'\to eRe$ and 
$\omega':V' \to Re$ providing an isomorphism
$\beta= {\alpha'}^{-1} \circ \alpha:F \to F'$
of division rings  and a
$\beta$-semilinear bijection $\omega' \circ \omega:V_F 
\to V'_{F'}$ which combine into  an isomorphism 
of the  sesquilinear spaces obtained from  $V_F$ and $V'_{F'}$
by scaling, thus  establishing the claimed similitude.

Now, assume that  $R$, whence $J(V_F)$, does not have any projection
generating a minimal right ideal. 
By \Fact \ref{alt}(v),  $V_{F}$ is an  alternate space;
in particular $\lambda=\lambda^\ast$ for all $\lambda\in F$, $F$ 
is commutative, and 
$ \pi\circ \pi^*=0 =\pi^* \circ \pi$
for any idempotent $\pi$ generating a minimal right ideal
in $J(V_F)$.
Choose $\pi=\vep(e)$.
It follows $\im \pi \cap \im \pi^*= \im \pi \cap (\im \pi^*)^\perp
= \im \pi^* \cap (\im \pi)^\perp =0$.
Also, $U':= \im \pi \oplus \im \pi^* \in\mathbb{O}(V_F)$
and $W=\im \pi^* +U'$. Thus, for any $v \in V$,
$\im \phi_v^*= W^\perp =  \im \pi ^*$.
For $\psi \in\End(V_F)$
we have $\psi(\im \pi^*) =\im \pi$ and 
$\kerr  \psi =U' +  \im \pi$ 
if and only if $\pi \circ \psi = \psi= \psi \circ \pi^*$. 
For any such  $\psi$ there is unique $0\neq \mu \in F$ such that
$\psi(u_1)=u_0\mu$ -- and vice versa.
Also, $\im \psi^*= (U'+\kerr \pi)^\perp=\im \pi$.
Choose $u_1$ such that  $u_1F =\im \pi^*$
and $\sk{u_1}{u_0}=1$.
  Choosing $\mu$ (and $\psi$),  given any  $v,w \in V$ one has 
\[ \sk{v}{w}= \sk{\phi_v(u_0)}{\phi_w(u_0)}
= \sk{\phi_v(\psi(u_1\mu))}{\phi_w(u_0)}=\]\[
= \sk{u_1\mu}{\psi^*(\phi^*_v(\phi_w(u_0)))}
=\sk{u_1}{u_0}\mu\sigma =\mu\sigma\]
where $\psi^*(\phi^*_v(\phi_w(u_0)))=u_0\sigma$.
It follows $\sk{v}{w}=\mu \sigma$  if and only if
$\psi^*\circ \phi^*_v\circ \phi_w =\phi_\sigma$.
Uniqueness of $V_F$ up to similitude follows as above,
with the additional choice of $u_1'$ and of $\mu'=\mu$.
cf. \cite[Proposition 4.6.6]{beidar}.
\end{proof}

\begin{remark}
For a primitive ring with minimal right ideal, according
to Kaplansky (cf. Corollary 4.3.4 and Theorem 4.6.2 \cite{beidar})
there is an idempotent $e$ such that $eR$ and $e^*R$ are minimal
and either $e=e^*$ or $ee^*=0=e^*e$.
Given such, the representation of Theorem \ref{atrep}
can be directly obtained from the Jacobson Density Theorem
in the context of non-empty socle 
cf. \cite[Theorem 4.6.2]{beidar}.
In the first case, $eRe$ is a $*$-subring of $R$.
In the second case, $S$ 
consisting of the $\lambda^+:=\lambda +\lambda^*$, $\lambda \in eRe$,
is a $*$-subring of $R$ and $(\lambda^+)^*=\lambda^+$.
Moreover, $\lambda \mapsto \lambda^+$ is a ring 
isomorphism of $eRe$ onto $S$ with inverse $\mu \mapsto e \mu $
and one has $v\lambda =v\lambda^+$ for all $v \in Re$ and
$\lambda$ in $eRe$. Thus, in both cases,
$\sk{v}{w}= e v^*w$ provides the required scalar
product on $Re$.
\end{remark}

\begin{fact}\lab{P:iso}
Up to isomorphism, the strictly simple artinian  regular $\ast$-$\Lambda$-algebras 
$R$ 
are exactly the endomorphism algebras $\LEnd(V_F)$, where $V_F$ is a pre-hermitian space and $\dim V_F<\omega$.
Moreover, $V_F$ is uniquely determined by $R$ up to similitude;
$V_F$ is anisotropic iff $R$ is $\ast$-regular.
\end{fact}

\begin{proof}
Having a unit, $R$ is noetherian and  $J(V_F)=\LEnd(V_F)$ cf. 
\cite[\S 3.3.5 Proposition 3]{lamb}. Thus, this follows from 
Theorem \ref{atrep2}. 
\end{proof}

\section{Lattices with Galois operator}

We focus on lattices  with Galois operator arising  in Orthogonal Geometry,
cf. \cite[Chapter I \S9]{gro}, \cite{gls}, and  \cite[\S 2]{gb2}. 
For basics on modular lattices we refer to 
\cite[\S 3--4, \S10, \S13]{CD}, alternatively 
\cite[Chapter V \S1, \S5]{gr}.
We consider \emph{lattices}, $L$,
as algebraic structures with binary operations
$\cdot$ (\emph{meet}) and $+$ (\emph{join});
that is,  for a suitable (unique)  partial order $\leq$,
$ab=a\cdot b=\inf\{a,b\}$, $a+b=\sup\{a,b\}$.
$L$ is \emph{modular} if
\[
a\geq c\ \text{implies}\ a(b+c)=ab+c.
\]
For both concepts 
there is a well known equivalent definition just by equations.
If $L$ has a smallest element $0$ and if $ab=0$ then
we write $a\oplus b$ instead of $a+b$.

A \emph{sublattice} of $L$ is a subset of $L$ closed under
meets and joins  and  a lattice (modular if so is $L$) endowed with
the restrictions of these operations;
for example the \emph{intervals}
$[u,v]=\{x \in L\mid u\leq x \leq v\}$.
A \emph{homomorphism} $\varphi:L\to M$ between lattices
is a map such that $\varphi(ab)=\varphi(a) \varphi(b)$ and
$\varphi(a+b)=\varphi(a)+\varphi(b)$ for all $a,b \in L$.
A \emph{congruence} (\emph{relation}) on a lattice $L$
is an equivalence relation $\theta$
 which is compatible with meet and join,
that is $a \,\theta\,b$ and $c\,\theta\,d$ jointly imply
$ac\,\theta\,bd$  and $(a+c)\,\theta\,(b+d)$. 
If $\varphi:L\to M$ is a  homomorphism 
then $a\,\theta \,b \Leftrightarrow \varphi(a)=\varphi(b)$
defines a congruence on $L$;  any congruence arises
this way with surjective $\varphi$ and if $L$ is modular so is $M$.
A lattice $L$ is \emph{subdirectly irreducible}
if it has a smallest non-trivial congruence $\mu$,
the \emph{monolith} of $L$.

A modular lattice $L$ has \emph{dimension} $n<\omega$,
(which is denoted by $\dim L$),
if $L$ has $(n+1)$-element maximal chains.
If $L$ has smallest element $0$, we put $\dim a =\dim[0,a]$ 
if that exists; 
we  call $a$
an \emph{atom} if  $\dim a=1$.

A \emph{bounded lattice} has smallest element
$0=\inf L$ and greatest element  $1=\sup L$ which are considered as constants.
A  bounded lattice $L$ is \emph{complemented} if for any
$a\in L$, there is $b\in L$ such that $a\oplus b=1$.
In a CML (i.e., a complemented modular lattice) $L$, any interval $[u,v]$ is complemented, too;
$L$ is \emph{atomic} if for any $a>0$ there is an atom $p\leq a$.
It follows that for elements $a$ and $b$ of an atomic CML $L$,
$a\nleq b$ in $L$ if and only if there is an atom $p\in L$ such that $p\leq a$ and
$p\nleq b$. Primary examples of atomic CMLs are the
lattices of all subspaces of vector spaces (see Proposition \ref{arg}, below).\\

Lattices relevant in orthogonal geometry
also support  an  operation $X \mapsto X^\perp$,
the Galois correspondence induced by the orthogonality relation,
 see \cite{gro}. 
This is captured by the following concept (cf. \cite{gls,gb2}). 
A \emph{Galois lattice} is a
bounded  lattice $L$  endowed with an additional operation
$x\mapsto x^\prime$ such that
$x \leq y^\prime$ implies $y \leq x^\prime$
for any $x,y\in L$ and such that  $1'=0$.
It is well known and easy to prove that $x \leq x^{\prime\prime}$,
that $x\leq y$ implies $y^\prime\leq x^\prime$,
that  $x^{\prime\prime\prime}=x^\prime$, 
  that $0'=1$, 
and that $(x+y)^\prime =x^\prime y^\prime$. 
For an equational definition see   \cite[IV.2.5]{gb2}.
A  \emph{Galois sublattice} $S$ of a Galois lattice $L$ is
a sublattice of $L$ such that $0,1 \in S$ and
 $x' \in S$ for
any $x \in S$; thus, it  is a Galois lattice with the
inherited operations. Similarly, a homomorphism $\varphi:L\to M$  
between 
Galois lattices is a lattice homomorphism $L \to M$ 
preserving $0,1$ and  such that 
$\varphi(x') =(\varphi(x))'$ for all $x \in L$. 
Also, a lattice congruence $\theta$  on $L$ is a \emph{Galois lattice
congruence} if $a\,\theta\,b$ implies $a'\,\theta\,b'$ 
for all $a,b \in L$.

A \emph{polarity lattice} is a Galois lattice
 such that $p^\prime$ is a dual atom of $L$ for any atom $p\in L$.
The referee pointed out the  need for such concept and provided
the smallest example of a Galois CML which is not
a polarity lattice: the $4$-element CML  with $x'=0$ for $x\neq 0$ and
 $0'=1$. Also, the class of modular polarity lattices is not closed under
 substructures: Consider a  subspace $X$ of a Hilbert space
 $H$ such that $X \neq X^c$, the closure of $X$; then     
then $0, X, X^c, X^\perp, X+X^\perp, H$ form a subalgebra
of the polarity lattice associated with $H$ cf. Proposition \ref{og0}
below.  It remains open whether the class of  CML polarity lattices
is closed under  complemented subalgebras.

A Galois lattice $L$ is a  \emph{lattice with involution}
if, in addition, $x^{\prime\prime}=x$ for all $x\in L$; equivalently,
if $x\mapsto x^\prime$ is a dual automorphism of order $2$ of the
lattice $L$; in particular, such is a polarity lattice.
$L$ is an \emph{ortholattice} if, in addition, the involution
satisfies  the identity $xx^\prime=0$ (or equivalently, $x+x^\prime=1$).
We write MIL [CMIL] shortly for a [complemented] modular lattice with involution
and MOL for a modular ortholattice.
We use each abbreviation also to denote the class of \emph{all}
Galois lattices
with the corresponding property.
Observe that $\dim u=\dim [u^\prime,1]$ in any MIL.
Also, observe that  in a  CMIL, in general,  
$x'$   fails to be a complement of $x$: cf. the
$4$ element CML with $p=p'$ for each  atom.
 The following statement is straightforward to prove.

\begin{lemma}\lab{L:1}
Let $L_0$, $L_1$ be  lattices with involution.
\begin{enumerate}
\item
A map $\varphi\colon L_0\to L_1$ is a homomorphism, if
$\varphi(x+y)=\varphi(x)+\varphi(y)$, $\varphi(x^\prime)=\varphi(x)^\prime$ for all $x$, $y\in L_0$, and $\varphi(0)=0$.
\item
A subset $X\subseteq L_0$ is a Galois sublattice of $L_0$,
if $0\in X$ and $X$ is closed under the operations $+$ and $\ ^\prime$.
\end{enumerate}
\end{lemma}
For a modular polarity lattice $L$, let $L_{\sf f}= F \cup \{u'\mid u \in F\}$ 
where $F=\{u\in L\mid  \dim u<\omega\}$.

\begin{fact}\lab{gal}
If $L$ is a polarity CML then $L_{\sf f}$ is a Galois
sublattice of $L$   and $L_{\sf f}$ is an atomic CMIL
which is the directed union of its subalgebras $[0,u]\cup [u^\prime,1]$,
where $\dim u<\omega$ and $u\oplus u^\prime=1$ $($which are all CMILs$)$.
If $L$ is a CMIL, then $L_{\sf f}=\{a\in L\mid\dim a<\omega\ \text{or}\ \dim[a,1]<\omega\}$.
\end{fact}

\begin{proof}
We have $\dim[v^\prime,u^\prime]\leqslant\dim[u,v]$ if the latter is finite.
Indeed, if $\dim[u,v]=1$ then $v=u+p$,
where $p$ is a complement of $u$ in $[0,v]$,
whence an atom and so $v^\prime=u^\prime p^\prime$ is a lower cover of $u^\prime$ unless $v^\prime=u^\prime$.
Thus, if $\dim u<\omega$, then $\dim u^{\prime\prime}\leq\dim [u^\prime,1]\leq\dim u$,
$u^{\prime\prime}=u$, and $x\mapsto x^\prime$ provides a pair of mutually
inverse lattice anti-isomorphisms between the intervals $[0,u]$ and $[u^\prime,1]$ of $L$.
Therefore, $[u^\prime,1]\subseteq L_{\sf f}$.
Since $\{u\in L\mid \dim u<\omega\}$ is closed under joins and $0\in L_{\sf f}$,
$L_{\sf f}$ is a Galois sublattice of $L$ by Lemma \ref{L:1}(ii) and, in particular, $L_{\sf f}$ is a MIL.

If $X\subseteq L_{\sf f}$ is finite, then there is $u\in L$ such that $\dim u<\omega$ and
$X=Y\cup Z$, where $y$, $z^\prime\in[0,u]$ for all $y\in Y$, $z\in Z$.
Choose $v$ as a complement of $u+u^\prime$ in $[u,1]$.
Then $\dim[u,v]=\dim[u+u^\prime,1]\leqslant\dim[u^\prime,1]=\dim u<\omega$ whence $\dim v<\omega$
We have $u$, $u^\prime$, $v\in L_{\sf f}$. Therefore, $u=v(u+u^\prime)$ implies
$u^\prime=v^\prime+uu^\prime$. It follows that
$v+v^\prime=v+u+uu^\prime+v^\prime=v+u+u^\prime=1$ and $vv^\prime=(v+v^\prime)^\prime=1^\prime=0$
whence $X\subseteq[0,v]\cup[v^\prime,1]$.
This proves the first statement.

If $L$ is a CMIL and
$\dim[a,1]<\omega$ then $\dim a^\prime=\dim[a,1]<\omega$, thus $a^\prime\in L_{\sf f}$ and $a=a^{\prime\prime}\in L_{\sf f}$.
\end{proof}

\noindent
From a lattice congruence $\theta$ on a MIL $L$, we put
$a\,\theta^\prime\,b :\Leftrightarrow a\,\theta\, b$.
Then $\theta^\prime$ is also a lattice congruence on $L$ and
the Galois lattice  congruences on $L$ are 
exactly the lattice congruences $\theta$ on $L$ such that 
$\theta=\theta^\prime$. 
We call an MIL  $L$ \emph{strictly subdirectly irreducible} 
if the underlying lattice
is subdirectly irreducible;
in that case, one has  $\mu=\mu^\prime$ for the lattice monolith.
Similarly, the MIL  $L$ is \emph{strictly simple} if 
the underlying lattice  is simple.
In the case of MOLs one has $\theta=\theta^\prime$ for all
$\theta$; thus subdirectly irreducible MOLs [simple MOLs] are strictly subdirectly irreducible
[strictly simple, respectively]. The following is well known.

\begin{fact}\lab{atom}
A  subdirectly irreducible CML $L$ is atomic provided it contains an atom.
If $L$ is, in addition, a  CMIL with  lattice monolith  $\mu$, then  one has $a\in L_{\sf f}$ iff $a\,\mu\,0
  $ or 
$a\,\mu\,1$.
In particular, $L_{\sf f}$ is strictly subdirectly irreducible and atomic, too.
\end{fact}

\begin{proof}
Let $p$ be an atom in $L$. By modularity,
the smallest  lattice congruence $\mu$ 
such that 
 $0\,\mu\,p)$ is minimal.
Thus given $a>0$, one has $0\,\theta\, p$ in the
smallest  lattice congruence
such that  $0\,\theta\,a$, whence by modularity,
the quotient  $p/0$ is projective
to some subquotient $c/d$ of $a/0$.
Then any complement $q$ of $d$ in $[0,c]$ is an atom.
Thus $L$ is atomic and it follows that
$x\,\mu\,y$ iff $\dim[xy,x+y]<\omega$. In view of \Fact \ref{gal}, we are done.
\end{proof}

The following Proposition associates a CMIL $\mathbb{L}(R)$ with any
regular $*$-$\Lambda$-algebra $R$.

\begin{fact}\lab{neu}
\begin{enumerate}
\item
The principal right ideals of a regular ring $R$, possibly without unit, form a sublattice $L(R)$, containing $0$,
of the lattice of all right ideals of $R$; $L(R)$ is sectionally complemented and modular. In the case with unit, $L(R)$ is a CML with top element $R$.
\item
For  any regular $[\ast$-regular$]$ 
$\ast$-$\Lambda$-algebra $R$, the CML 
$L(R)$ becomes  a CMIL $[$MOL, respectively$]$ endowed with the involution
$eR\mapsto(1-e^\ast)R$, where $e$ is an idempotent $[$a projection, respectively$]$;
we denote it by $\mathbb{L}(R)$.
\item
 For  any regular $[\ast$-regular$]$ 
$\ast$-$\Lambda$-algebras $R_i$, $i\in I$, and $R=\prod_{i\in I}R_i$
one has $\mathbb{L}(R)\cong\prod_{i\in I}\mathbb{L}(R_i)$.
\item
If $\varepsilon\colon R\to S$ is a homomorphism and $R$, $S$ are regular rings,
then $\overline{\varepsilon}\colon L(R)\to L(S)$,
$\overline{\varepsilon}\colon aR\mapsto\varepsilon(a)S$ is a 
lattice homomorphism preserving $0$ and $1$.
If $\varepsilon$ is injective, then so is $\overline{\varepsilon}$;
if $\varepsilon$ is surjective, then so is $\overline{\varepsilon}$.
If $R$ and  $S$  are  regular  
$\ast$-$\Lambda$-algebras and $\vep$ a homomorphism of such,
then $\overline{\varepsilon}\colon\mathbb{L}(R)\to\mathbb{L}(S)$ is a homomorphism between MILs.
\end{enumerate}
\end{fact}

\noindent
In \Fact \ref{neu}(ii), one can consider the preorder
$e\leq f$ iff $fe=e$ on the set of idempotents of $R$
and obtain the lattice $\mathbb{L}(R)$ factoring by the equivalence relation
$e\sim f$ iff $e\leq f\leq e$;
the involution is given by $e\mapsto 1-e^\ast$.
In the $\ast$-regular case,
any of the equivalence classes contains a unique projection
so that $\mathbb{L}(R)$ is also called the \emph{projection $[$ortho$]$lattice} of $R$.

Recalling  \Fact \ref{endreg}
for a pre-hermitian space $V_F$,
the principal right ideals of $J(V_F)$ form an atomic sectionally complemented sublattice
of the lattice of all right ideals of $J(V_F)$, which is
isomorphic to the lattice of finite-dimensional subspaces of $V_F$ via the map $\varphi J(V_F)\mapsto\im\varphi$.

\begin{proof} These results originate with \cite{fry}.
(i)--(ii): See \cite[\S 8-3.3.13]{fred}.
For $R\in\mathcal{R}_{\Lambda}$, the map $eR\mapsto\xc{{\sf Ann}^l(eR)=}R(1-e)\mapsto(1-e^\ast)R$
combines a dual isomorphism of $L(R)$ onto the lattice of left principal ideals
with an isomorphism of the latter onto $L(R)$.
(iii)--(iv):  See \cite[\S 8-3.3.14--15]{fred}.
\end{proof}

\section{Projective spaces and orthogeometries}

Synthetic geometries are a convenient link between lattice and 
vector space structures.
We follow \cite[Chapter 2]{FF}.
A \emph{projective space} $P$
is a set, whose elements are called \emph{points}, endowed with a ternary relation
$\Delta\subseteq P^3$ of \emph{collinearity} satisfying the following conditions:
\begin{enumerate}
\item
if $\Delta(p_0,p_1,p_2)$, then $\Delta(p_{\sigma(0)},p_{\sigma(1)},p_{\sigma(2)})$
and $p_{\sigma(0)}\neq p_{\sigma(1)}$
for any permutation $\sigma$ on the set $\{0,1,2\}$;
\item
if $\Delta(p_0,p_1,a)$ and $\Delta(p_0,p_1,b)$, then $\Delta(p_0,a,b)$;
\item
if $\Delta(p,a,b)$ and $\Delta(p,c,d)$, then $\Delta(q,a,c)$ and $\Delta(q,b,d)$ for some $q\in P$.
\end{enumerate}
The space
$P$ is \emph{irreducible} if for any $p\neq q$ in $P$ there is $r\in P$ such that $\Delta(p,q,r)$.
A set $X\subseteq P$ is a \emph{subspace} of $P$
if $p$, $q\in X$ and $\Delta(p,q,r)$ together imply that $r\in X$.

\noindent
Any projective space $P$ is the disjoint union of its irreducible subspaces $P_i$, $i\in I$,
which are called its \emph{components}.
The set $L(P)$ of all subspaces of an [irreducible] projective space $P$ is a complete
[subdirectly irreducible] atomic CML, in which 
the atoms are the subspaces $\{p\}$, $p \in P$, all of which   
are compact.
Moreover, $L(P)\cong\prod_{i\in I}L(P_i)$ via the map $X\mapsto(X\cap P_i\mid i\in I)$.
Conversely, any complete atomic CML $L$ with compact atoms is isomorphic to
$L(P)$ via the map $a\mapsto\{p\in P\mid p\leq a\}$,
where $P$ is the set of atoms of $L$ and distinct points $p$, $q$, $r\in P$ are collinear if and only if $r<p+q$.
Recall that J\'{o}nsson's \emph{Arguesian} lattice identity \cite{jo2}
holds in $L(P)$ if and only if $P$ is desarguean.
For a vector space $V_F$, let $L(V_F)$ denote the lattice of all linear subspaces of $V_F$.

\begin{fact}\lab{arg}
\begin{enumerate}
\item
For any vector space $V_F$, $L(V_F)$ is a CML.
Moreover, there exists an irreducible desarguean projective space $P$ such that
$L(V_F)\cong L(P)$.
\item
For any irreducible desarguean projective space $P$ with $\dim L(P)>2$,
there is a vector space $V_F$ such that $L(P)\cong L(V_F)$;
$F$ is  unique up to isomorphism, $V_F$ up to semilinear bijection.
\item
If $P$ is irreducible and $\dim L(P)>3$, then $P$ is desarguean.
\item
Any subdirectly irreducible CML of dimension at least $4$ is Arguesian.
\end{enumerate}
\end{fact}

\begin{proof}
Claim (i) is the content of \cite[Proposition 2.4.15]{FF}.
For (ii), see \cite[Proposition 2.5.6]{FF} and \cite[Chapter 9]{FF}.
For (iii), see \cite[Chapter 13]{CD}. As to claim (iv),
according to Frink \cite{frink}, any CML $L$ embeds into $L(P)$ for some projective space $P$.
Since $L$ is subdirectly irreducible as a lattice, it embeds into $L(P_i)$
for some irreducible component $P_i$ of $P$, which is desarguean since $\dim L(P_i)>3$,
whence statement (iv) follows.
\end{proof}

\begin{fact}\lab{L:ortho} Let $P$ be a projective space.
There is a 1-1-correspondence between maps
$X \mapsto X^\perp$  turning $L(P)$ into a Galois CML  $\mathbb{L}(P,\perp)$
on one side 
and, on the other side,  symmetric binary relations $\perp$  on $P$
such that
\begin{itemize} 
\item[(a)]
for any  $p$, $q$, $r$, $s\in P$, if
 $p\perp q$, $p\perp r$, and $\Delta(q,r,s)$ then  $p\perp s$; 
\item[(b)] for any $p \in P$ there is 
$q\in P$ with $p\not\perp q$.
\end{itemize} 
The correspondence is given by 
\[
X^\perp:=\{q\in P\mid q\perp p\ \text{for all}\ p\in X\},
\quad \quad   p \perp q :\Leftrightarrow p \leq q^\perp.
\] 
Given such relation $\perp$, the following are equivalent
\begin{enumerate} 
\item The Galois CML  $\mathbb{L}(P,\perp)$ is a polarity CML;
\item for all $p,q,r \in P$,
if  $p \neq q$, then $r\perp t$ for some $t\in P$ such that
$\Delta(p,q,t)$;
\item  same as (ii) with additional hypotheses
$r\not\perp p$ and $r \not\perp q$.
\end{enumerate}
\end{fact}
A  pair $(P,\perp)$ satisfying all of the above is an 
\emph{orthogeometry}. Compare \cite[Definition 14.1.1]{FF} and
 \cite[\S 4]{he4}. $\mathbb{L}(P,\perp)$ is defined, accordingly.

\begin{proof}
Given $\perp$, 
it follows   $X^\perp\in L(P)$ by (a) while
(b) and symmetry of $\perp$ yield that $\mathbb{L}(P,\perp)$
is a Galois lattice.  The converse is obvious.
Assuming (i), if $\dim X=2$ and $r \in P$ then
$X \cap \{r\}^\perp \neq \emptyset$ by modularity,
proving (ii) cf. \cite[Remark 14.1.2]{FF}.
 While  (iii) is a special case of  (ii)
it implies that $\{r\}^\perp$ is a coatom: 
by (b) choose  $p \in P$, $p \not\perp r$.
Now, for any $q \in P$, if $q \neq p$ and $q \not\in \{r\}^\perp$
one has $\Delta(p,q,t)$ for some $t \in \{r\}^\perp$
whence $q \in \{r\}+ \{r\}^\perp$. This  proves 
 $\{r\}+ \{r\}^\perp =P$ and,  by modularity, that 
$\{r\}^\perp$ is a coatom of $L(P)$. 
\end{proof}

For a MIL, $L$, let $P_L$ denote  the set of atoms of $L$.
We define a collinearity on $P_L$ by putting $\Delta(p,q,r)$ for distinct atoms $p$, $q$, $r\in P_L$
such that $p\leq q+r$ in $L$.
Furthermore, we put $p\perp q$ if $p\leq q^\prime$.

\begin{fact}
\cite[Lemma 4.2]{he4}
For any MIL $L$, $\mathbb{G}(L)=(P_L,\perp)$ is an orthogeometry.
\end{fact}

\begin{fact}\lab{arg2}
For any orthogeometry $(P,\perp)$,
the Galois lattice  $L=\mathbb{L}(P,\perp)_{\sf f}$,
consisting of all $X,X^\perp$ with $X\in L(P)$ and $ \dim X< \omega$,  is a CMIL with $L=L_{\sf f}$.
Conversely, for any CMIL $L$ with $L=L_{\sf f}$, one has $L\cong\mathbb{L}(\mathbb{G}(L))_{\sf f}$.
\end{fact}

\begin{proof}
See \cite[Theorem 1.1]{he4} and \Fact \ref{gal}.
\end{proof}

\section{Subspace lattices}

 The following Proposition relates any pre-hermitian space  $V_F$ with 
an orthogeometry  $\mathbb{G}(V_F)$
and a polarity  lattice $\mathbb{L}(V_F)$.

\begin{fact}\lab{og0}
If $V_F$ is a pre-hermitian space then 
the CML of all linear subspaces becomes a 
polarity lattice 
$\mathbb{L}(V_F)$  with 
the unary operation $U \mapsto U^\perp$;
moreover, $U \mapsto \{vF\mid 0\neq v \in U\}$
defines an isomorphism of
$\mathbb{L}(V_F)$ onto 
 $\mathbb{L}(\mathbb{G}(V_F))$
where $\mathbb{G}(V_F)=(P,\perp)$
is the  orthogeometry $(P,\perp)$
with $P=\{vF\mid 0\neq v \in V\}$
and $vF \perp wF \Leftrightarrow v \perp w$.
\end{fact}
\begin{proof}
Since  orthogonals are subspaces, (a) of \Fact \ref{L:ortho}
is satisfied  while (b) follows from non-degeneracy. 
Thus, $\mathbb{L}(V_F)$ is a Galois CML.
Observe that for any $u \in V$,
$f_u(v) =\sk{u}{v}$ is a linear map $V_F\to F_F$; thus,
one has $uF^\perp =\kerr f_u$ a coatom and $\mathbb{L}(V_F)$  a  polarity lattice.
The isomorphism $\mathbb{L}(V_F) \to \mathbb{L}(\mathbb{G}(V_F))$
being obvious, it follows 
that $\mathbb{G}(V_F)$ is an orthogeometry.
See also \cite[Proposition 14.1.6]{FF},
\end{proof}

\begin{fact}\lab{og} 
Let $V_F$ be a pre-hermitian space.
\begin{enumerate}
\item The polarity  lattice
$\mathbb{L}(V_F)_{\sf f}$ is the directed union of its  Galois sublattices $[0,U]\cup[U^\perp,V]$, $U\in\mathbb{O}(V_F)$.
Moreover, for any $U\in\mathbb{O}(V_F)$
there is a Galois lattice embedding 
of $[0,U]\cup[U^\perp,V]$ 
into  $\mathbb{L}(W_F)$
for some $W\in   \mathbb{O}(V_F)$.
\item 
The Galois lattices 
$\mathbb{L}\bigl(\mathbb{G}(V_F)\bigr)_{\sf f}$ and $\mathbb{L}(V_F)_{\sf f}$ 
are mutually isomorphic and  strictly subdirectly irreducible
Arguesian CMILs; if $V_F$ is anisotropic, then $\mathbb{L}(V_F)_{\sf f}$ is a MOL.
\item
For any strictly subdirectly irreducible Arguesian CMIL $L$ of dimension at least $3$ such that $L=L_{\sf f}$, there is
a $($unique up to similitude$)$ pre-hermitian space $V_F$ such that
$L\cong\mathbb{L}(V_F)_{\sf f}$; if $L$ is a MOL, then $V_F$ is anisotropic.
\item
If  $\dim V_F<\omega$ then $\mathbb{L}(V_F)$ is a MIL.
\end{enumerate}
\end{fact}
In (iv), for  cardinality reasons,  
  $\dim V_F<\omega$ is also necessary; see also \cite{Kel}.

\begin{proof}
The first claim in (i) and, in case  $U\neq V$,
$[0,U]\cup[U^\perp,V] \cong
\mathbb{L}(U_F)\times\textbf{2}$
 follow from \Facts \ref{gal} and \ref{og0}.
Moreover,
 this Galois lattice  is isomorphic to 
the Galois sublattice 
$[0,U]\cup [U^\perp \cap W, W]$ 
of 
$\mathbb{L}(W_F)$, choosing $W$ by \Fact  \ref{du}
such that $U\subset W\in\mathbb{O}(V_F)$.

To prove (ii), we notice first that $\mathbb{L}\bigl(\mathbb{G}(V_F)\bigr)_{\sf f}$ is a CMIL by \Facts \ref{og0} and \ref{arg2}.
Moreover, as a sublattice of $\mathbb{L}(V_F)$,
 $\mathbb{L}(V_F)_{\sf f}$ is an Arguesian lattice.
Strict subdirect irreducibility of $\mathbb{L}\bigl(\mathbb{G}(V_F)\bigr)_{\sf f}$ follows from Fact \ref{og0},
\cite[Example 2.7.2]{FF}, and \cite[Corollary 1.5]{he4}. Furthermore, if $V_F$ is anisotropic, then
$X^\perp$ is an orthocomplement of $X$ for any $X\in\mathbb{L}(V_F)$ with $\dim X<\omega$.

We prove now (iii).
 By \cite[Corollary 1.5]{he4}, there is an irreducible orthogeometry $(P,\perp)$
such that $L\cong\mathbb{L}(P,\perp)_{\sf f}$.
Combining  \Fact \ref{arg}(ii) and \cite[Theorem 14.1.8]{FF},
one gets  a pre-hermitian  space $V_F$ 
such that  $\mathbb{L}(P,\perp)_{\sf f}\cong\mathbb{L}(V_F)_{\sf f}$.
For uniqueness, see \cite[Theorem 14.3.4]{FF} or \cite[p. 33]{gro}.
If $L$ is a MOL, then $V_F$ is obviously anisotropic.

Finally, if $\dim V_F<\omega$, then $\mathbb{L}(V_F)=\mathbb{L}(V_F)_{\sf f}$ is a MIL by \Facts \ref{og0} and \ref{gal}.
\end{proof}

\begin{fact}\lab{atex}
Any Galois sublattice $L$ of $\mathbb{L}(V_F)$ which is a MIL extends
to a Galois sublattice  $\hat{L}$ of $\mathbb{L}(V_F)$
which is a MIL and such that $\hat{L}_{\sf f}=\mathbb{L}(V_F)_{\sf f}$.
In particular, $\hat{L}$ is a strictly subdirectly irreducible atomic MIL.
Moreover, if $L$ is a CMIL then $\hat{L}$ is a CMIL.
\end{fact}

\begin{proof}
Existence of $\hat{L}$ with the required properties follows from the proof of \cite[Theorem 2.1]{he4}.
In particular, $\hat{L}$ is atomic.
Strict subdirect irreducibility of $\hat{L}$ follows from \cite[Corollary 1.5]{he4},
see also \Fact \ref{og}(ii). For a first such construction see \cite{br}.
\end{proof}

A \emph{representation} of a MIL (or CMIL) $L$ in $V_F$ is a homomorphism $\varepsilon\colon L\to\mathbb{L}(V_F)$ of Galois lattices.
It is \emph{faithful} if it is injective, i.e. an embedding;
in this case, we usually identify $L$ with its image in $\mathbb{L}(V_F)$.
A map $\varepsilon\colon L\to\mathbb{L}(V_F)$ is a representation if
and only if  it preserves joins, involution,
and the least element.

\begin{lemma}\lab{MOL}
Let $\varepsilon$ be a representation of a MIL $L$ in a pre-hermitian space $V_F$.
\begin{enumerate}
\item
Any element in the image of $\varepsilon$ is closed.
\item
If $\varepsilon$ is faithful and $V_F$ is anisotropic, then $L$ is a MOL.
\end{enumerate}
\end{lemma}

\begin{proof}
Let $x\in L$ be arbitrary.

(i)
We have $\varepsilon(x)=\varepsilon(x^{\prime\prime})=\varepsilon(x^\prime)^\perp=\varepsilon(x)^{\perp\perp}$.

(ii)
If $V_F$ is anisotropic, then we have 
$\varepsilon(xx^\prime)=\varepsilon(x)\cap\varepsilon(x)^\perp=0$.
As $\varepsilon$ is faithful, we conclude that $xx^\prime=0$. Hence $^\prime$ is an orthocomplementation.
\end{proof}

\noindent
A \emph{representation} of a MIL $L$ within an orthogeometry
$(P,\perp)$ is a homomorphism $\eta\colon L\to\mathbb{L}(P,\perp)$.
The following obvious fact relates the two concepts of a representation.

\begin{fact}\lab{ogrep}
For a MIL $L$, $\varepsilon$ is a $[$faithful$]$ representation in $V_F$
if and only if the mapping $\eta\colon a\mapsto\{p\in P_V\mid p\subseteq\varepsilon(a)\}$
is a $[$faithful$]$ representation of $L$ in the orthogeometry $\mathbb{G}(V_F)$.
\end{fact}

\begin{theorem}\lab{atrep}
Let $L$ be an Arguesian strictly subdirectly irreducible CMIL $[$MOL$]$ such that $\dim L>2$ and $L$ has an atom.
Then $L$ admits a faithful representation $\varepsilon$ within some $[$anisotropic$]$ pre-hermitian space $V_F$
such that $\varepsilon$ induces a bijection between the sets of atoms of $L$ and of $\mathbb{L}(V_F)$.
In particular, $\varepsilon$ restricts to an isomorphism from $L_{\sf f}$ onto $\mathbb{L}(V_F)_{\sf f}$.
The space $V_F$ is unique up to similitude.
\end{theorem}

\begin{proof}
By \Fact \ref{atom}, $L$ is atomic and
$L_{\sf f}$ is strictly subdirectly irreducible and atomic.
Moreover by \Fact \ref{og}(iii), $L_{\sf f}\cong\mathbb{L}(V_F)_{\sf f}$ for some
[anisotropic] pre-hermitian space $V_F$ which is unique up to isomorphism and scaling.
By definition and \Fact \ref{arg2}, $\mathbb{G}(L)=\mathbb{G}(L_{\sf f})\cong \mathbb{G}(V_F)$.
By \cite[Lemma 10.4]{he4}, $L$ has a faithful representation within the orthogeometry $\mathbb{G}(L)$, whence
in the orthogeometry $\mathbb{G}(V_F)$.
The desired conclusion follows from \Fact \ref{ogrep}.
\end{proof}

\noindent
The following fact is a corollary of Theorem \ref{atrep} which is in principle already in \cite{BvN}.

\begin{fact}\lab{atrepc}
Up to isomorphism, the strictly simple Arguesian CMILs $L$ of finite dimension $n>2$ are the 
$\mathbb{L}(V_F)$, where $V_F$ is a pre-hermitian space with $\dim V_F=n$. The space
$V_F$ is determined by $L$ up to similitude; $V_F$ is anisotropic, iff $L$ is a MOL.
\end{fact}

\begin{proposition}\lab{Lrep}
\hfill
\begin{enumerate}
\item
If $\varepsilon$ is a faithful representation of
the regular $\ast$-$\Lambda$-algebra $R$ in a pre-hermitian space $V_F$,
then the map $\eta\colon aR\mapsto\im\varepsilon(a)$
defines a faithful representation of $\mathbb{L}(R)$ in $V_F$.
\item
If $\dim V_F<\omega$ then $\mathbb{L}(V_F)\cong\mathbb{L}(\LEnd(V_F))$.
\end{enumerate}
\end{proposition}

\begin{proof}
(i)
We refer to \cite{Luca}.
We may assume that $R\subseteq\LEnd(V_F)$; that is, $\varepsilon$ is the inclusion map.
By \Facts \ref{reri}(i) and \ref{neu}(iv),
$\eta$ is a $0$ and $1$ preserving  lattice embedding of $L(R)$ into $L(V_F)$.
Moreover, for any $v\in V$ and an idempotent $\varphi\in R$,
one has $v\in\bigl(\eta(\varphi R)\bigr)^\perp=(\im\varphi)^\perp$
iff $\sk{\varphi^\ast(v)}{w}=\sk{v}{\varphi(w)}=0$ for all $w\in V$,
iff $\varphi^\ast(v)=0$,
iff $v=(\id_V-\varphi^\ast)(v)$,
iff $v\in\im(id_V-\varphi^\ast)=\eta\bigl((\varphi R)^\prime\bigr)$,
whence $\eta$ preserves the involution.

(ii)
By (i) and \Fact \ref{db}(ii), the identical map $\varepsilon$ on $\LEnd(V_F)$ defines a faithful representation of $\mathbb{L}(V_F)$.
It is surjective since any subspace is the image of some endomorphism
$\varphi\in\LEnd(V_F)$, cf. also \Fact \ref{neu}(iv).
\end{proof}

\section{Representations as multi-sorted structures}\lab{sort}

Given the commutative $*$-ring $\Lambda$,
let $\mc{F}_\Lambda$ denote the class of all division rings with
involution which are $*$-$\Lambda$-algebras.
Unless stated otherwise, pre-hermitian  spaces $V_F$, $F \in \mc{F}_\Lambda$, are dealt with as
$2$-sorted structures with sorts $V$ and $F$.
That is, $V$ carries the structure of an abelian group,
$F$ that of a ring with involution $\nu$ and with an unary operation
$\lambda  \mapsto \zeta \lambda $
associated to each  $\zeta \in \Lambda$;
moreover one has the maps $V\times F \to V$ with 
 $(v,\lambda)\mapsto v\lambda$
and $V\times V \to F$ with $(v,w)\mapsto \sk{v}{w}$.

In general, a  \emph{similarity type} for $n$-sorted   
algebraic structures will have a list $S_1,\ldots ,S_n$
of names for sorts, a list of typed operation symbols
$f:S_{j_1}\times \ldots \times S_{j_{k_f}} \to S_{j_{k_f+1}}$,
and a list of typed relation symbols 
$R\subseteq S_{j_1}\times \ldots \times S_{j_{k_R}}$.
A \emph{structure} $A$ of this type
is a family $S_1^A,\ldots ,S_n^A$ of sets 
together with a map $f^A:S_{j_1}^A\times \ldots \times S_{j_{k_f}}^A\to S_{j_{k_f+1}}^A$
for each operation symbol $f$ and with a    set $R^A \subseteq 
S^A_{j_1}\times \ldots \times S^A_{j_{k_R}}$
for each relation symbol $R$.

Recall the notion of \emph{ultrafilter} 
on a set $I$: a set $\mc{U}$ of subsets of
$I$ which is maximal with the following properties:
  $\emptyset \not\in \mc{U}$, $U\cap V \in U$
for all $U,V  \in \mc{U}$, $V \in \mc{U}$ for all
$U\in \mc{U}$; in particular,
either  $U \in \mc{U}$ or $I\setminus U \in \mc{U}$
for any $U\subseteq I$.
Given  $n$-sorted 
 structures $A_i$, $i \in I$,
of fixed sorted similarity type 
  and any ultrafilter
$\mc{U}$ on $I$, for each sort $S_j$ one has an equivalence
relation $\equiv_{S_j}$ on the
 direct product $\prod_{i \in I} S_j^{A_i}$ of sets
 such that 
\[(a_i\mid i \in I)\equiv_{S_j} (b_i\mid i \in I)
\;\Leftrightarrow \;\exists U \in \mc{U}\,\forall i \in U\;
a_i=b_i.\]
The equivalence classes $[a_i\mid i \in I]_{S_j}$
are the elements of the ultraproduct 
$S_j^A= \prod_{i\in I}S_j^{A_i}\slash\mathcal{U}$ of sort $S_j$.
Now, one defines
 the relations and operations  of
the \emph{ultraproduct} $A=\prod_{i\in I}A_i\slash\mathcal{U}$ 
as follows:
\[ ([a_i^{{j_1}}\mid i \in I]_{S_{j_1}}, \ldots , 
[a_i^{{j_{k_f}}}\mid i \in I]_{S_{j_{k_R}}}) \in R^A 
\Leftrightarrow
\exists U \in \mc{U}\,\forall i \in U\;
(a_i^{j_1}, \ldots ,a_i^{j_{k_R}} ) \in R^{A_i}\]
for each relation symbol $R$ (with type as above)
\[ f^A([a_i^{{j_1}}\mid i \in I]_{S_{j_1}}, \ldots , 
[a_i^{{j_{k_f}}}\mid i \in I]_{S_{j_{k_f}}})= [
f^{A_i}(a_i^{{j_1}}, \ldots, a_i^{{j_{k_f}}}) \mid i \in I]_{S_{j_{k_f+1}}}
\]
for each operation symbol $f$ (with type as above).
The operations  are well defined as one easily sees.

\begin{fact}\lab{los}
Let $\Phi(x_1,\ldots,x_m)$ 
any formula in the first order language
associated to the given similarity type
with free variables $x_k$ of sort $S_{j_k}$.
For an ultraproduct as above one has
$\Phi([a_i^{{j_1}}\mid i \in I]_{S_{j_1}}, \ldots , 
[a_i^{{j_m}}\mid i \in I]_{S_{j_{k_m}}})$ valid in $A$ 
if and only if there is $ U \in \mc{U}$
such that
$\Phi(a_i^{j_1}, \ldots ,a_i^{j_{k_m}})$ is valid in $A_i$
 for all  $i \in U$. 
\end{fact}

This is a variant of the well known
Theorem of \L o\'{s} \cite[Theorem 9.5.1]{hodges}. To derive it
from the $1$-sorted case, 
multi-sorted structures may be conceived as
$1$-sorted relational structures, assuming
sorts to be   pairwise disjoint
and captured by   unary  predicates. 
The following is an immediate consequence.

\begin{lemma}\lab{ultra0}
Let $\mathcal{U}$ be any ultrafilter over a set $I$.
Let also $(V_i)_{F_i}$ be a pre-hermitian space over $F_i\in\mathcal{F}_\Lambda$ for all $i\in I$.
Then $F=\prod_{i\in I}F_i\slash\mathcal{U}\in\mathcal{F}_\Lambda$ and
$V=\prod_{i\in I}{V_i}\slash\mathcal{U}$ is a pre-hermitian space over $F$.
Here, for $v=[v_i \mid i \in I]$ and $w=[w_i \mid i \in I]$
in the abelian group $V$ and 
$\lambda=[\lambda_i\mid i \in I]$ in $F$ one has
\[ v\lambda =[v_i\lambda_i\mid i \in I],\quad 
\sk{v}{w}=[\sk{v_i}{w_i}_i\mid i \in I]\]
where $\sk{v_i}{w_i}_i \in F_i$ is the value
under the scalar product from   $(V_i)_{F_i}$.
\end{lemma}

Recall that a representation of 
a $\ast$-$\Lambda$-algebra $R$ within a pre-hermitian space $V_F$ is a 
$\ast$-$\Lambda$-algebra homomorphism $\varepsilon\colon R\to\LEnd(V_F)$.
It is convenient to consider representations as unitary $R$-$F$-bimodules.
More precisely, one has an action
$(r,v)\mapsto rv=\varepsilon(r)(v)$ of $R$ on the left and an action
$(v,\lambda)\mapsto v\lambda$ of $F$ on the right satisfying the laws of unitary left and right modules
and such that
\[
(\lambda r)v=(rv)\lambda=r(v\lambda)\quad\text{for all}\ v\in V,\ r\in R,\ \lambda\in\Lambda,
\]
where $v\lambda=v(\lambda 1_F)$. Moreover,
\[
\sk{rx}{y}=\sk{x}{r^\ast y}\quad\text{for all}\ r\in R,\ x,y\in V
\]
\[
(\lambda r)^\ast v=(\lambda^\ast r^\ast)v=(r^\ast v)\lambda^\ast
\quad\text{for all}\ v\in V,\ r\in R,\ \lambda\in\Lambda.
\]
We denote a representation of $R$ in $V_F$ by $_RV_F$.
The $R$-$F$-bimodule $_RV_F$
with scalar product
 will be considered as a
$3$-sorted structure with sorts $V$, $R$, and $F$;
the $\ast$-$\Lambda$-algebras 
$R$ and  $F$ are considered as $1$-sorted structures,
where $\lambda\in\Lambda$ serves to denote the unary operation $x\mapsto\lambda x$.
Our main concern will be faithful representations; that is,
representations $_RV_F$ such that $rv=0$ for all $v\in V$ if only if $r=0$.
Observe that a regular algebra $R$ is $\ast$-regular, if it admits a
faithful representation in an anisotropic space.

\xc{
\begin{fact}\lab{axiom}
Given a recursive commutative $\ast$-ring $\Lambda$ with unit,
there is a recursive first order axiomatization of the class of all $3$-sorted structures
$_RV_F$ where $R$ and  $F$ are \las, $V_F$ is a \red{pre-}hermitian space,
and $\varepsilon(r)(v)=w$ iff $rv=w$ defines a faithful representation of $R$ in $V_F$.
\end{fact}}

\noindent
The following is as obvious as crucial:
A representation of a MIL $\varepsilon\colon L\to\mathbb{L}(V_F)$
can be viewed as a $3$-sorted structure with sorts $L$, $V$, and $F$ and with
the map $\varepsilon$ being captured by the binary relation
(cf. \cite{mal,mal1,Sch} for this method)
\[
\{(a,v)\mid v\in\varepsilon(a)\}\subseteq L\times V,
\]
which we denote by $\varepsilon$ again.

\xc{
\begin{fact}\lab{axiom2}
There is a recursive first order axiomatization of the class of all
$3$-sorted structures associated with $[$faithful$]$ representations of MILs in pre-hermitian spaces.
\end{fact}}

\begin{lemma}\lab{ultra}
Under the hypotheses of Lemma \ref{ultra0} one has the following.
\begin{enumerate}
\item
If $L_i$ is a MIL and $(L_i,V_i,F_i;\varepsilon_i)$ is a faithful representation for all $i\in I$,
then the associated ultraproduct $(L,V_F,F;\varepsilon)$ is a faithful representation of
$L=\prod_{i\in I}L_i\slash\mathcal{U}$.
\item
If  $R_i$ is a  $\ast$-$\Lambda$-algebra and ${_{R_i}}(V_i){_{F_i}}$  a faithful representation for all $i\in I$,
then the associated ultraproduct $_RV_F$ is a faithful representation of
$R=\prod_{i\in I}R_i\slash\mathcal{U}$.
\item
Let $U$ be an $n$-dimensional subspace of $V_F$, $n<\omega$.
Then there are $J\in\mathcal{U}$ and $n$-dimensional subspaces $U_i$ 
of $(V_i)_{F_i}$, $i\in J$,
such that $U\cong\prod_{i\in J}{U_i}\slash\mathcal{U}_J$, where $\mathcal{U}_J=\{X\in\mathcal{U}\mid X\subseteq J\}$, and
\[
\mathbb{L}(U_F)\cong\prod_{i\in J}\mathbb{L}((U_i)_{F_i})\slash\mathcal{U}_J,\quad\LEnd(U_F)\cong\prod_{i\in J}\LEnd((U_i)_{F_i})\slash\mathcal{U}_J
\]
\end{enumerate}
\end{lemma}

\begin{proof} 
Statements (i) and (ii) are immediate
 by \Fact \ref{los} and  the observation that both 
types of $3$-sorted structures can be characterized by 
first order axioms.
In (iii) observe that for a fixed positive integer $n$, there is a 
first order formula in the two sorted language
for vector spaces)
 expressing that a set of vectors $\{v_1,\ldots,v_n\}$ is independent [is a basis],
as well as a  first order formula  expressing that a vector $v$ is in the span of $\{v_1,\ldots,v_n\}$.
Thus, by the \L o\'{s} Theorem, a basis of $U$ determines $J$ and bases of spaces $U_i$, $i\in J$.
Now, apply (i) to lattices $L_i=\mathbb{L}((U_i)_{F_i})$, $i\in J$,
to get an embedding of $\prod_{i \in J}L_i\slash\mathcal{U}_J$ into $\mathbb{L}(U_F)$.
Surjectivity of this embedding is granted by the sentence stating that for any $v_1$, \ldots, $v_n$,
there is $a$ such that $v\in\varepsilon(a)$ if and only if $v$ in the span of $v_1$, \ldots, $v_n$.
Similarly, we apply (ii) in the ring case and use the sentence stating
that for any basis $v_1$, \ldots, $v_n$ and any $w_1$, \ldots, $w_n$,
there is $r$ such that $rv_i=w_i$ for all $i\in\{1,\ldots,n\}$.
\end{proof}

\noindent
Inheritance of existence of representations under homomorphic images has been dealt with, in different contexts,
in \cite{proat,he4} for CMILs and by Micol in \cite{Flo} for $\ast$-rings.
Apparently, this needs saturation properties of ultrapowers.
Considering a fixed $1$-sorted algebraic structure $A$, add a new constant symbol $\underline{a}$,
called a \emph{parameter}, for each $a\in A$.
In what follows, $\Sigma(x_1,\ldots,x_n)$ is a set of formulas with free
variables $x_1$, \ldots, $x_n$ in this extended language.
Given an embedding $h\colon A\to B$,
we call $B$ \emph{modestly saturated} 
[\emph{$\omega$-saturated}] \emph{over} $A$ \emph{via} $h$,
if for any $n<\omega$ and for any set of formulas $\Sigma(x_1,\ldots,x_n)$, with parameters from $A$
[and finitely many parameters from $B$, respectively],
which is finitely realized in $A$ [in $B$, respectively] is realized in $B$
(where $\underline{a}$ is interpreted as $\underline{a}^B=h(a)$).
The following is a particular case of \cite[Corollary 4.3.14]{CK}.

\begin{fact}\lab{sat}
Every $1$-sorted algebraic structure $A$ admits an elementary embedding $h$
into some structure $B$ which is $[\omega$-$]$saturated over $A$ via $h$.
One can choose $B$ to be an ultrapower of $A$ and $h$ to be the canonical embedding.
Identifying $a$ with $h(a)$, one may assume $B$ to be an elementary extension of $A$.
The analogous result holds for multi-sorted algebraic
structures.
\end{fact}

\begin{theorem}\lab{homlat}
Let a CMIL $L$ $[$a $\ast$-$\Lambda$-algebra $R]$
have a faithful representation within a pre-hermitian space $V_F$.
There is an ultrapower $\hat{V}_{\hat{F}}$ of $V_F$
such that any homomorphic image of $L$ $[$such that for any regular ideal $I=I^\ast$,
the algebra $R{\slash}I]$
admits a faithful representation within $(U{\slash}\rad U)_{\hat{F}}$,
where $U=U^{\perp\perp}$ is a subspace of $\hat{V}_{\hat{F}}$.
\end{theorem}

\begin{proof}
For a $\ast$-$\Lambda$-algebra  $R$ we use the same idea as in the proof of \cite[Proposition 25]{PartI}.
Though here, the scalar product induced on $U$, as defined below, might be degenerated.
According to \Fact \ref{sat}, there is an
ultrapower $_{\hat{R}}\hat{V}_{\hat{F}}$
of the faithful representation $_RV_F$ which is
modestly saturated over $_RV_F$ via the canonical embedding.
Then $\hat{V}$ is an $R$-module via the canonical embedding of $R$
into $\hat{R}$ and the set
\[
U=\{v\in\hat{V}\mid av=0\ \text{for all}\ a\in I\}=\bigcap_{a\in I}(a^\ast\hat{V})^\perp
\]
is a closed subspace of $\hat{V}_{\hat{F}}$ and a left {$(R/I)$}-module.
Moreover as $I=I^\ast$, one has
\[
\langle(r+I)v\mid w\rangle=\langle v\mid (r^\ast+I)w\rangle\ \text{for all}\ v,w \in U\ \text{and all}\ r\in R.
\]
We observe that $U^\perp$ is also an $(R{\slash}I)$-module. Indeed,
if $v\in U^\perp$ then
\[
\langle(r+I)v\mid u\rangle=\langle v\mid(r^\ast+I)u\rangle=0\ \text{for all}\ u\in U.
\]
Thus with $W=\rad U$, one obtains an $(R{\slash}I)$-$\hat{F}$-bimodule $U{\slash}W$, where
\[
(r+I)(v+W)=rv+W\ \text{for all}\ r\in R\ \text{and all}\ v\in U,
\]
which is also a subquotient of $V_F$.

We show that $_{R{\slash}I}(U{\slash}W)_{\hat{F}}$ is a faithful representation of $R{\slash}I$;
that is, for any $a\in R{\setminus}I$, there has to be $u\in U$ such that $au\notin W$.
It suffices to show that for any $a\in R{\setminus}I$,
there are $u$, $v\in U$ such that $\langle au\mid v\rangle\neq 0$.
Since $u\in U$ means $bu=0$ for all $b\in I$, we have to show that the set
\[
\Sigma(x,y)=\{\langle\underline{a}x\mid y\rangle
\neq 0\}\cup\{\underline{b}x=0=\underline{b}y\mid b\in I\}
\]
of formulas with parameters from $\{a\}\cup I$ and variables $x$, $y$ of type $V$
is satisfiable in $_{\hat{R}}\hat{V}_{\hat{F}}$.
Due to saturation, it suffices to show that for any $b_1$, \ldots, $b_{n}\in I$,
there are $u$, $v\in V$ such that $\langle au\mid v\rangle\neq 0$ and $b_iu=b_iv=0$ for all $i\in\{1,\ldots,n\}$.
In view of \Fact \ref{reri}(iv) and regularity of $I$,
there is an idempotent $e\in I$ such that $b_ie=b_i$ for all $i\in\{1,\ldots,n\}$;
in particular $b_iu=b_iv=0$ whenever $eu=ev=0$.
Thus it suffices to show that there are $u$, $v\in V$ such that $eu=ev=0$
but $\langle au\mid v\rangle\neq 0$.

Assume the contrary; namely, let
$eu=ev=0$ imply $\langle au\mid v\rangle=0$ for all $u$, $v\in V$.
For arbitrary $u^\prime$, $v^\prime\in V$, let $u=(1-e)u^\prime$ and $v=(1-e)v^\prime$.
As $eu=ev=0$, we get by our assumption that $\langle(1-e^\ast)au\mid v^\prime\rangle=\langle au\mid v\rangle=0$.
This holds for all $v^\prime\in V$,
whence $(1-e^\ast)au=0$ since $V_F$ is non-degenerated.
Thus $(1-e^\ast)a(1-e)u^\prime=0$ for all $u^\prime\in V$,
whence $(1-e^\ast)a(1-e)=0$, as $_RV_F$ is a faithful representation.
But then $a=e^\ast a+ae-e^\ast ae\in I$, a contradiction.

\medskip
In the case of CMILs,
given a representation $\varepsilon\colon L\to\mathbb{L}(V_F)$,
let $G=\mathbb{G}(V_F)$ and let $\pi(v)=vF$ for $v\in V$.
We consider the $4$-sorted structure $(L,V,F,G;\varepsilon,\pi)$.
According to \Fact \ref{sat}, there is an ultrapower $(\hat{L},\hat{V},\hat{F},\hat{G};\hat{\varepsilon},\hat{\pi})$
of $(L,V,F,G;\varepsilon,\pi)$ which is modestly saturated
over $(L,V,F,G;\varepsilon)$ via the canonical embedding.
By Lemma \ref{ultra}(i), $(\hat{L},\hat{V},\hat{F};\hat{\varepsilon})$ is a faithful representation.
In view of \Fact \ref{og}(ii), $\hat{G}\cong\mathbb{G}(\hat{V}_{\hat{F}})$ via $\hat{\pi}$;
and $\hat{\rho}\colon W\mapsto\{v\in\hat{V}\mid\hat{\pi}(v)\in W\}$
defines an isomorphism from $\mathbb{L}(\hat{G})$ onto $\mathbb{L}(\hat{V}_{\hat{F}})$ by \Fact \ref{og}(iii).

Now, let $\theta$ be a congruence of the Galois lattice $L$.
According to the proof of \cite[Theorem 13.1]{he4}, there is
a faithful representation $\eta\colon L{\slash}\theta\to\mathbb{L}(W{\slash}W^\prime)$
in a subquotient $W{\slash}W^\prime$ of $\hat{G}$,
where the subspace $W$ is closed and $W^\prime=W\cap W^\perp$.
Then $\hat{\rho}(W){\slash}\hat{\rho}(W^\prime)$ is a subquotient of $\hat{V}_{\hat{F}}$,
$\hat{\rho}(W)$ is a closed subspace of $\hat{V}$, and
$\hat{\rho}\eta$ is a faithful representation of
$L{\slash}\theta$ in $\hat{\rho}(W){\slash}\hat{\rho}(W^\prime)$ by \Fact \ref{ogrep}.
The proof is complete.
\end{proof}

\begin{corollary}\lab{C:83}
Let a MOL $L$ have a faithful representation within a pre-hermitian space $V_F$.
There is an ultrapower $\hat{V}_{\hat{F}}$ of $V_F$
such that any homomorphic image of $L$
admits a faithful representation within a pre-hermitian closed subspace $U_{\hat{F}}$ of $\hat{V}_{\hat{F}}$.
\end{corollary}

\begin{proof}
According to the proof of \cite[Theorem 13.1]{he4} and the proof of Theorem \ref{homlat},
there is an ultrapower $\hat{V}_{\hat{F}}$ of $V_F$
such that any homomorphic image of $L$ admits a faithful representation within a subquotient
$W{\slash}W^\prime$ of the orthogonal geometry $\mathbb{G}(\hat{V}_{\hat{F}})$.
As $L$ is a MOL, according to the definition of $W^\prime$
(given in \cite[page 355]{he4} and denoted by $U$ there),
one has $W^\prime=\varnothing$. Hence in the proof of Theorem \ref{homlat}, $\rad U=\hat{\rho}(W^\prime)=0$.
\end{proof}

\section{Classes of structures}

We consider classes $\mc{C}$ of $\ast$-$\Lambda$-algebras on one side, Galois lattices on the other. With the familiar concepts,
by ${\sf H}(\mathcal{C})$, ${\sf S}(\mathcal{C})$,
${\sf P}(\mathcal{C})$, ${\sf P}_{\sf s}(\mathcal{C})$,
${\sf P}_\omega(\mathcal{C})$, and ${\sf P}_{\sf u}(\mathcal{C})$,
we denote the class of all homomorphic images,  subalgebras,
direct products, subdirect products, direct products of finitely many factors,
and ultraproducts of members of $\mathcal{C}$, respectively,
allowing isomorphic copies in all cases.
 Of course, all fundamental
operations have to be taken care of. In particular, in the case of
$\ast$-$\Lambda$-algebras also  the unit $1$, the additive inverse, and the ``scalars''
$\lambda \in  \Lambda$, that is, ``subalgebra'' means
$*$-subring and and $\Lambda$-subalgebra. 
 In the case of Galois lattices, also
 the bounds $0,1$ and the operation $x \mapsto x'$
are to be preserved, that is,
``subalgebra'' means Galois sublattice.
In terms of Universal Algebra we have 
classes of algebraic structures of given ``similarity type'' or ``signature''
and the associated class operators, cf. 
 \cite[Chapter II]{bs} and   \cite[Chapter I]{gorb}, also   \cite{mal2}.

A class $\mathcal{C}$ of algebraic structures of the same type is a \emph{universal class}
if it is closed under ${\sf S}$ and ${\sf P}_{\sf u}$;
a \emph{positive universal class} (shortly a \emph{semivariety}),
if it is closed also under ${\sf H}$;
a \emph{variety} if, in addition, it is closed under ${\sf P}$.
Let ${\sf W}(\mc{C})$ and ${\sf  V}(\mc{C})$
denote the smallest semivariety and smallest variety containing $\mc{C}$.
The following statement is well known and easily verified, cf. Theorem \ref{exvar} in Appendix A.

\begin{fact}\lab{semi}
A class $\mathcal{K}$ is universal $[$a semivariety, a variety$]$ if and only if it can be defined by
universal sentences $[$positive universal sentences, identities, respectively$]$. 
\end{fact}

Dealing with a class $\mathcal{C}$ of $\ast$-$\Lambda$-algebras or MILs, let
${\sf S}_\exists(\mathcal{C})$ [${\sf P}_{{\sf s}\exists}(\mathcal{C})$]
consist of all regular or complemented members of the class ${\sf S}(\mathcal{C})$
[of the class ${\sf P}_{{\sf s}}(\mathcal{C})$, respectively].
Call $\mathcal{C}$ an $\exists$-\emph{semivariety} if it is closed
under the operators
${\sf H}$, ${\sf S}_\exists$, ${\sf P}_{\sf u}$ and an $\exists$-\emph{variety}
if it is also closed under ${\sf P}$, cf. \cite{HS},
also \cite{ks} for an analogue within semigroup theory.
Let ${\sf W}_\exists(\mathcal{C})$ [${\sf V}_\exists(\mathcal{C})$] denote the least
$\exists$-semivariety [$\exists$-variety, respectively] which contains the class $\mathcal{C}$.

Recall that MIL also denotes the class of all MILs, similarly for CMIL and MOL.
Let $\mc{A}_ \Lambda$ denote the class of all
$\ast$-$\Lambda$-algebras, with the subclasses $\mc{R}_\Lambda$, $\mc{R}^*_\Lambda$,
and $\mc{F}_\Lambda$ consisting of its members which are regular, $*$-regular,
and divisions rings, respectively.

\begin{fact}\lab{hs1}
Let $\mathcal{C}\subseteq\mathcal{R}_{\Lambda}$ or $\mathcal{C}\subseteq\mathrm{C}\mathrm{M}\mathrm{I}\mathrm{L}$.
\begin{enumerate}
\item
${\sf O}{\sf S}_\exists(\mathcal{C})\subseteq{\sf S}_\exists{\sf O}(\mathcal{C})$
for any class operator ${\sf O}\in\{{\sf P}_{\sf u},{\sf P},{\sf P}_\omega\}$.
\item
$ {\sf S}_\exists{\sf H}(\mathcal{C})\subseteq{\sf H}{\sf S}_\exists(\mathcal{C})$.
\item
${\sf W}_\exists(\mathcal{C})={\sf H}{\sf S}_\exists{\sf P}_{\sf u}(\mathcal{C})$.
\item
${\sf V}_\exists(\mathcal{C})={\sf H}{\sf S}_\exists{\sf P}(\mathcal{C})=
{\sf H}{\sf S}_\exists{\sf P}_{\sf u}{\sf P}_\omega(\mathcal{C})={\sf P}_{{\sf s}\exists}{\sf W}_\exists(\mathcal{C})$.
\item
${\sf W}_\exists(\mathcal{C})$ and ${\sf V}_\exists(\mathcal{C})$ are axiomatic classes.
\item
$A\in{\sf W}_\exists(\mathcal{C})$
if $B\in{\sf W}_\exists(\mathcal{C})$ for all finitely generated $B\in{\sf S}_\exists(A)$.
\end{enumerate}
\end{fact}

\noindent
These statements are well known for arbitrary algebraic structures if the suffix $\exists$ is omitted.
For the proof of \Fact \ref{hs1}, we refer to the Appendix A.

\medskip
Dealing with pre-hermitian spaces, we primarily 
adhere to the $2$-sorted point of view as explained in Section \ref{sort}.
A ($2$-sorted)  \emph{embedding} $V'_{F'}\to V_F$ 
is given by a $*$-$\Lambda$-algebra  embedding $\alpha:F'\to F$
and an injective  $\alpha$-semilinear map
$\omega$ such that $\sk{\omega(v)}{\omega(w)}=\alpha(\sk{v}{w}')$
for all $v,w \in V'$.
 An embedding is an \emph{isomorphism} if both
$\alpha$ and $\omega$ are bijections.
$V'_{F'}$ is a ($2$-sorted) substructure of $V_F$
if it embeds into $V_F$ with $\alpha$ and $\omega$ being inclusion maps.
 In contrast, a \emph{subspace} of $V_F$
will always mean an $F$-linear subspace
with the induced scalar product; that is, here we follow the $1$-sorted view on the vector space $V_F$.

Let $\mathcal{S}$ be a class of pre-hermitian spaces $V_F$, where $F\in\mathcal{F}_\Lambda$
and $\Lambda$ is a fixed commutative $\ast$-ring.
In such a case, we also speak of a class of spaces \emph{over} $\Lambda$.
Introducing  operators for classes of spaces,
let ${\sf S}(\mathcal{S})$ and ${\sf P}_{\sf u}(\mathcal{S})$
denote the classes of all non-degenerate $2$-sorted substructures
and all ultraproducts of members of $\mathcal{S}$ respectively,
or spaces which are isomorphic to such.
In contrast to that, following the one-sorted view, let
${\sf S}_{1 {\sf f}}(\mathcal{S})$ [${\sf S}_{1{\sf q}}(\mathcal{S})$] 
denote the class of (isomorphic copies of) non-degenerate finite-dimensional subspaces 
[of all subquotients $U{\slash}\rad U$ such that $V_F\in\mathcal{S}$,
$U_F$ is a subspace of $V_F$, and  $U=U^{\perp\perp}$, respectively]
of members of $\mathcal{S}$. 
The next statement follows from \Facts \ref{du} and \ref{sq}.

\begin{lemma}\lab{L:S}
For any class $\mathcal{S}$ of spaces over $\Lambda$,
${\sf S}_{1 {\sf f}}(\mathcal{S})\subseteq{\sf S}_{1{\sf q}}(\mathcal{S})$ and 
${\sf S}_{1 {\sf f}}{\sf S}_{1{\sf q}}(\mathcal{S})={\sf S}_{1 {\sf f}}(\mathcal{S})$.
\end{lemma}

\noindent
Let also ${\sf I}_{\sf s}(\mathcal{S})$ denote the class of spaces which arise from $\mathcal{S}$ by scaling
and observe that ${\sf I}_{\sf s}{\sf O}(\mathcal{S})\subseteq{\sf O}{\sf I}_{\sf s}(\mathcal{S})$
for any of the  class operators introduced, above.
Call $\mathcal{S}$ a \emph{universal class},
if it is closed under ${\sf P}_{\sf u}$, $\sf S$, and ${\sf I}_{\sf s}$.
Observe that ${\sf S}{\sf P}_{\sf u}{\sf I}_{\sf s}(\mathcal{S})$ is the smallest
universal class containing a class $\mathcal{S}$.
Call $\mathcal{S}$ a \emph{semivariety} if it is closed under ${\sf P}_{\sf u}$ and ${\sf S}_{1 {\sf f}}$.
Of course, any universal class is a semivariety,
and the smallest semivariety containing a class $\mathcal{S}$ is 
contained in ${\sf SP}_{\sf u}(\mathcal{S})$.

\section{Reduction to finite dimension}
Importance of representations for the universal algebraic theory of CMILs and regular $\ast$-rings
derives from the following

\begin{theorem}\lab{findim}
Let $V_F$ be a pre-hermitian space
and let $L\in\mathrm{MIL}$ $[R\in\mathcal{A}_\Lambda
]$
have a faithful representation within $V_F$.
Then $L\in{\sf W}\bigl(\mathbb{L}(U_F)\mid U\in\mathbb{O}(V_F)\bigr)$
$[R\in{\sf W}\bigl(\LEnd(U_F)\mid U\in\mathbb{O}(V_F)\bigr)$, respectively$]$.
If $L\in\mathrm{CMIL}$ $[R\in\mathcal{R}_{\Lambda}]$,
then $L\in{\sf W}_\exists\bigl(\mathbb{L}(U_F)\mid
U\in\mathbb{O}(V_F)\bigr)$
$[R\in{\sf W}_\exists\bigl(\LEnd(U_F)\mid U\in\mathbb{O}(V_F)\bigr)$, respectively$]$.
\end{theorem}

\begin{proof}
We may assume that $\dim V_F\geqslant\omega$.
In view of \Fact \ref{atex},  $L$ embeds into  an atomic
subalgebra $M$ of $\mathbb{L}(V_F)$ such that $M_{\sf f}=\mathbb{L}(V_F)_{\sf f}$
and $M$ may be chosen a CMIL if $L$ is such.
\Fact \ref{og}(i)--(ii) yields that
$M_{\sf f}$ is
a CMIL and the directed union
of its subalgebras $[0,U]\cup[U^\perp,V]$, $U\in\mathbb{O}(V_F)$, each of which is in ${\sf S}_\exists(W_F)$
for some $W\in \mathbb{O}(V_F)$.
Since any directed union of algebraic structures $A_i$ 
embeds into an ultraproduct of the $A_i$
(cf  \cite[Theorem 1.2.12(1)]{gorb}), one gets
\[
M_{\sf f}\in{\sf S}_\exists{\sf P_u}\bigl(\mathbb{L}(U_F)\mid U\in\mathbb{O}(V_F)\bigr).
\]
Finally, the proof of \cite[Theorem 16.3]{he4} yields $M\in{\sf W}(M_{\sf f})$
and $M\in{\sf W}_\exists(M_{\sf f})$ if $M$ is complemented.
The claim about $L$ follows, immediately.

Dealing with an algebra $R\in\mathcal{A}_\Lambda$, we first show that
 $\hat{J}(V_F)\in{\sf W}_\exists\bigl(\LEnd(U_F)\mid U\in\mathbb{O}(V_F)\bigr)$.
By \Fact \ref{hs1}(vi), it suffices to prove this inclusion for finitely generated algebras
$B\in{\sf S}_\exists\bigl(\hat{J}(V_F)\bigr)$.
By \Fact \ref{endreg}(iv), we may assume that $B$ is of the form
$\{\varepsilon_U\varphi\pi_U +\lambda\id_V\mid\varphi\in\LEnd(V_F),\ \lambda\in F\}$
for some $U\in\mathbb{O}(V_F)$.
Thus $B\cong\LEnd(U_F)\times F$
and the latter embeds into 
$\LEnd(U_F)\times \LEnd((U^\perp\cap W)_F)$
which in turn into 
$\LEnd(W_F)$ for $U\subset W\in\mathbb{O}(V_F)$, cf. \Facts \ref{du} and \ref{db}.

In view of \Fact \ref{atex2}, we may assume
that $R$ is a subalgebra of $\LEnd(V_F)$ containing $A=\hat{J}(V_F)$.
Let $J=J(V_F)$ and let $J_0$ denote the set of projections in $J$.
By \Fact \ref{sat}, there is an ultrapower $(_{\hat{R}}\hat{V}_{\hat{F}};\hat{A})$
of $(_RV_F;A)$ which is $\omega$-saturated over $(_RV_F;A)$.
We may assume that $R$ is a subalgebra of $\hat{R}$ and $\hat{A}$ is an ultrapower of $A$;
in particular, $\hat{A}\in{\sf W}_\exists\bigl(\LEnd(U_F)\mid U\in\mathbb{O}(V_F)\bigr)$.
For $a\in\hat{A}$ and $r\in R$, we put
\[
a\sim r,\quad\text{if}\quad 
ae=re\ \text{and}\ a^\ast e=r^\ast e\ \text{for all}\ e\in J_0.
\]

\setcounter{claim}{0}
\begin{claim}\lab{c-1}
For any $a\in\hat{A}$ and any $r$, $s\in R$, $a\sim r$ and $a\sim s$ imply $r=s$.
\end{claim}

\begin{scproof}
For any $U\in\mathbb{O}(V_F)$, we have $\pi_U\in J_0$, whence $r\pi_U=a\pi_U=s\pi_U$.
Considering $r$ and $s$ as endomorphisms of $V_F$, we get that
they coincide on any $U\in\mathbb{O}(V_F)$, whence they coincide on $V_F$ by \Fact \ref{du}.
\end{scproof}

\begin{claim}\lab{c-2}
$S=\{a \in \hat{A}\mid a\sim r\ \text{for some}\ r\in R\}$ is a subalgebra of $\hat{A}$ and the map
\[
g\colon\hat{A}\to R,\quad g\colon a\mapsto r,\ \text{where}\ a\sim r
\]
is a homomorphism.
\end{claim}

\begin{scproof}
It follows from Claim \ref{c-1} that $g$ is well-defined.
Let $a$, $b\in\hat{A}$ and $r$, $s\in R$ be such that $a\sim r$ and $b\sim s$.
Then, obviously, $a+b\sim r+s$, $\lambda a\sim\lambda r$ for any $\lambda\in\Lambda$,
and $a^\ast\sim r^\ast$. Let $e\in J_0$, then $be\in J$.
By \Fact \ref{endreg}(iv), there is $f\in J_0$ such that $fbe=be$.
Therefore, we get $abe=afbe=rfbe=rbe=rse$, whence $ab\sim rs$.

Obviously, $0_{\hat{V}}$, $\id_{\hat{V}}\in\hat{A}$.
For any $U\in\mathbb{O}(V_F)$ we have $\pi_U\in J_0$.
Therefore, $0_{\hat{V}}\pi_U=0_U$ and $\id_{\hat{V}}\pi_U=\pi_U$ imply in view of \Fact \ref{du}
that $0_{\hat{V}}\sim 0_R$ and $\id_{\hat{V}}\sim 1_R$.
\end{scproof}

\begin{claim}\lab{c-3}
The homomorphism $g$ is surjective.
\end{claim}

\begin{scproof}
Surjectivity of $g$ is shown via the saturation property.
Given $r\in R$, consider a finite set $E\subseteq J_0$.
According to \Fact \ref{endreg}(iv), there is $e\in J_0$ such that $ef=f$ for all $f\in E$
and $er^\ast f=r^\ast f$ for all $f\in E$.
Take $a=re$ and observe that $af=ref=rf$ and $a^\ast f =er^\ast f=r^\ast f$ for all $f\in E$.
Thus the set of formulas
\[
\Sigma(x)=\bigl\{[xe=re]\ \&\ [x^\ast e=r^\ast e]\mid e\in J_0\bigr\}
\]
with a free variable $x$ of type $A$ is finitely realized in $(_RV_F;A)$.
As $(_{\hat{R}}\hat{V}_{\hat{F}};\hat{A})$ is $\omega$-saturated over $(_RV_F;A)$,
we get that there is $a\in\hat{A}$ with $a\sim r$.
\end{scproof}

\begin{claim}\lab{c-4}
If $R$ is regular, then $S$ is also regular.
\end{claim}

\begin{scproof}
In view of \Fact \ref{reri}(ii), it suffices to prove that $\kerr g=\{a\in S\mid a\sim 0\}$ is regular.
Observe that $a\sim 0$ means that
$ae=0=a^\ast e$ for any $e\in J_0$, equivalently $(1-e)a=a=a(1-e)$.
Again, let $E\subseteq J_0$ be finite.
By \Fact \ref{endreg}(iv), there is $e\in J_0$ such that $ef=f$ for any $f\in E$.
The ring $A$ is regular by \Facts \ref{endreg}(i) and \ref{atex2}, whence $\hat{A}$ is also regular.
Therefore, the ring $(1-e)\hat{A}(1-e)$ is regular by \cite[2.4]{ber4}. Thus
there is $b\in\hat{A}$ such that $aba=a$ and $(1-e)b=b=b(1-e)$; in particular, $be=0=eb$
whence $b^\ast e=0$.
This implies that $bf=bef=0$ and $b^\ast f=b^\ast ef=0$ for all $f\in E$.
Therefore, the set of formulas
\[
\Sigma(x)=\{axa=a\}\cup\bigl\{[xe=0]\ \&\ [x^\ast e=0]\mid e\in J_0\bigr\}
\]
with a variable $x$ of type $A$ is finitely realized in $(_{\hat{R}}\hat{V}_{\hat{F}};\hat{A})$.
Thus $\Sigma(x)$ is realized in $(_{\hat{R}}\hat{V}_{\hat{F}};\hat{A})$,
and we obtain $b\in\hat{A}$ such that $aba=a$ and $b\sim 0$; that is, $b\in\kerr g$.
\end{scproof}

\noindent
The desired statements concerning $\ast$-$\Lambda$-algebras follow from Claims \ref{c-2}-\ref{c-4}
and Theorem \ref{exvar}(iii).
\end{proof}

\begin{remark}
The statements of Theorem \ref{findim} concerning $\ast$-$\Lambda$-algebras
were proved in case of representability in inner product spaces in \cite[Theorem 16]{PartI}.
Requiring semivariety generation only, a more direct approach is possible.
For $R\in\mathcal{A}_\Lambda$,
one chooses in the proof of \cite[Theorem 16]{PartI} $I=\mathbb{O}(V_F)$.
By \Fact \ref{du}, any finite-dimensional subspace of $V_F$ is contained in some $U\in I$.
Moreover, with the induced scalar product, $U_F$ is a pre-hermitian space.
A similar approach works for MILs.
\end{remark}

\xc{
\begin{remark}
Here, we provide an alternative proof for the statements of Theorem \ref{findim}
concerning MILs which does not refer to orthogeometries.
We may assume that the representation $\varepsilon$ of a MIL $L$ is the identity map.
Let $I$ consist of all finite-dimensional subspaces of $V_F$.
For $U\in I$, we set $U^+=\{W\in I\mid U\subseteq W\}$.
Notice that $U_1^+\cap U_2^+=(U_1+U_2)^+$.
Thus there is some ultrafilter $\mathcal{U}$ on $I$ such that $U^+\in\mathcal{U}$ for all $U\in I$.
Let $L_U=\mathbb{L}(U_F{\slash}\rad U)$ with involution $X\mapsto X^\perp\cap U$
and observe that $S_U\in L_U$ iff $S_U\in\mathbb{L}(V_F)$ and $\rad U\subset S_U\subseteq U$.
Moreover, with $U=W\oplus\rad U$ one has by \Fact \ref{sq} that
$U_F{\slash}\rad U$ is isomorphic to $W_F\in\mathbb{O}(V_F)$.
Form the direct product $M=\prod_{U\in I}L_U$
and the ultraproduct $\hat{M}=\prod_{U\in I}L_U\slash\mathcal{U}$.
The elements of $M$ and $\hat{M}$ will be denoted as $S=(S_U\mid U\in I)$ and $[S]$ respectively.
Moreover, $S^\perp=(S_U^\perp\cap U\mid U\in I)$.

We relate $M$ with $L$.
Given $S\in M$, $X\in L$, $U_0$, $U_1\in I$, and $J\in\mathcal{U}$,
we write
\[
[U_0,U_1,J]\colon S\rightarrow X,\quad\text{if}\ J\subseteq U_0^+\cap U_1^+\ 
\text{and}\ U_0\subseteq S_U\subseteq U_1^\perp\ \text{for all}\ U\in J.
\]
We put also
\begin{align*}
X\rightarrow U\quad&\text{if for all}\ U_0,\ U_1\in I\ \text{with}\ U_0\subseteq X\subseteq U_1^\perp,\\
&\text{there is}\ J\in\mathcal{U}\ \text{such that}\ [U_0,U_1,J]\colon S\rightarrow X.
\end{align*}

\setcounter{claim}{0}
\begin{claim}\lab{c-5}
For any $S\in M$, there is at most one $X\in L$ with $S\rightarrow X$.
\end{claim}

\begin{scproof}
Let $S\rightarrow X$ and $S\rightarrow Y$ for some $X$, $Y\in L$.
Fix a vector $v\in X$ and put $U_0=vF$, $U_1=\{0\}$.
By the hypothesis, there is $J\in\mathcal{U}$ such that $[U_0,U_1,J]\colon S\rightarrow X$.
Now, consider any $U_3\in I$ with $Y\subseteq U_3^\perp$.
Again, there if $K\in\mathcal{U}$ such that $[U_1,U_3,K]\colon S\rightarrow Y$.
As $J\cap K\in\mathcal{U}$, there is $U\in J\cap K$.
Then $v\in S_U\subseteq U_3^\perp$.
Obviously, $Y^\perp=\Sigma\{W\in I\mid W\subseteq Y^\perp\}$.
By Lemma \ref{MOL}(i), $Y$ is closed, whence
$Y=\bigcap\{W^\perp\mid W\in I,\ Y\subseteq W^\perp\}$.
Therefore, $v\in Y$ and thus $X\subseteq Y$. Symmetrically, we obtain that $Y\subseteq X$.
\end{scproof}

\begin{claim}\lab{c-6}
For $S\in M$ and $X\in L$, if $S\rightarrow X$ then $S^\perp\rightarrow X^\perp$.
\end{claim}

\begin{scproof}
Let $U_0$, $U_1 \in I$ be such that $U_0\subseteq X^\perp\subseteq U_1^\perp$.
Then $U_1\subseteq X\subseteq U_0^\perp$.
By assumption, there is $J\in\mathcal{U}$ such that $[U_1,U_0,J]\colon S\rightarrow X$.
Let $U\in J$; in particular, we have $U_0\subseteq U$.
This means that $U_1\subseteq S_U\subseteq U_0^\perp$, whence
$U_0\subseteq S_U^\perp\cap U\subseteq U_1^\perp$ and thus
and $[U_0,U_1,J]\colon S^\perp\rightarrow X^\perp$.
\end{scproof}

\begin{claim}\lab{c-7}
If $S\rightarrow X$ and $T\rightarrow Y$ then $S+T\rightarrow X+Y$
for $S$, $T\in M$ and $X$, $Y\in L$.
\end{claim}

\begin{scproof}
Let $U_0$, $U_1\in I$ be such that $U_0\subseteq X+Y\subseteq U_1^\perp$.
Let $\dim U_0=n<\omega$ and let $U_0$ be spanned by vectors $v_1$, \ldots, $v_n$.
For any $i\in\{1,\ldots,n\}$, we have $v_i=x_i+y_i$ with $x_i\in X$ and $y_i\in Y$.
Let $U_2$ be the subspace of $V_F$ spanned by vectors $x_1$, \ldots, $x_n$, and let
$U_3$ be the subspace of $V_F$ spanned by vectors $y_1$, \ldots, $y_n$;
in particular, $U_2$, $U_3\in I$, $U_0\subseteq U_2+U_3$, $U_2\subseteq X\subseteq U_1^\perp$,
and $U_3\subseteq Y\subseteq U_1^\perp$.
Thus we have by our hypothesis that there are $J$, $K\in\mathcal{U}$ such that
$[U_2,U_1,J]\colon S\rightarrow X$ and $[U_3,U_1,K]\colon T\rightarrow Y$.
This means that $J\cap K\in\mathcal{U}$ and for any $U\in J\cap K$, one has
$U_0\subseteq U_2+U_4\subseteq S_U +T_U\subseteq U_1^\perp$.
This proves that $[U_0,U_1,J\cap K]\colon S+T\rightarrow X+Y$.
\end{scproof}

\noindent
According to \Fact \ref{og}(ii), $\mathbb{L}(U_F)$ is a CMIL for any finite-dimensional subspace $U_F$ of $V_F$,
whence $M$ is a CMIL.
From Claims \ref{c-5}-\ref{c-7} and Lemma \ref{L:1}(ii), we get the following

\begin{claim}\lab{c-8}
The set $N=\{S\in M\mid S\rightarrow X\ \text{for some}\ X\in L\}$
is a subalgebra of $M$ and the map
\[
g\colon N\to L;\quad g\colon S\mapsto X,\ \text{if}\ S\rightarrow X
\]
is a well defined homomorphism.
\end{claim}

\begin{claim}\lab{c-9}
The homomorphism $g$ is surjective.
\end{claim}

\begin{scproof}
Given $X\in L$, let $S\in M$ be such that
$S_U=X\cap U+U\cap U^\perp$ for all $U\in I$.
Suppose now that $U_0$, $U_1\in I$ are such that $U_0\subseteq X\subseteq U_1^\perp$
and put $J=U_0^+\cap U_1^+$.
If $U\in J$, then $U_0$, $U_1\subseteq U$; in particular,
$U^\perp\subseteq U_1^\perp$. This implies that
$U_0\subseteq X\cap U\subseteq S_U=X\cap U+U\cap U^\perp\subseteq U_1^\perp$,
whence $[U_0,U_1,J]\colon S\rightarrow X$ and $g(S)=X$.
\end{scproof}

\noindent
We put $\hat{N}=\{[S]\mid S\in N\}$.

\begin{claim}\lab{c-10}
$\hat{N}$ is a subalgebra of $\hat{M}$ and the map 
\[
\hat{g}\colon\hat{N}\to L,\quad f\colon[S]\mapsto X,\ \text{if}\ S\rightarrow X
\]
is a well defined surjective homomorphism.
\end{claim}

\begin{scproof}
In view of Claims \ref{c-8}-\ref{c-9}, it suffices to show that $[T]=[S]$ and $S\rightarrow X$ imply $T\rightarrow X$.
Indeed, let $K=\{U\in I\mid S_U=T_U\}$; then $K\in\mathcal{U}$.
If $J\in\mathcal{U}$ is such that $[U_0,U_1,J]\colon S\rightarrow X$, then obviously,
$[U_0,U_1,J\cap K]\colon T\rightarrow X$.
\end{scproof}

\begin{claim}\lab{c-12}
If $S\rightarrow V$ for some $S\in M$, then $S=1_M$.
\end{claim}

\begin{scproof}
Let $U\in I$ be arbitrary and let $u\in U$ be also arbitrary.
As $S\rightarrow V$ and $uF\subseteq V$, there is $J\in\mathcal{U}$
such that $J\subseteq uF^+$ and $[uF,\{0\},J]\colon S\rightarrow V$.
As $uF\subseteq U$, we get that $U\in J$ and thus $u\in uF\subseteq S_U$.
As $u$ was an arbitrary vector, we conclude that $S_U=U$.
As $U\in I$ was arbitrary, $S=1_M$.
\end{scproof}

\begin{claim}\lab{c-11}
If $L$ is complemented, then $\hat{N}$ is complemented.
\end{claim}

\begin{scproof}
As $\hat{N}$ is a homomorphic image of $N$, it suffices to prove that $N$ is complemented.
Let $S\in M$. Then $S\rightarrow X$ for some $X\in L$.
Let $X\oplus Y=V$ in $L$.
By Claim \ref{c-9}, there is $T\in N$ such that $T\rightarrow Y$.
According to Claim \ref{c-7}, $S+T\rightarrow X+Y=V$, whence $S+T=1_M$ by Claim \ref{c-12}.
According to Claim \ref{c-6}, $S^\perp\rightarrow X^\perp$ and $T^\perp\rightarrow Y^\perp$,
whence $(ST)^\perp=S^\perp+T^\perp\rightarrow X^\perp+Y^\perp=(X\cap Y)^\perp=V$
and $(ST)^\perp=1_M$ Claim \ref{c-12}.
Therefore, $ST=0_M$ and $T$ is a complement of $S$.
\end{scproof}

\noindent
All desired statements of Theorem \ref{findim} follow from Claims \ref{c-10}-\ref{c-11}.
\end{remark}
}

\section{($\exists$-)semivarieties of representable structures}

We denote by $\mathcal{L}(\mathcal{S})$ [$\mathcal{R}(\mathcal{S})$, respectively
the class of all CMILs [all $R\in\mathcal{R}_{\Lambda}$ respectively] having
a faithful representation within some member of $\mathcal{S}$
(we also say that these structures are \emph{representable} within $\mathcal{S}$).
We consider here conditions on $\mathcal{S}$ which ensure that classes $\mathcal{L}(\mathcal{S})$ and
$\mathcal{R}(\mathcal{S})$ are $\exists$-(semi)varieties.
Observe that
\[
\mathcal{L}\bigl({\sf I}_{\sf s}(\mathcal{S})\bigr)=\mathcal{L}(\mathcal{S})\quad\text{and}\quad
\mathcal{R}\bigl({\sf I}_{\sf s}(\mathcal{S})\bigr)=\mathcal{R}(\mathcal{S}).
\]

\begin{proposition}\lab{malcev}
Let $\mathcal{S}$ be a $[$recursively$]$ axiomatized class of pre-hermitian spaces over a $[$recursive$]$
commutative $\ast$-ring $\Lambda$.
Then $\mathcal{L}(\mathcal{S})$ and $\mathcal{R}(\mathcal{S})$ are $[$recursively$]$ axiomatizable.
\end{proposition}

\begin{proof}
Let $\Gamma_r$ denote the set of first order axioms
defining representations $_RV_F$ with $R \in \mc{R}_\Lambda$ and $V_F\in\mathcal{S}$ 
and let $\Sigma_r$ denote the set of all universal sentences in the signature
of $\ast$-$\Lambda$-algebras which are consequences of $\Gamma_r$.
Then $\Sigma_r$ defines the class of all $\ast$-$\Lambda$-algebras representable in $\mathcal{S}$.
Adding to $\Sigma_r$ the $\forall\exists$-axiom of regularity defines
the subclass $\mathcal{R}(\mathcal{S})$.
If $\Lambda$ is recursive and $\mathcal{S}$ is recursively axiomatizable, then $\Gamma_r$ is recursive.
By G\"{o}del's Completeness Theorem, $\Sigma_r$ is recursively enumerable. 
By Craig's trick \cite[Exercise 6.1.3]{hodges}, $\Sigma_r$ is also recursive.

Similarly, taking $\Gamma_l$ to be the set of first order axioms
defining representations of CMILs within spaces from $\mathcal{S}$,
and denoting by $\Sigma_l$ the set of all universal sentences in the signature
of CMILs which are consequences of $\Gamma_l$, we get that
$\Sigma_l$ defines the class $\mathcal{L}(\mathcal{S})$ of all CMILs representable in $\mathcal{S}$.
Moreover, if $\Gamma_l$ is recursive, then $\Sigma_l$ is also recursive.
We also refer to \cite{mal,Sch}.
\end{proof}

\noindent
A \emph{tensorial embedding} of a pre-hermitian space $V_F$ into another one $W_K$
is given by a $\ast$-$\Lambda$-algebra embedding $\alpha\colon F\to K$ and
an injective $\alpha$-semilinear map $\varepsilon\colon V_F\to W_K$
such that $W_K$ is spanned by $\im\varepsilon$ as a $K$-vector space and
$\sk{\varepsilon(v)}{\varepsilon(w)}=\alpha\bigl(\sk{v}{w}\bigr)$ for all $v$, $w\in V$;
in particular, $\varepsilon$ is an isomorphism of $V_F$ onto a two-sorted substructure of $W_K$.
A \emph{joint tensorial extension} of spaces ${V_i}_{F_i}$, $i\in\{0,1\}$,
is given by a pre-hermitian space $W_F=U_0\oplus^\perp U_1$
and tensorial embedding of ${V_i}_{F_i}$ into ${U_i}_F$ for $i\in\{0,1\}$.

\begin{lemma}\lab{jeck}
Let $F$, $F_0$, $F_1\in\mathcal{A}_\Lambda$, let $V_F$ be a pre-hermitian space,
and let ${V_0}_{F_0}$ and ${V_1}_{F_1}$ be finite-dimensional pre-hermitian spaces.
\begin{enumerate}
\item
If $\alpha_i$ and $\varepsilon_i$ define a tensorial embedding of ${V_i}_{F_i}$ into $V_F$, $i<2$,
then $\LEnd({V_i}_{F_i})$ embeds into $\LEnd(V_F)$ and $\mathbb{L}({V_i}_{F_i})$ embeds into $\mathbb{L}(V_F)$.
\item
If $V_F$ is a joint tensorial extension of ${V_0}_{F_0}$ and ${V_1}_{F_1}$,
then $\LEnd({V_0}_{F_0})\times\LEnd({V_1}_{F_1})$ embeds into $\LEnd(V_F)$
and $\mathbb{L}({V_0}_{F_0})\times\mathbb{L}({V_1}_{F_1})$ embeds into $\mathbb{L}(V_F)$.
\end{enumerate}
\end{lemma}

\begin{proof}
(i) In view of \Fact \ref{db}(i),
${V_i}_{F_i}$ has a dual pair $\{v_1,\ldots,v_n\}$, $\{w_1,\ldots,w_n\}$ of bases;
applying $\varepsilon_i$, one obtains such a pair for $V_F$.
Indeed, $V_F$ is obviously spanned by both, $\{\varepsilon_i(v_1),\ldots,\varepsilon_i(v_n)\}$
and $\{\varepsilon_i(w_1),\ldots,\varepsilon_i(w_n)\}$.
Suppose that $\Sigma_{j=1}^n\varepsilon_i(v_j)\lambda_j=0$ for some $\lambda_1$, \ldots, $\lambda_n\in F$.
Then for any $k\in\{1,\ldots,n\}$, one gets
\[
0=\sk{0}{w_k}=\Sk{\Sigma_{j=1}^n\varepsilon_i(v_j)\lambda_j}{\varepsilon_i(w_k)}=
\Sigma_{j=1}^n\alpha_i\bigl(\sk{v_j}{w_k}\bigr)\lambda_j=\lambda_k,
\]
whence $\{\varepsilon_i(v_1),\ldots,\varepsilon_i(v_n)\}$ is a basis of $V_F$. Similarly,
$\{\varepsilon_i(w_1),\ldots,\varepsilon_i(w_n)\}$ is a basis of $V_F$.

For $\varphi\in\LEnd({V_i}_{F_i})$, let $\xi_i(\varphi)$ be the $F$-linear map on $V$ defined by
$\xi_i(\varphi)\colon\varepsilon_i(v_j)\mapsto\varepsilon_i\bigl(\varphi(v_j)\bigr)$ for all $j\in\{1,\ldots,n\}$.
Clearly, $\xi_i$ is a $\Lambda$-algebra embedding of $\LEnd({V_i}_{F_i})$ into $\LEnd(V_F)$.
Moreover, for any $j$, $k\in\{1,\ldots,n\}$, one has
\begin{align*}
&\Sk{\varepsilon_i(v_j)}{\xi_i(\varphi^\ast)\varepsilon_i(v_k)}=
\Sk{\varepsilon_i(v_j)}{\varepsilon_i\varphi^\ast(v_k)}=
\alpha\bigl(\sk{v_j}{\varphi^\ast(v_k)}\bigr)=\alpha\bigl(\sk{\varphi(v_j)}{v_k}\bigr)=\\
&=\Sk{\varepsilon_i\varphi(v_j)}{\varepsilon_i(v_k)}=
\Sk{\xi_i(\varphi)\varepsilon_i(v_j)}{\varepsilon_i(v_k)},
\end{align*}
whence $\xi_i(\varphi^\ast)=\xi_i(\varphi)^\ast$ by \Fact \ref{db}(ii).
For the claim about polarity lattices, apply \Facts \ref{reri}(i), \ref{neu}(iv), and \ref{endreg}(ii)-(iii).

(ii) As $V_F=U_0\oplus^\perp U_1$, by (i), there are $\ast$-$\Lambda$-algebra embeddings
\[
\xi_i\colon\LEnd({V_i}_{F_i})\to\LEnd({U_i}_F),\quad i\in\{0,1\}.
\]
Thus there is a unique embedding
\[
\xi\colon\LEnd({V_0}_{F_0})\times\LEnd({V_1}_{F_1})\to\LEnd(V_F)
\]
such that $\xi(\varphi_0,\varphi_1)|_{U_i}=\xi_i(\varphi_i)$ for $i\in\{0,1\}$.
By \Facts \ref{neu}(iii), \ref{db}(i), and \ref{Lrep}(ii),
\[
\mathbb{L}(V_{0F_0})\times\mathbb{L}(V_{1F_1})\cong
\mathbb{L}\bigl(\LEnd(V_{0F0})\bigr)\times\mathbb{L}\bigl(\LEnd(V_{1F_1})\bigr)\cong
\mathbb{L}\bigl(\LEnd(V_{0F_0})\times\LEnd(V_{1F_1})\bigr).
\]
By (i) and Proposition \ref{Lrep}(i),
the latter admits a faithful representation in $V_F$.
\end{proof}

\begin{theorem}\lab{flo}
Let $\mathcal{S}$ be a class of pre-hermitian spaces over $\Lambda$. Then
\begin{enumerate}
\item
$\mathcal{L}\bigl({\sf S}_{1{\sf q}}{\sf P}_{\sf u}(\mathcal{S})\bigr)=
\mathcal{L}\bigl({\sf S}{\sf P}_{\sf u}{\sf I}_{\sf s}(\mathcal{S})\bigr)=
{\sf W}_{\exists}\bigl(\mathcal{L}(\mathcal{S})\bigr)=
{\sf W}_{\exists}\bigl(\mathbb{L}(V_{F})\mid V_F\in{\sf S}_{1 {\sf f}}(\mathcal{S})\bigr)$;
\item
$\mathcal{R}\bigr({\sf S}_{1{\sf q}}{\sf P}_{\sf u}(\mathcal{S})\bigr)=
\mathcal{R}\bigl({\sf S}{\sf P}_{\sf u}{\sf I}_{\sf s}(\mathcal{S})\bigr)=
{\sf W}_\exists\bigl(\mathcal{R}(\mathcal{S})\bigr)=
{\sf W}_\exists\bigl(\LEnd(V_{F})\mid V_F\in{\sf S}_{1 {\sf f}}(\mathcal{S})\bigr)$.
\end{enumerate}
In particular, if the class $\mathcal{S}$ is a semivariety then the classes
$\mathcal{L}(\mathcal{S})=\mathcal{L}\bigl({\sf S}{\sf P}_{\sf u}{\sf I}_{\sf s}(\mathcal{S})\bigr)$
and $\mathcal{R}(\mathcal{S})=\mathcal{R}\bigl({\sf S}{ P}_{\sf u}{\sf I}_{\sf s}(\mathcal{S})\bigr)$
are $\exists$-semivarieties generated by their
strictly simple finite-dimensional or artinian members, respectively.
\end{theorem}

\begin{proof}
The proofs of (i) and (ii) follow the same lines.
We prove (ii).

The fact that ${\sf S}_\exists{\sf P}_{\sf u}\bigl(\mathcal{R}(\mathcal{S})\bigr)\subseteq
\mathcal{R}\bigl({\sf P}_{\sf u}(\mathcal{S})\bigr)$ follows immediately from Lemma \ref{ultra}(iii).
Then ${\sf W}_\exists\bigl(\mathcal{R}(\mathcal{S})\bigr)\subseteq
\mathcal{R}\bigl({\sf S}_{1{\sf q}}{\sf P}_{\sf u}(\mathcal{S})\bigr)$ by Theorem \ref{homlat}.
By Theorem \ref{findim},
$\mathcal{R}\bigl({\sf S}_{1{\sf q}}{\sf P}_{\sf u}(\mathcal{S})\bigr)\subseteq
{\sf W}_\exists\bigl(\LEnd(V_F)\mid V_F\in{\sf S}_{1 {\sf f}}{\sf S}_{1{\sf q}}{\sf P}_{\sf u}(\mathcal{S})\bigr)$.
By Lemmas \ref{ultra}(iii) and \ref{L:S}, for any
$V_F\in{\sf S}_{1 {\sf f}}{\sf S}_{1{\sf q}}{\sf P}_{\sf u}(\mathcal{S})={\sf S}_{1 {\sf f}}{\sf P}_{\sf u}(\mathcal{S})$,
we have $V_F\in{\sf P}_{\sf u}{\sf S}_{1 {\sf f}}(\mathcal{S})$
and $\LEnd(V_F)\in{\sf P}_{\sf u}\bigl(\LEnd(W_K)\mid W_K\in{\sf S}_{1 {\sf f}}(\mathcal{S})\bigr)$.
It follows that
\[
{\sf W}_\exists\bigl(\mathcal{R}(\mathcal{S})\bigr)\subseteq
\mathcal{R}\bigl({\sf S}_{1{\sf q}}{\sf P}_{\sf u}(\mathcal{S})\bigr)\subseteq 
{\sf W}_\exists\bigl(\LEnd(W_K)\mid W_K\in{\sf S}_{1 {\sf f}}(\mathcal{S})\bigr)\subseteq
{\sf W}_\exists\bigl(\mathcal{R}(\mathcal{S})\bigr).
\]
Now, consider $R\in\mathcal{R}\bigl({\sf S}{\sf P}_{\sf u}(\mathcal{S})\bigr)$;
that is, $R$ is represented in a $2$-sorted substructure $W_K$ of some $V_F\in{\sf P}_{\sf u}(\mathcal{S})$.
By Theorem \ref{findim}, we have
$R\in{\sf W}_\exists\bigl(\LEnd(U_K)\mid U_K\in{\sf S}_{1 {\sf f}}(W_K)\bigr)$.
Let $U^\prime_F$ denote the $F$-subspace of $V_F$ spanned by $U$.
By Lemma \ref{jeck}(i), $\LEnd(U_K)\in {\sf S}_\exists\bigl(\LEnd(U^\prime_F)\bigr)$.
Thus, $R\in{\sf W}_\exists\bigl(\mathcal{R}(\mathcal{S})\bigr)$.
Hence
\[
\mathcal{R}\bigl({\sf S}{\sf P}_{\sf u}{\sf I}_{\sf s}(\mathcal{S})\bigr)\subseteq
\mathcal{R}\bigl({\sf I}_{\sf s}{\sf S}{\sf P}_{\sf u}(\mathcal{S})\bigr)
=\mathcal{R}\bigl({\sf S}{\sf P}_{\sf u}(\mathcal{S})\bigr)\subseteq
{\sf W}_\exists\bigl(\mathcal{R}(\mathcal{S})\bigr)=\mathcal{R}\bigl({\sf S}_{1{\sf q}}{\sf P}_{\sf u}(\mathcal{S})\bigr).
\]
The containment $\mathcal{R}\bigl({\sf S}_{1{\sf q}}{\sf P}_{\sf u}(\mathcal{S})\bigr)\subseteq
\mathcal{R}\bigl({\sf S}{\sf P}_{\sf u}{\sf I}_{\sf s}(\mathcal{S})\bigr)$ is trivial by Lemma \ref{L:S}.
\end{proof}

\noindent
More closure properties on $\mathcal{S}$ are needed if one intends to get a one-to-one correspondence
between classes of spaces and classes of structures in Theorem \ref{flo}.

\begin{definition}\lab{spread}
Let $V_F$, $W_K$ be pre-hermitian spaces over $\Lambda$,
$\dim V_F<\omega$, and let $\mathcal{S}$ be a class of pre-hermitian spaces over $\Lambda$.
\begin{enumerate}
\item
The sesquilinear space
$V_F$ is an $L$-\emph{spread} of $W_K$ if $\dim V_F>2$
and $\mathbb{L}(V_F)\in\mathcal{L}(W_K)$.
The class $\mathcal{S}$ is $L$-\emph{spread closed}, if it contains
all $L$-spreads of its members.
\item
The sesquilinear space
$V_F$ is an $R$-\emph{spread} of $W_K$ if $\LEnd(V_F)\in\mathcal{R}(W_K)$.
The class $\mathcal{S}$ is $R$-\emph{spread closed},
if it contains all $R$-spreads of its members.
\item
An $R$-[$L$-]spread closed universal class or a semivariety $\mathcal{S}$ is \emph{small},
if $\mathcal{S}$ coincides with the smallest $R$-[$L$-]spread closed universal class or a semivariety
which contains all members of $\mathcal{S}$ of dimension $n<\omega$ [of dimension $2<n<\omega$, respectively].
\end{enumerate}
\end{definition}

\xc{\noindent
The following statement follows directly from Definition \ref{spread}(i)-(ii).

\begin{lemma}\lab{LR-sp}
Let $\mathcal{S}$ be a class of pre-hermitian spaces over $\Lambda$.
\begin{enumerate}
\item
If $\mathcal{S}$ is $L$-spread closed and $V_F\in{\sf S}_{1 {\sf f}}(\mathcal{S})$ with $\dim V_F>2$, then
$V_F\in\mathcal{S}$.
\item
If $\mathcal{S}$ is $R$-spread closed, then ${\sf S}_{1 {\sf f}}(\mathcal{S})\subseteq\mathcal{S}$.
\end{enumerate}
\end{lemma}}

\begin{example}
Consider the class $\mathcal{S}$ of all anisotropic hermitian spaces,
where $F\in{\sf SP}_{\sf u}(\mathbb{Q})$;
in particular, $F\models\forall x\, [x^2\neq 2]$ and $\mathcal{S}$ is a universal class which does not contain $K^3_K$ with the canonical scalar product, where $K=\mathbb{Q}(\sqrt{2})$.
Nonetheless, $K^{3\times 3}$ and whence $\mathbb{L}(K^{3\times 3})$
is representable within $\mathbb{Q}^6_{\mathbb{Q}}\in\mathcal{S}$ by
\[
a+b\sqrt{2}\mapsto a\left(
\begin{matrix}
1&0\\
0&1
\end{matrix}
\right)+b\left(
\begin{matrix}
1&1\\
1&-1
\end{matrix}
\right);\quad\text{where}\ a,b\in\mathbb{Q},
\] 
which yields a $\ast$-ring embedding of $K$ into $\mathbb{Q}\ma{2}$
thus giving rise to an embedding of $K^{3\times 3}$ into $(\mathbb{Q}\ma{2})\ma{3}$.
In the sense of Definition \ref{spread},
$K^3_K$ is an $L$-spread and an $R$-spread of $\mathbb{Q}^6_{\mathbb{Q}}$.
\end{example}

\begin{theorem}\lab{11}
\hfill
\begin{enumerate}
\item
For any $\exists$-semivariety $\mathcal{V}$ of
Arguesian CMILs generated by its strictly simple members of finite dimension at least $3$,
there is a small $L$-spread closed semivariety $[$universal class$]$ $\mathcal{S}$
of pre-hermitian spaces over $\mathbb{Z}$ such that $\mathcal{V}=\mathcal{L}(\mathcal{S})$.
Moreover, the class of members of $\mathcal{S}$ of dimension at least $3$ is unique.
\item
For any $\exists$-semivariety $\mathcal{V}\subseteq\mathcal{R}_{\Lambda}$
generated by its strictly simple artinian members,
there is a small $R$-spread closed semivariety $[$universal class$]$ $\mathcal{S}$
of pre-hermitian spaces over $\Lambda$ such that $\mathcal{V}=\mathcal{R}(\mathcal{S})$.
Moreover, such a class $\mathcal{S}$ is unique.
\end{enumerate}
The class $\mathcal{S}$ above is anisotropic, if $\mathcal{V}$ consists of MOLs or $\mathcal{V}\subseteq\mathcal{R}^\ast_{\Lambda}$.
\end{theorem}

\begin{remark}
If  $\mathcal{V}$ consists of MOLs (context of (i))
or $\mathcal{V} \subseteq  \mathcal{R}_\Lambda^\ast $
(context of (ii))
 it suffices to require that $\mathcal{V}$
is generated by its  simple members which are 
of finite dimension  respectively artinian and that, in context of (i),  $\mathcal{V}$ is not $2$-distributive.
Then, in context of  (i), $\mathcal{V}$ contains all MOLs of dimension $2$.
\end{remark}

\begin{proof}
(i)
Given an $\exists$-semivariety $\mathcal{V}\subseteq \mc{R}_\Lambda$ with all required properties,
let $\mathcal{K}_{\mathcal{V}}$ denote the class of strictly simple
artinian  members of $\mathcal{V}$.
By \Fact \ref{P:iso}, for any $R\in\mathcal{K}_{\mathcal{V}}$,
there is a pre-hermitian space $V_F$ over $\Lambda$ such that
 $R\cong\LEnd(V_F)$.
By $\mathcal{S}_{\mathcal{V}}$, we denote the class of spaces $V_F$ over $\Lambda$ 
such that $\LEnd(V_F)\in\mathcal{K}_{\mathcal{V}}$.

We put $\mathcal{G}_0={\sf S}_{1 {\sf f}}(\mathcal{S}_{\mathcal{V}})$.
For any ordinal $\alpha$, let $\mathcal{G}_{\alpha+1}$ be the union of
two classes: ${\sf P}_{\sf u}(\mathcal{G}_\alpha)$ and the class of all
$V_F\in{\sf S}_{1 {\sf f}}(V^\prime_F)$,
where $V^\prime_F$ is an $R$-spread of some $W_K\in\mathcal{G}_\alpha$.
Let also $\mathcal{G}_\alpha=\bigcup_{\beta<\alpha}\mathcal{G}_\beta$,
if $\alpha$ is a limit ordinal.

\setcounter{claim}{0}
\begin{claim}\lab{c-13}
${\sf S}_{1 {\sf f}}(\mathcal{G}_\alpha)\subseteq\mathcal{G}_\alpha$ and $\LEnd(V_F)\in\mathcal{V}$
for any $\alpha$ and $V_F\in\mathcal{G}_\alpha$ with $\dim V_F<\omega$.
\end{claim}

\begin{scproof}
We argue by induction on $\alpha$.
For $\alpha=0$, the first claim follows from the definition of $\mathcal{G}_0$. 
Moreover, if $U_F\in{\sf S}_{1 {\sf f}}(V_F)$ and $\LEnd(V_F)\in\mathcal{V}$
then $\LEnd(U_F)\in{\sf H}{\sf S}_\exists(\LEnd(V_F))\subseteq\mathcal{V}$ by \Fact \ref{db}(iii).
The limit step is trivial.
In the step from $\alpha$ to $\alpha+1$,
we assume first that $V_F$ is isomorphic to an ultraproduct of spaces
${V_i}_{F_i}\in\mathcal{G}_\alpha$, $i\in I$.
If $U_F\in{\sf S}_{1 {\sf f}}(V_F)$ and $n=\dim U_F$ then, by Lemma \ref{ultra}(iii),
$U_F$ is isomorphic to an ultraproduct of some $U_{iF_i}\in{\sf S}_{1 {\sf f}}(V_{iF_i})$ with $\dim U_{iF_i}=n$,
$i\in J$, for some $J\subseteq I$.
By the inductive hypothesis, $U_{iF_i}\in\mathcal{G}_\alpha$
and $\LEnd(U_{iF_i})\in\mathcal{V}$. Thus $U_F\in\mathcal{G}_{\alpha+1}$
and $\LEnd(U_F)\in\mathcal{V}$ by Lemma \ref{ultra}(iii).
 
Now, let $V^\prime_F$ be an $R$-spread of $W_K\in\mathcal{G}_\alpha$ and let $V_F\in{\sf S}_{1 {\sf f}}(W_K)$.
If $U_F\in{\sf S}_{1 {\sf f}}(V_F)$ then $U_F\in{\sf S}_{1 {\sf f}}(V^\prime_F)$,
whence $U_F\in\mathcal{G}_{\alpha+1}$ by definition.
By Theorem \ref{findim} and the inductive hypothesis,
\[
\LEnd(V^\prime_F)\in{\sf W}_\exists\bigl(\LEnd(W^\prime_K)\mid W^\prime_K\in\mathbb{O}(W_K)\bigr)
\subseteq\mathcal{V}.
\]
By \Fact \ref{og}(i),
$\LEnd(U_F)\in{\sf H}{\sf S}_\exists(\LEnd(V^\prime_F))\subseteq\mathcal{V}$.
\end{scproof}

It follows that the $R$-spread closed semivariety $\mathbb{K}(\mathcal{V})$ of pre-hermitian spaces
over $\Lambda$ generated by $\mathcal{S}_{\mathcal{V}}$
is the union of the classes $\mathcal{G}_\alpha$, where $\alpha$ ranges over all ordinals.
Thus in view of the assumption $\mathcal{V}={\sf W}_{\exists}(\mathcal{K}_{\mathcal{V}})$
and Claim \ref{c-13}, one gets by Theorem \ref{flo}(ii)
\[
\mathcal{V}\subseteq\mathcal{R}\bigl(\mathbb{K}(\mathcal{V})\bigr)=
{\sf W}_\exists\bigl(\LEnd(V_F)\mid
V_F\in\mathbb{K}(\mathcal{V}),\,\dim V_F<\omega\bigr)\subseteq\mathcal{V}.
\]
To prove uniqueness, let $\mathcal{S}$ and $\mathcal{S}^\prime$ be small $R$-spread closed
semivarieties of pre-hermitian spaces over $\Lambda$ such that
$\mathcal{R}(\mathcal{S})=\mathcal{V}=\mathcal{R}(\mathcal{S}^\prime)$.
For any $V_F\in\mathcal{S}$ with $\dim V_F<\omega$, we have
$\LEnd(V_F)\in\mathcal{R}(\mathcal{S})=\mathcal{R}(\mathcal{S}^\prime)$,
whence $V_F$ is an $R$-spread of $\mathcal{S}^\prime$ and $V_F\in\mathcal{S}^\prime$.
Similarly, interchanging the roles of $\mathcal{S}$ and $\mathcal{S}^\prime$,
we get that $\mathcal{S}$ and $\mathcal{S}^\prime$ have the same
artinian members.

To deal with the case of universal classes, one includes into the union
$\mathcal{G}_\alpha$ a third class, namely ${\sf S}(\mathcal{G}_\alpha)$.
Claim \ref{c-13} and its proof remain valid, only the case of the third class remains to be considered.
Indeed, assume that $V_F\in\mathcal{G}_{\alpha+1}$ is a $2$-sorted substructure of $W_K\in\mathcal{G}_\alpha$
and let $U_F\in{\sf S}_{1 {\sf f}}(V_F)$. Then $U_F\in{\sf S}(W_K)$ and $U_F\in\mathcal{G}_{\alpha+1}$ by definition.
Moreover, $U_F$ is a $2$-sorted substructure of the $K$-subspace $U^\prime_K$ of $W_K$ spanned by $U$.
In particular, $U^\prime_K\in{\sf S}_{1 {\sf f}}(W_K)$ and the inductive
hypothesis yields $U^\prime_K\in\mathcal{G}_\alpha$ and $\LEnd(U^\prime_K)\in\mathcal{V}$.
As $\LEnd(U_F)$ embeds into $\LEnd(U^\prime_K)$ by Lemma \ref{jeck}(i),
it follows that $\LEnd(U_F)\in\mathcal{V}$.

(i) The proof follows the same lines as the one of (ii) replacing
$\Lambda$ by $\mb{Z}$,
 \Fact \ref{P:iso} by \Fact \ref{atrepc}. \Fact \ref{db}(iii) by
\Fact \ref{og}(i), and Theorem \ref{flo}(ii) by Theorem \ref{flo}(i).
For  $L\in \mathcal{K}_{\mathcal{V}}$ one has to require 
$3 \leq \dim L < \omega$.
\end{proof}
\noindent
For results of the same type as Theorem \ref{11}, see also \cite[Theorems 4.4-5.4]{HS2}.

\section{$\exists$-varieties and representations}\lab{var}

We first consider a condition on $\mathcal{S}$ under which the class of representables is an $\exists$-variety.
Then we review the approach of Micol \cite{Flo} to capture $\exists$-varieties via the concept of generalized representation.

A semivariety $\mathcal{S}$ of pre-hermitian spaces over $\Lambda$ is a \emph{variety} if
for any finite-dimensional ${V_0}_{F_0}$, ${V_1}_{F_1}\in\mathcal{S}$,
there is a joint tensorial extension $V_F\in\mathcal{S}$.

\begin{proposition}\lab{varsp}
If $\mathcal{S}$ is a variety of pre-hermitian spaces over $\Lambda$,
then $\mathcal{L}(\mathcal{S})$ and $\mathcal{R}(\mathcal{S})$ are $\exists$-varieties.
\end{proposition}

\begin{proof}
In view of Proposition \ref{hs1}(iv) and Theorem \ref{flo}, it suffices to notice that
for any finite-dimensional spaces ${V_0}_{F_0}$, ${V_1}_{F_1}\in\mathcal{S}$, the structures
$\LEnd({V_0}_{F_0})\times\LEnd({V_1}_{F_1})$ and $\mathbb{L}({V_0}_{F_0})\times\mathbb{L}({V_1}_{F_1})$
have a faithful representation within some member of $\mathcal{S}$ by Lemma \ref{jeck}(ii).
\end{proof}

\noindent
Classes $\mathcal{L}(\mathcal{S})$ of CMILs having a faithful
representation within some member of a class $\mathcal{S}$ of orthogeometries
have been considered in \cite{he4}.
The closure properties of Theorem \ref{flo}(i) hold also in this case
with ${\sf S}(\mathcal{S})$ denoting formation of non-degenerate subgeometries of members of $\mathcal{S}$,
${\sf S}_{1 {\sf f}}(\mathcal{S})$ and ${\sf S}_{1{\sf q}}(\mathcal{S})$~--- formation
of non-degenerate finite-dimensional subspaces and of subquotients
$U{\slash}\rad U$, where $U_F\in\mathcal{S}$ and $U=U^{\perp\perp}$.
In addition, one has the class ${\sf U}(\mathcal{S})$ of all disjoint orthogonal unions of members of $\mathcal{S}$
and thus ${\sf P}\bigl(\mathcal{L}(\mathcal{S})\bigr)\subseteq\mathcal{L}\bigl({\sf U}(\mathcal{S})\bigr)$,
cf. \cite[Theorem 2.2]{he4}.
Moreover, mimicking the concept of an $L$-spread and the proof of Theorem \ref{11}, one obtains

\begin{theorem}
For any $\exists$-variety $\mathcal{V}$ of CMILs generated by its
finite-dimensional  members,
there is a small $L$-spread and ${\sf U}$-closed semivariety $[$universal class$]$ $\mathcal{S}$ of orthogeometries
such that $\mathcal{V}=\mathcal{L}(\mathcal{S})$. Moreover, such a class $\mathcal{S}$ is unique.
\end{theorem}

\noindent
The objective of Micol \cite{Flo} was to derive results for $\ast$-regular rings, analogous to those above.
Of course, representation requires some structure of the type of sesquilinear spaces.
Apparently, in general there is no axiomatic class of such spaces which would serve for
representing direct products of representable structures.
Micol solved this problem by introducing the concept of a \emph{generalized representation}.
This concept was transferred to MOLs by Niemann \cite{Nik}.

A $g$-\emph{representation} of $A\in\mathrm{CMIL}$ $[A\in\mathcal{R}_{\Lambda}]$
within a class $\mathcal{S}$ of pre-hermitian spaces
is a family $\{\varepsilon_i\mid i\in I\}$ of representations $\varepsilon_i$ of $A$
in $V_{iF_i}\in\mathcal{S}$, $i\in I$.
It is \emph{faithful} if $\bigcap_{i\in I}\kerr\varepsilon_i$ is trivial.
Let $\mathcal{L}_{\sf g}(\mathcal{S})$ [$\mathcal{R}_{\sf g}(\mathcal{S})$]
denote the class of all $A\in\mathrm{CMIL}$ $[A\in\mathcal{R}_{\Lambda}]$
having a faithful $g$-representation within $\mathcal{S}$;
equivalently, the class of structures $A$ having a subdirect decomposition
into factors $\varepsilon_i(A)$, $i\in I$, which have a faithful representation within $\mathcal{S}$.

Call an artinian algebra $R\in\mathcal{R}_{\Lambda}$ \emph{strictly artinian} if $I=I^\ast$ for any ideal $I$ of $R$.
By the Wedderburn-Artin Theorem, this is equivalent to the fact that $R$ is isomorphic to
a direct product of strictly simple factors (cf. \cite[\S 3.4]{lamb}).
Similarly, call a finite-dimensional CMIL $L$ \emph{strictly finite-dimensional} if $\theta=\theta^\prime$ for any lattice
congruence $\theta$ of $L$.
By \cite[Theorem IV.7.10]{birk}), this is equivalent to the fact that
$L$ is a direct product of strictly simple factors.

\begin{proposition}\lab{flonik}
The following statements are true.
\begin{enumerate}
\item
For any semivariety $\mathcal{S}$ of pre-hermitian spaces, the class
$\mathcal{L}_{\sf g}(\mathcal{S})={\sf P}_{{\sf s}\exists}\bigl(\mathcal{L}(\mathcal{S})\bigr)$
$[\mathcal{R}_{\sf g}(\mathcal{S})={\sf P}_{{\sf s}\exists}\bigl(\mathcal{R}(\mathcal{S})\bigr)]$
is an $\exists$-variety generated by its strictly simple finite-dimensional $[$artinian$]$ members,
which are of the form $\mathbb{L}(V_F)$ $[\LEnd(V_F)]$ with $V_F\in\mathcal{S}$, $\dim V_F<\omega$.
\item
For any $\exists$-variety $\mathcal{V}\subseteq\mathrm{CMIL}$ $[\mathcal{V}\subseteq\mathcal{R}_{\Lambda}]$
which is generated by its strictly finite-dimensional at least $3$ $[$artinian$]$ members,
there is a semivariety $\mathcal{S}$ of pre-hermitian spaces such that
$\mathcal{V}=\mathcal{L}_{\sf g}(\mathcal{S})$ $[\mathcal{V}=\mathcal{R}_{\sf g}(\mathcal{S})]$.
\item
$A\in\mathcal{L}_{\sf g}(\mathcal{S})$ $[A\in\mathcal{R}_{\sf g}(\mathcal{S})]$
if and only if $A$ has an atomic extension $\hat{A}$ which is a subdirect product
of atomic strictly subdirectly irreducible structures $A_i$ such that
$(A_i)_{\sf f}\cong\mathbb{L}({V_i}_{F_i})_{\sf f}$
$[$the minimal ideal of $A_i$ is isomorphic to $J({V_i}_{F_i})]$ with ${V_i}_{F_i}\in\mathcal{S}$.
\end{enumerate}
\end{proposition}

\begin{proof}
Statement (i) follows from \Facts \ref{hs1}(iii)-(iv), \ref{atrepc}, \ref{P:iso}, and Theorem \ref{flo}.
Statement (ii) follows from \Facts \ref{hs1}(iv), \ref{atrepc}, \ref{P:iso}, and Theorem \ref{11}.
Finally, statement (iii) follows from \Facts \ref{atex}, \ref{atex2} and Theorems \ref{atrep}, \ref{atrep2}.
\end{proof}

\noindent
For $\ast$-regular rings, the result of Proposition \ref{flonik} is in essence due to Micol \cite{Flo}.
To prove that $g$-representability is preserved under homomorphic images,
she axiomatized families of inner product spaces as $3$-sorted structures, where
the third sort mimics the index set $I$.
Again, a saturation property is needed for the proof and regularity is crucial.
The fact that the $\exists$-variety of $g$-representable structures is generated by its artinian members
was shown by her reducing to countable subdirectly irreducible structures $R$,
deriving countably based representation spaces (and forming $2$-sorted subspaces),
and using the approach of Tyukavkin \cite{tyu} with respect to a countable orthogonal basis.
Conversely, a substantial part of Theorem \ref{flo} follows from Proposition \ref{flonik}.

\begin{appendix}
\section{Existence semivarieties}

We characterize $\exists$-(semi)varieties
contained in $\mathrm{CMIL}$ or in $\mathcal{R}_{\Lambda}$
as model classes, proving at the same time the
operator identities of \Fact  \ref{hs1}.
With no additional effort, this can be done
to include other classes of algebraic structures.

Given a set $\Sigma$ of first order axioms,
by $\Mod\Sigma$ we denote the model class
$\{A\mid A\models\Sigma\}$ of $\Sigma$.
By $\Th\mathcal{C}$ [$\Th_L\mathcal{C}$], we denote the set
of sentences [from the fragment $L$] of first order language which are valid in $\mathcal{C}$.
As usual, let $\overline{x}$ denote a sequence of variables of length being given by context.

\begin{definition}\lab{RegC}
A class $\mathcal{C}_0$
of algebraic structures of the same similarity type is
\emph{regular} if there is a (possibly empty) set $\Psi_0$ of
conjunctions $\alpha(\overline{x},y)$ of atomic formulas
$\bigl($i.e. formulas of the form $\bigwedge_{i=1}^k s_i(\overline{x},y)=t_i(\overline{x},y)\bigr)$
and a class $\mathcal{S}$ such that
\begin{enumerate}
\item
$\mathcal{C}_0=\mathcal{S}\cap\Mod\{\forall\overline{x}\exists y\ \alpha(\overline{x},y)\mid
\alpha(\overline{x},y)\in\Psi_0\}$;
\item
$\mathcal{S}$ is closed under ${\sf S}$ and $\mathcal{C}_0$ is closed under ${\sf H}$ and ${\sf P}$;
\item
For any structures $A$, $B\in\mathcal{C}_0$, for any
surjective homomorphism $\varphi\colon A\to B$, for any formula $\alpha(\overline{x},y)\in\Psi_0$,
and for any $\overline{a}$, $b\in B$ such that $B\models\alpha(\overline{a},b)$,
there are $\overline{c}$, $d\in A$ such that $\varphi(\overline{c})=\overline{a}$, $\varphi(d)=b$,
and $A\models\alpha(\overline{c},d)$.
\end{enumerate}
\end{definition}

\noindent
Without loss of generality, one may consider also the case when $\alpha(\overline{x},y)$ is of the form
$\bigwedge_{i=1}^kp_i\bigl(t_1(\overline{x},y),\ldots,t_m(\overline{x},y)\bigr)$, where
$p_i$ is a relation symbol of arity $m$ or the symbol $=$ with $m=2$.

From Definition \ref{RegC}(ii) it follows immediately that any regular class is closed under ${\sf P}_{\sf u}$.
In the sequel, we shall fix a regular class $\mathcal{C}_0$
and write for any $\mathcal{C}\subseteq\mathcal{C}_0$:
\[
{\sf S}_\exists(\mathcal{C})=\mathcal{C}_0\cap{\sf S}(\mathcal{C})\quad
\text{and}\quad{\sf P}_{{\sf s}\exists}(\mathcal{C})=\mathcal{C}_0\cap{\sf P}_{\sf s}(\mathcal{C}).
\]
Let $\mathcal{C}_0$ be a regular class.
A \emph{Skolem expansion} $A^\ast$ of $A\in\mathcal{C}_0$ adds for each $\alpha(\overline{x},y)\in\Psi_0$
an operation $f_\alpha$ on $A$ such that
$A\models\alpha(\overline{a},f_\alpha(\overline{a}))$ for all $\overline{a}\in A$.

\begin{definition}\lab{SRegC}
A class $\mathcal{C}_0$ is \emph{strongly regular} if it is regular and
\begin{enumerate} 
\item[$(\mathrm{iii}^\prime)$]
For any structures $A$, $B\in\mathcal{C}_0$, for any
surjective homomorphism $\varphi\colon A\to B$, for any formula $\alpha(\overline{x},y)\in\Psi_0$,
for any $\overline{a}$, $b\in B$ such that $B\models\alpha(\overline{a},b)$,
and for any $\overline{c}\in A$ such that $\varphi(\overline{c})=\overline{a}$
there is $d\in A$ such that $\varphi(d)=b$ and $A\models\alpha(\overline{c},d)$.
\end{enumerate}
\end{definition}

\begin{remark}\lab{SRC}
It is obvious that if a class $\mathcal{C}_0$ satisfies $(\mathrm{iii}^\prime)$ of Definition \ref{SRegC},
then $\mathcal{C}_0$ satisfies (iii) of Definition \ref{RegC}.
For any strongly regular class $\mathcal{C}_0$, for any $A$, $B\in \mathcal{C}_0$,
and for any surjective homomorphism $\varphi\colon A\to B$,
if $B^\ast$ is a Skolem expansion of $B$, then there is a Skolem expansion
$A^\ast$ of $A$ such that $\varphi\colon A^\ast\to B^\ast$ is a homomorphism.
Clearly, $\mathcal{C}_0$ is strongly regular if it
satisfies (i)-(ii) of Definition \ref{RegC} and for any $\alpha\in\Psi_0$ and for any
$\overline{a}\in A\in\mathcal{C}_0$, there is unique $b$ such that $\alpha(\overline{a},b)$.
This applies, in particular, to completely regular [inverse] semigroups.
\end{remark}

\noindent
In what follows, when we speak of a [strongly] regular class $\mathcal{C}$,
we always assume that the set of formulas $\Psi_0$ and
the classes $\mathcal{C}_0$ and $\mathcal{S}$ are given according to Definition \ref{RegC}
[Definition \ref{SRegC}, respectively].

\begin{proposition}\lab{exlem}
For any variety $\mathcal{V}$ with a $\ast$-ring reduct, the class of structures $A\in\mathcal{V}$ having
$\ast$-regular reducts forms a strongly regular class.
In particular, the class $\mathcal{R}^\ast_{\Lambda}$ of all $\ast$-regular
$\ast$-$\Lambda$-algebras is strongly regular.
\end{proposition}

\begin{proof}
Let $\Psi_0=\{xyx=y\}$ and let
$\mathcal{S}=\mathcal{V}\cap\Mod(\forall x\ xx^\ast=0\rightarrow x=0)$.
Then $\mathcal{C}_0$ defined as in Definition \ref{RegC}(i)
consists of the $\ast$-regular members of $\mathcal{V}$.
Closure of $\mathcal{C}_0$ under ${\sf H}$ and ${\sf P}$ follows from the fact
that $\ast$-regularity can be defined by the sentence:
\[
\forall x\exists y\ (y=y^2=y^\ast)\ \&\ (\exists u\ x=uy)\ \&\ (\exists u\ y=ux).
\]
The proof of ($\mathrm{iii}^\prime$) essentially goes as in \cite[Lemma 1.4]{gmm},
cf. \cite[Lemma 9]{HS}.
Indeed, the two-sided ideal $I=\kerr\varphi$ is regular.
Let $c\in A$ be such that $a=\varphi(c)$, and let $aba=a$ in $B$.
There is $y\in A$ such that $\varphi(y)=b$.
Then $c-cyc\in I$. Since $I$ is regular, there is $u\in I$ such that
$(c-cyc)u(c-cyc)=c-cyc$. It follows from the latter that
$cuc-cycuc-cucyc+cycucyc+cyc=c$. Taking $d=u-ucy-ycu+ycucy+y$, we get
$cdc=cuc-cucyc-cycuc+cycucyc+cyc=c$ and
$d-y=u-ucy-ycu+ycucy\in I$, whence $\varphi(d)=b$.
\end{proof}

\noindent
Further examples of strongly regular classes are the class of all
regular [complemented] members of any variety
having ring [bounded modular lattice, respectively] reducts, see \cite[Lemma 9]{HS}.
The latter can be easily modified to the class of all relatively complemented lattices;
here $\alpha(x_1,x_2,x_3,y)$ is given by
$y\bigl((x_1+x_2)x_3 +x_1x_2\bigr)=x_1x_2\ \&\ y+\bigl((x_1+x_2)x_3 +x_1x_2\bigr)=x_1+x_2$.

\bigskip
We consider fragments of the first order language associated with a given
regular class $\mathcal{C}_0$. Let $L_u$ consist of all
quantifier free formulas; up to equivalence,
we may assume that $L_u$ consists of conjunctions of formulas
$\bigwedge_{i=1}^n\beta_i\ \rightarrow\ \bigvee_{j=1}^m\gamma_j$,
where $\beta_i$, $\gamma_j$ are atomic formulas and $n$, $m\geqslant 0$.
The set $L_q\subseteq L_u$ of all \emph{quasi-identities} is defined by $m=1$.
The set $L_p$ consists of all formulas of the form
\[
\bigwedge_{i=1}^n\alpha_i(\overline{x}_i,y_i)
\ \rightarrow\ \bigvee_{j=1}^m\gamma_j,
\]
where $n\geqslant 0$, $m\geqslant 1$,
and $\alpha_i(\overline{x}_i,y_i)\in\Psi_0$.
Then $L_e\subseteq L_p$ is defined by $m=1$;
its members are called \emph{conditional identities},
while those of $L_p$ are \emph{conditional disjunctions of equations}.
As usual, validity of a formula means validity of its universal closure.
We write $\Th_x$ instead of $\Th_{L_x}$.

\begin{theorem}\lab{exvar}
Let $\mathcal{C}_0$ be a regular class and let $\mathcal{C}\subseteq\mathcal{C}_0$.
Then the following  equalities  and definabilities
 are granted.
\begin{enumerate}
\item
$\mathcal{C}_0\cap\Mod\Th_{u}\mathcal{C}={\sf S}_\exists{\sf P}_{\sf u}(\mathcal{C})$.
In particular, $\mathcal{C}$ is definable by universal sentences relatively to $\mathcal{C}_0$
if and only if it is closed under ${\sf S}_\exists$ and ${\sf P}_{\sf u}$.
\item
$\mathcal{C}_0\cap\Mod\Th_{q}\mathcal{C}={\sf S}_\exists{\sf P}_{\sf u}{\sf P}_\omega(\mathcal{C})=
{\sf S}_\exists{\sf P}{\sf P}_{\sf u}(\mathcal{C})$.
In particular, $\mathcal{C}$ is definable by quasi-identities relatively to $\mathcal{C}_0$
if and only if it is closed under ${\sf S}_\exists$, ${\sf P}_{\sf u}$, and ${\sf P}_\omega$
$[$under ${\sf S}_\exists$, ${\sf P}_{\sf u}$, and ${\sf P}$, respectively$]$.
\item
$\mathcal{C}_0\cap\Mod\Th_{p}\mathcal{C}={\sf H}{\sf S}_\exists{\sf P}_{\sf u}(\mathcal{C})$.
In particular, $\mathcal{C}$ is definable by conditional disjunctions of equations relatively to $\mathcal{C}_0$
if and only if it is closed under ${\sf H}$, ${\sf S}_\exists$, and ${\sf P}_{\sf u}$.
\item
$\mathcal{C}_0\cap\Mod\Th_{e}\mathcal{C}={\sf H}{\sf S}_\exists{\sf P}_{\sf u}{\sf P}_\omega(\mathcal{C})=
{\sf H}{\sf S}_\exists{\sf P}{\sf P}_{\sf u}(\mathcal{C})$.
In particular, $\mathcal{C}$ is definable by conditional identities relatively to $\mathcal{C}_0$
if and only if it is closed under ${\sf H}$, ${\sf S}_\exists$, ${\sf P}_{\sf u}$, and ${\sf P}_\omega$
$[$under ${\sf H}$, ${\sf S}_\exists$, ${\sf P}$, and ${\sf P}_{\sf u}$, respectively$]$.
\end{enumerate}
\end{theorem}

\noindent
Classes as in (iii) and (iv) will be called $\exists$-\emph{semivarieties}
and $\exists$-\emph{varieties}, respectively.
If $\Psi_0$ is empty, one has \emph{semivarieties} and \emph{varieties}.
By ${\sf W}_\exists(\mathcal{C})$ $[$by ${\sf V}_\exists(\mathcal{C})$,
${\sf W}(\mathcal{C})$, ${\sf V}(\mathcal{C})$, respectively$]$,
we denote the smallest $\exists$-semivariety $[\exists$-variety, semivariety, variety, respectively$]$
containing $\mathcal{C}$, cf. Theorem \ref{exvar}(iii)-(iv).

Of course, the statements of Theorem \ref{exvar} are well known results in the case of empty $\Psi_0$.
Proofs of (i) and (ii) are included since they can be seen as a preparation for proofs of (iii)-(iv);
the latter are our primary interest.

\begin{proof}
Inclusion in the model class is well known and easy to verify in any of the cases (i)-(iv) using Definition \ref{RegC}.
In particular in cases (iii)-(iv), inclusion ${\sf H}(\mathcal{C})\subseteq\Mod\Th_{x}\mathcal{C}$
follows directly from Definition \ref{RegC}(iii).

The proof of the reverse inclusion relies on adapting the
method of diagrams. Given a structure $A$, let $a\mapsto x_a$
be a bijection onto a set of variables and let
$\bar x=(x_a\mid a\in A)$. We consider quantifier free
formulas $\chi(\bar x)$ in these variables; evaluations
$\bar x$ in a structure $B$ are given as $\bar b=(b_a\mid a\in A)\in B^A$,
and we write $B\models\chi(\bar b)$ if $\chi(\bar x)$
is valid under evaluation $\bar b$.
For a set $\Phi=\Phi(\bar x)$ of formulas, $B\models\Phi(\bar b)$ if
$B\models\chi(\bar b)$ for all $\chi(\bar x)\in\Phi$.
Let $At$ denote the set of atomic formulas and let
\begin{align*}
\Delta^+(A)=&\{\chi(\bar x)\in At\mid A\models\chi(\bar a)\};\\
\Delta^-(A)=&\{\neg\chi(\bar x)\mid\chi(\bar x)\in At,\ A\not\models\chi(\bar a)\};\\
\Delta^0(A)=&\Bigl\{\alpha\bigl(t_1(\bar x),\ldots,
t_n(\bar x),x_a\bigr)\mid\\
&t_1,\ldots,t_n\ \text{are terms},\ \alpha(x_1,\ldots,x_n,y)\in\Psi_0,\ 
A\models\alpha\bigl(t_1(\bar a),\ldots,t_n(\bar a),a\bigr)\Bigr\};
\end{align*}
\begin{align*}
&\Delta_u(A)=\Delta_q(A)=\Delta^+(A)\cup\Delta^-(A);\\
&\Delta_p(A)=\Delta_e(A)=\Delta^0(A)\cup\Delta^-(A).
\end{align*}
For $x\in\{u,q,p,e\}$ and a finite subset $\Phi$ of $\Delta_x(A)$, let
$\Phi^-=\Phi\cap\Delta^-(A)$, $\Phi^+=\Phi{\setminus}\Phi^-$, and let $\Phi^\dagger$ denote the formula
\[
\bigwedge_{\phi\in\Phi^+}{\phi}\ \rightarrow\ \bigvee_{\neg\chi\in\Phi^-}{\chi};
\]
while for $\neg\chi\in\Phi^-$, let $\Phi^\dagger_\chi$ denote the quasi-identity
\[
\bigwedge_{\phi\in\Phi^+}{\phi}\ \rightarrow\ \chi.
\]
Thus for any finite $\Phi\subseteq\Delta_u(A)$ and for $\chi\in\Phi^-$,
we have $\Phi^\dagger\in L_u$ and $\Phi^\dagger_\chi\in L_q$,
while for any finite $\Phi\subseteq\Delta_p(A)$ and for $\chi\in\Phi^-$,
we have $\Phi^\dagger\in L_p$ and $\Phi^\dagger_\chi\in L_e$.
Observe that $A\not\models\Phi^\dagger$ and $A\not\models\Phi^\dagger_\chi$ in any case
(verified by substituting $x_a$ with $a$).
Let $A\in\mathcal{C}_0\cap\Mod\Th_x\mathcal{C}$.
We have to obtain $A$ from $\mathcal{C}$ by means of operators.

First, we consider the case $x\in\{u,p\}$.
Let $\Phi\subseteq\Delta_x(A)$ be finite.
As $A\not\models\Phi^\dagger$, we have that $\Phi^\dagger\notin\Th_x\mathcal{C}$.
Thus there are a structure $B_\Phi\in\mathcal{C}$ and $\bar b_{\Phi}=(b_{\Phi a}\mid a\in A)\in B_\Phi^A$
such that $B_\Phi\not\models\Phi^\dagger(\bar b_\Phi)$, i.e. $B_\Phi\models\Phi(\bar b_{\Phi})$.

As in the proof of the Compactness Theorem,
let $I$ be the set of all finite subsets of
$\Delta_x(A)$ and let $\mathcal{U}$ be one of he the  ultrafilters containing all
sets $\{\Psi\in I\mid\Psi\supseteq\Phi\}$, where $\Phi\in I$.
Let $B=\prod_{\Phi\in I}B_\Phi\slash\mathcal{U}$,
$b_a=(b_{\Phi a}\mid\Phi\in I)\slash\mathcal{U}$ and $\bar b=(b_a\mid a\in A)$.
By (the quantifier free part of) the {\L}o\'{s} Theorem, we have
$B\models\Delta_x(A)(\bar b)$.
Moreover, $B\in{\sf P}_{\sf u}(\mathcal{C})\subseteq\mathcal{C}_0$.

Let $C$ be the subalgebra of $B$ generated by the set $\{b_a\mid a\in A\}$.
We claim that $C\in\mathcal{C}_0$, i.e. $C\in{\sf S}_\exists(B)$.
Indeed, let $\alpha(x_1,\ldots,x_n,y)\in\Psi_0$ and let $c_1$, \ldots, $c_n\in C$.
As $C$ is generated by the set $\{b_a\mid a\in A\}$, there are terms $t_1(\bar x)$, \ldots, $t_n(\bar x)$
such that $c_i=t_i(\bar b)$ for all $i\in\{1,\ldots,n\}$.
Since $A\in\mathcal{C}_0$, by Definition \ref{RegC}(i) there is $a\in A$ such that
\[
A\models\alpha\bigl(t_1(\bar a),\ldots,t_n(\bar a),a\bigr).
\]
Therefore,
\[
\alpha\bigl(t_1(\bar x),\ldots,t_n(\bar x),x_a\bigr)\in\Delta^+(A)\cap\Delta^0(A).
\]
Since $B \models\Delta_x(A)(\bar b)$, we conclude that
$B\models\alpha\bigl(t_1(\bar b),\ldots,t_n(\bar b),b_a\bigr)$.
This implies that $C\models\alpha(c_1,\ldots,c_n,b_a)$.
On the other hand, $B\in{\sf P}_{\sf u}(\mathcal{C})\subseteq\mathcal{C}_0\subseteq\mathcal{S}$,
as $\mathcal{C}_0$ is closed under ${\sf P}_{\sf u}$ by Definition \ref{RegC}(ii).
Therefore, $C\in{\sf S}(B)\subseteq{\sf S}(\mathcal{S})\subseteq\mathcal{S}$ again by Definition \ref{RegC}(ii).
This implies by Definition \ref{RegC}(i) that $C\in\mathcal{C}_0$ which is our desired conclusion.
Furthermore, the map
\[
\varphi\colon C\to A;\quad t(\bar b )\mapsto t(\bar a)
\]
is well-defined (since $B\models \Delta^-(A)(\bar b)$),
a homomorphism (in view of term composition),
and surjective (since $\varphi(b_a)=a$).
Moreover, in case $x=u$, $\varphi$ is an isomorphism,
as $B\models\Delta^+(A)(\bar b)$. This proves (i) and (iii).

Let $x\in\{q,e\}$.
Given a finite subset $\Phi\subseteq\Delta_x(A)$ and $\neg\chi\in\Phi^-$,
one has $A\not\models\Phi^\dagger_\chi$, whence
$\Phi^\dagger_\chi\notin\Th_x\mathcal{C}$.
Thus there are a structure $B_{\Phi,\chi}\in\mathcal{C}$
and $\bar b_{\Phi\chi}=(b_{\Phi\chi a}\mid a\in A)\in B_{\Phi,\chi}^A$ such that
\[
B_{\Phi,\chi}\models\Phi^+(\bar b_{\Phi\chi})\quad\text{and}\quad
B_{\Phi,\chi}\models\neg\chi(\bar b_{\Phi\chi}).
\]
Taking $B_\Phi=\prod_{\neg\chi\in\Phi^-}B_{\Phi,\chi}\in{\sf P}_\omega(\mathcal{C})$
and $b_{\Phi a}=(b_{\Phi\chi a}\mid\neg\chi\in\Phi^-)$,
we get that $B_\Phi\models\Phi(\bar b_{\Phi})$.
As above, let $B=\prod_{\Phi \in I}B_\Phi\slash\mathcal{U}$,
$b_a=(b_{\Phi a}\mid\Phi\in I)\slash\mathcal{U}$, so that $B\models\Delta_x(A)(\bar b)$.
Let $C$ be again the subalgebra of $B$ generated by the set $\{b_a\mid a\in A\}$.
We get as above that $C\in{\sf S}_\exists{\sf P}_{\sf u}{\sf P}_\omega(\mathcal{C})$.
Thus $A\in{\sf H}(C)$ for $x=e$
and $A\cong C$ for $x=q$ follow exactly as above.

It remains to show that $A\in{\sf H}{\sf S}_\exists{\sf P}{\sf P}_{\sf u}(\mathcal{C})$ if $x=e$ and
$A\in{\sf S}_\exists{\sf P}{\sf P}_{\sf u}(\mathcal{C})$ if $x=q$.
Here, we fix $\neg\chi\in\Delta(A)^-$ and consider the set
$I_\chi=\{\Phi\in I\mid\neg\chi\in\Phi^-\}$. Then there is a non-principal ultrafilter
$\mathcal{U}_\chi$ on $I$ which contains all sets $\{\Psi\in I_\chi\mid\Psi\supseteq\Phi\}$ with $\Phi\in I_\chi$.
Take
\[
B_\chi=\prod_{\Phi\in I_\chi}B_{\Phi,\chi}\slash\mathcal{U}_\chi;\quad
b_{\chi a}=(b_{\Phi\chi a}\mid\Phi\in I_\chi)\slash\mathcal{U}_\chi;\quad
\bar b_\chi=(b_{\chi a}\mid a\in A),
\]
so that $B_\chi\models\neg\chi(\bar b_\chi)$ and $B_\chi\models\Delta^+(A)(\bar b_\chi)$
if $x=q$, $B_\chi\models\Delta^0(A)(\bar b_\chi)$ if $x=e$. Then
\[
B^\prime=\prod_{\neg\chi\in\Delta^-(A)}B_\chi\in{\sf P}{\sf P}_{\sf u}(\mathcal{C});\quad
B^\prime\models\Delta_x(A)(\bar b^\prime),\ \text{where}\ b^\prime_a=(b_{\chi a}\mid\chi\in\Delta^-(A)).
\]
Let $C^\prime$ be the subalgebra of $B^\prime$ generated by the set $\{b^\prime_a\mid a\in A\}$.
As above, $C^\prime\in{\sf S}_\exists(B^\prime)$ and
$A\in{\sf H}(C^\prime)$ (if $x=e$) or
$A\cong C^\prime$ (if $x=q$) via the map $\varphi^\prime\bigl(t(\bar b^\prime)\bigr)=t(\bar a)$.
The proof is now complete.
\end{proof}

\noindent
The following recaptures \cite[Proposition 10]{HS}.
For convenience, we include proofs.

\begin{proposition}\lab{hs1b}
Let $\mathcal{C}_0$ be a strongly regular class and let $\mathcal{C}\subseteq\mathcal{C}_0$.
\begin{enumerate}
\item ${\sf S}_\exists{\sf H}(\mathcal{C})\subseteq{\sf H}{\sf S}_\exists(\mathcal{C})$;
\item
${\sf V}_\exists(\mathcal{C})={\sf H}{\sf S}_\exists{\sf P}(\mathcal{C})$;
\item
If all members of $\mathcal{C}_0$ have a distributive congruence lattice,
then $A\in{\sf W}_\exists(\mathcal{C})$ for any subdirectly irreducible structure $A\in{\sf V}_\exists(\mathcal{C})$.
\end{enumerate}
\end{proposition}

\begin{proof}
(i) Let structures $A$, $B$ and $C$ be such that $A\in\mathcal{C}$, $C\in{\sf S}_\exists(B)$, and let
$\varphi\colon A\to B$ be a surjective homomorphism.
Then $B$, $C\in\mathcal{C}_0$ by Definition \ref{RegC}(ii).
Choose a Skolem expansion $C^\ast$ of $C$ and extend it to a Skolem expansion $B^\ast$ of $B$.
According to Remark \ref{SRC}, there is a Skolem expansion $A^\ast$ of $A$ such that
$\varphi\colon A^\ast\rightarrow B^\ast$ is a homomorphism. Then
$C^\ast\in{\sf S}(B^\ast)\subseteq{\sf S}{\sf H}(A^\ast)\subseteq{\sf H}{\sf S}(A^\ast)$,
whence $C^\ast\in{\sf H}(D^\ast)$ for some $D^\ast\in{\sf S}(A^\ast)$ and $C\in{\sf H}(D)$
with $D\in{\sf S}_\exists(A)$.

(ii) According to Theorem \ref{exvar}(iv),
${\sf V}_\exists(\mathcal{C})={\sf H}{\sf S}_\exists{\sf P}{\sf P}_{\sf u}(\mathcal{C})$.
Straightforward inclusions ${\sf P}_{\sf u}(\mathcal{C})\subseteq{\sf H}{\sf P}(\mathcal{C})$
and ${\sf P}{\sf H}(\mathcal{C})\subseteq{\sf H}{\sf P}(\mathcal{C})$ together with (i) imply:
\[
{\sf V}_\exists(\mathcal{C})\subseteq{\sf H}{\sf S}_\exists{\sf P}{\sf H}{\sf P}(\mathcal{C})\subseteq
{\sf H}{\sf S}_\exists{\sf H}{\sf P}(\mathcal{C})\subseteq{\sf H}{\sf S}_\exists{\sf P}(\mathcal{C}).
\]
The reverse inclusion is obvious.
 
(iii) Let $A\in{\sf V}_\exists(\mathcal{C})$ be subdirectly irreducible.
Then by (ii), there is $B\in{\sf S}_\exists{\sf P}(\mathcal{C})$ such that $A\in{\sf H}(B)$.
By J\'{o}nsson's Lemma, there is $C\in{\sf S}{\sf P}_{\sf u}(\mathcal{C})$ such that
$A\in{\sf H}(C)$ and $C\in{\sf H}(B)$.
The latter inclusion implies by Definition \ref{RegC}(ii) that $C\in\mathcal{C}_0$,
whence $C\in{\sf S}_\exists{\sf P}_{\sf u}(\mathcal{C})$.
\end{proof}

\end{appendix}

\end{document}